\documentclass[12pt]{amsart}
\usepackage{anysize}
\marginsize{2.7cm}{2.44cm}{2.0cm}{3.0cm}
\usepackage[small,it]{caption}
\usepackage[OT2,T1]{fontenc}
\DeclareSymbolFont{cyrletters}{OT2}{wncyr}{m}{n}
\DeclareMathSymbol{\Sha}{\mathalpha}{cyrletters}{"58}
 \usepackage[latin1]{inputenc}
 \usepackage[dvips]{graphicx}
 \usepackage{wrapfig}
 \usepackage{amsmath}
 \usepackage{amsthm}
 \usepackage{amsfonts}
 \usepackage{amssymb}
 \usepackage{layout} \usepackage{verbatim}
 \usepackage{alltt}
\usepackage{yfonts}

\usepackage[all]{xy}
\usepackage{setspace}

\newtheorem*{thma}{Theorem A}
\newtheorem*{thmb}{Theorem B}
\newtheorem*{thmc}{Theorem C}

\newcommand{\LL}{\Lambda}
\newcommand{\TT}{\mathbb{T}}
\newcommand{\QQ}{\mathbb{Q}}
\newcommand{\FF}{\mathcal{F}}

\newcommand{\lra}{\longrightarrow}
\newcommand{\ZZ}{\mathbb{Z}}
\newcommand{\PP}{\mathcal{P}}

\newcommand{\KKK}{\mathcal{K}}

\newcommand{\Gal}{\textup{Gal}}
\newcommand{\KS}{\textbf{\textup{KS}}}
\newcommand{\ES}{\textbf{\textup{ES}}}
\newcommand{\NN}{\mathcal{N}}
\newcommand{\ra}{\rightarrow}

\newcommand{\be}{\begin{equation}}
\newcommand{\ee}{\end{equation}}

\newcommand{\al}{\mathcal{L}}

\newcommand{\oo}{\mathcal{O}}

\newcommand{\FFc}{\mathcal{F}_{\textup{\lowercase{can}}}}

\newcommand{\ooo}{\frak{O}}
\newcommand{\cyc}{\textup{cyc}}

\newcommand{\locu}{\textup{loc}_{/V}}
\newcommand{\charr}{\textup{char}}

\numberwithin{equation}{section}
\newtheorem{thm}{Theorem}[section]
\newtheorem{lemma}[thm]{Lemma}
\newenvironment{define}{\par\medskip\noindent\refstepcounter{thm}
\bgroup{\hspace*{-0.15 cm}\bf{Definition}
\thethm.}\bgroup}{\egroup \egroup\par\medskip}
\newtheorem{prop}[thm]{Proposition}
\newtheorem{cor}[thm]{Corollary}
\newtheorem{conj}[thm]{Conjecture}
\newenvironment{rem}{\par\medskip\noindent\refstepcounter{thm}
\bgroup{\hspace*{-0.15 cm}\bf{Remark} \thethm.}\bgroup}{\egroup
\egroup\par\medskip} \parskip 2pt

\newenvironment{example}{\par\medskip\noindent\refstepcounter{thm}
\bgroup{\hspace*{-0.15 cm}\bf{Example}
\thethm.}\bgroup}{\egroup \egroup\par\medskip}

\newcounter{Athm}[section]

\renewcommand{\theAthm} {\arabic{Athm}}

\long\def\symbolfootnote[#1]#2{\begingroup%
\def\thefootnote{\fnsymbol{footnote}}\footnote[#1]{#2}\endgroup}

\begin{document}
\title{O\lowercase{n the }I\lowercase{wasawa theory of} CM \lowercase{fields for supersingular primes}}

\author{K\^az\i m B\"uy\"ukboduk}

\email{kbuyukboduk@ku.edu.tr}
\address{K\^az\i m B\"uy\"ukboduk \hfill\break\indent Ko\c{c} University Mathematics, 34450 Sariyer, Istanbul, Turkey
\hfill\break\indent}

\keywords{Euler systems, Kolyvagin systems, Deformations of Kolyvagin systems, Equivariant Tamagawa Number Conjecture}
\keywords{Iwasawa Theory, Iwasawa's main conjectures, Rubin-Stark elements}
\subjclass[2000]{11G05; 11G07; 11G40; 11R23; 14G10}

\begin{abstract}
The goal of this article is two-fold: First, to prove a (two-variable) main conjecture for a CM field $F$ without assuming the $p$-ordinary hypothesis of Katz, making use of what we call the Rubin-Stark $\al$-restricted Kolyvagin systems which is constructed out of the conjectural Rubin-Stark Euler system of rank $g$. (We are also able to obtain weaker unconditional results in this direction.) Second objective is to prove the Park-Shahabi plus/minus main conjecture for a CM elliptic curve $E$ defined over a general totally real field again using (a twist of the) Rubin-Stark Kolyvagin system. This latter result has consequences towards the Birch and Swinnerton-Dyer conjecture for $E$. 
\end{abstract}

\maketitle
\tableofcontents
\section{Introduction}
\label{sec:intro}
Let $F$ be a CM field and suppose $[F:\QQ]=2g$. In the particular case when $F$ is an imaginary quadratic field, the main conjectures of Iwasawa theory over $F$ has been settled in \cite{rubinmainconj} using elliptic units. For a general CM field $F$, all major work related to Iwasawa's main conjecture utilized congruences of modular forms (and have relied on the CM-form method in \cite{ht93,ht94} or the Eisenstein ideal technique in \cite{mainardi,hsiehCMmainconj}) as the main tool. That approach required that the following $p$-ordinary condition (\ref{eqn:pord}) of Katz holds true. Fix an embedding $\iota_p: \overline{\QQ}\hookrightarrow \overline{\QQ}_p$.
 \begin{align}\label{eqn:pord}& \textup{ There exists a CM-type } \Sigma \textup{ such that the embeddings } \Sigma_p:=\{\iota_p\circ \sigma\}_{\sigma \in \Sigma}  \\
 \notag &\textup{ induce exactly half of the places of } F \textup{ over } p.\end{align}
Let $\widetilde{F}_\infty$ denote the maximal $\ZZ_p$-power extension of $F$ and set $\widetilde{\Gamma}=\textup{Gal}(\widetilde{F}_\infty)/F$. Let $\widetilde{\LL}=\mathcal{W}[[\widetilde{\Gamma}]]$, where $\mathcal{W}$ is the valuation ring of $\widehat{\overline{\QQ}}_p$. Assuming (\ref{eqn:pord}), the relevant Iwasawa module $\widetilde{\frak{X}}_\Sigma$ is $\widetilde{\LL}$-torsion and Katz in~\cite{katz78} has constructed a $p$-adic $L$-function $\al_\Sigma \in \widetilde{\LL}$. In this case, Hsieh in \cite{hsiehCMmainconj} proved that the characteristic ideal of $\widetilde{\frak{X}}_\Sigma$ is generated by $\al_\Sigma$ under suitable hypothesis, thereby proving the Iwasawa main conjecture for $F$. The author has also obtained results along these lines in~\cite{kbbCMabvar} using the conjectural Rubin-Stark elements. The approach in loc.cit. is based on a refinement of the rank-$g$ Euler/Kolyvagin system machinery and  relies crucially on the assumption (\ref{eqn:pord}) for an analysis of the local cohomology groups above $p$.

All these techniques towards the proof of main conjectures for a general CM field $F$ alluded to above fall apart when the $p$-ordinary condition (\ref{eqn:pord}) fails. One difficulty is that in the absence of (\ref{eqn:pord}), neither the relevant Iwasawa module is $\widetilde{\LL}$-torsion nor we have a $p$-adic $L$-function available in this set up (in any case, it is not expected to belong to $\widetilde{\LL}$). Beyond the case $g=1$,  nothing substantial along these lines was known; when $g=1$ Rubin has proved the two-variable main conjecture in \cite{rubinmainconj}. Furthermore (still when $g=1$), if $A/\QQ$ is an elliptic curve that has CM by the ring of integers $\oo_F$ of $F$, Kobayashi in~\cite{kobayashi03} formulated a pair of conjectures which are both equivalent to the {cyclotomic} main conjectures of Perrin-Riou and Kato \cite{pr93grenoble,ka1} for $A$. Pollack and Rubin in~\cite{pollackrubin} proved Kobayashi's conjectures using Rubin's proof of the two-variable main conjecture in~\cite{rubinmainconj} and incorporating Kobayashi's theory of plus/minus Selmer groups with the elliptic unit Euler system.

The goal of this article is to appropriately modify and extend the methods of \cite{kbbCMabvar} so as to prove (conditional on some standard conjectures)
\begin{itemize}
\item a two-variable main conjecture for a general CM field $F$ in the absence of the hypothesis \eqref{eqn:pord} using the (conjectural) Rubin-Stark elements (this is Theorem A below);
\item a divisibility in the (one-variable) \emph{cyclotomic} main conjecture for a $p$-supersingular CM elliptic curve defined over a general totally real field (this is Theorem B below);
\item prove that the divisibility in the previous item may be upgraded to an equality using the structure of the module of $\LL$-adic Kolyvagin systems, as described in \cite{kbbdeform} (we provide a detailed account of this in Section~\ref{Sec:KSforGmandE} below).\end{itemize}

\subsection*{Notation}
Before we explain our results in greater detail, we set some notation. Let $E$ be an elliptic curve defined over a totally real field $F^+$, which has CM by an order $\oo$ of an imaginary quadratic field $K$. Let $g:=[F^+:K]$ and let $F=F^+K$ be the composite CM field. Fix once and for all an odd prime $p$ that is coprime to the index $[\oo_K:\oo]$  of $\oo$ inside the maximal order $\oo_K$ and which is inert in $K/\QQ$. We denote the unique prime of $K$ above $p$ also by $p$ and we denote the completion $K_p$ by $\Phi$. We let $\ooo$ denote the ring of integers of $\Phi$.

 Let $K_\infty$ denote the unique $\ZZ_p^2$-extension of $K$ and $K^{\textup{cyc}}$ the cyclotomic $\ZZ_p$-extension. Let $F_\infty=FK_\infty$ and $F^{\textup{cyc}}=FK^{\textup{cyc}}$. Let $\Gamma=\textup{Gal}(F_\infty/F)$ and $\Gamma_{\textup{cyc}}=\textup{Gal}(F^\textup{cyc}/F)$. We define the two-variable (resp., one-variable) Iwasawa algebra $\LL:=\ooo[[\Gamma]]$ (resp., $\LL_{\textup{cyc}}:=\ooo[[\Gamma_\textup{cyc}]]$). For a Dirichlet character $\chi:\textup{Gal}(\overline{F}/F)\ra \ooo^\times$, let $L$ denote the extension of $F$ cut by $\chi$ and let $\frak{U}$ denote the inverse limit of the $\chi$-isotypic part of the local units up the tower of finite extensions contained in $LF_\infty/L$. Let $\mathcal{Q}$ denote a certain quotient of $\frak{U}$ (see Definition~\ref{define:calucyc}) and let $\LL\cdot\locu(\varepsilon_{F_\infty}^\chi)$ denote the submodule of $\wedge^g \,\mathcal{Q}$ generated by the image of the tower of Rubin-Stark elements (defined as in Definition~\ref{def:rubinstarktowerlocal}). Let $\widehat{X}_\infty$ be a certain Iwasawa module (denoted by $H^1_{\FF_{\frak{tr}}^*}(F,\TT^*)^\vee$ in the main text, which is given as in Definition~\ref{def:selmerstructure}). \emph{Assume the truth of Rubin-Stark conjectures and Leopoldt's conjecture for Theorems A, B and C below.} See Remarks~\ref{rem:removeRSconjhypo} and \ref{rem:unconditional} below for the portion of the results in this article that we are able to prove unconditionally.
\subsection*{Statements of the results}
The first main result in this article is the (two-variable\footnote{Since our sights are mainly set on the proof of the cyclotomic main conjecture for CM elliptic curves over $F^+$ (that is, Theorem B below), we contend ourselves to prove only a two-variable supersingular main conjecture over $F$. However, the methods of this article seem flexible enough to treat the more general case and prove a more general main conjecture (e.g., over the maximal $\ZZ_p$-power extension $\widetilde{F}_{\infty}/F$).}) Iwasawa main conjecture for $F_\infty/F$. Let $\charr(M)$ denote the characteristic ideal of a finitely generated torsion $\LL$-module.

\begin{thma}[See Theorem~\ref{thm:mainconjforTchi} and \ref{thm:mainconjdivisibility}] The $\LL$-module $\widehat{X}_\infty$ is torsion and 
$\textup{char}(\widehat{X}_\infty)$ divides $\textup{char} \left(\wedge^g \,\mathcal{Q}/\LL\cdot  \locu\left(\varepsilon_{F_\infty}^\chi \right)\right)$. These two ideals are equal if we further assume a strong version of the Rubin-Stark Conjecture (Conjecture~\ref{conj:strongRS} below).
\end{thma}
This statement was proved by Rubin~\cite[\S11]{rubinmainconj} when $F^+=\QQ$, using elliptic elliptic units. To obtain the generalization above we make use of the Rubin-Stark elements. To do so, the CM \emph{rank-$g$ Euler/Kolyvagin system machinery} developed by the author in \cite{kbbCMabvar} (relying crucially on the $p$-ordinary hypothesis (\ref{eqn:pord})) requires a non-trivial refinement. 
This is one of the major tasks we carry out in this article.

For the rest of our results, we assume that the prime $p$ splits completely in $F^+/\QQ$. This assumption could be removed (but allowing also only weaker results); see Remark~\ref{rem:lei} below. Thanks to this assumption we may adopt the (local) methods of Kobayashi \cite{kobayashi03} and define the signed Selmer groups $\textup{Sel}^\pm_p(E/F^{\textup{cyc}})$. In this situation, we are lead to formulate a (conjectural) explicit reciprocity law for the Rubin-Stark elements; see Conjecture~\ref{conj:reciprocity}. This conjecture on one hand proposes a natural extension of the Coates-Wiles explicit reciprocity law and on the other, it furnishes us with a link between the tower of Rubin-Stark elements and the Park-Shahabi signed $p$-adic $L$-functions $L_p^\pm(E/F^+)$. 

Theorem B gives a proof of the cyclotomic main conjecture for $E$ for a supersingular prime $p$ under our running assumptions. 

 \begin{thmb}[Theorem~\ref{thm:improvedmainconj}]
Assuming the Explicit Reciprocity Conjecture~\ref{conj:reciprocity} for Rubin-Stark elements, the divisibility 
$$\textup{char}\left(\textup{Sel}^\pm(E/F^{\textup{cyc}})^\vee\right)\mid L_p^\pm(E/F^+)\,\LL_{\textup{cyc}}$$
in the signed main conjecture holds true, with equality if we assumed a strong version of the Rubin-Stark Conjecture (Conjecture~\ref{conj:strongRS} below).
 \end{thmb}
 
 We remark that we do not descent from the two-variable main conjecture in order to deduce Theorem B (as done so in \cite{pollackrubin}) but instead, we rely further on the author's results on the structure of the module of $\LL$-adic Kolyvagin systems. This alternative approach has the advantage that we need not worry about pseudo-null submodules of various Iwasawa modules.

 \begin{rem}
 \label{rem:removeRSconjhypo}
 Although the existence of the Rubin-Stark elements is highly conjectural, one may prove (Theorem~\ref{thm:lamdaadicKS} below) that the Kolyvagin systems that they descend do exist \emph{unconditionally}. Notice also that the Kolyvagin systems which descend from the conjectural Rubin-Stark elements are non-trivial, since we assumed Leopoldt's conjecture (c.f., Proposition~\ref{prop:KSwithnonzeroinitialterm}). One may work with these Kolyvagin systems for the most part to prove statements which lead to Theorem A and B (Theorem~\ref{thm:mainapplicationofrestrictedKS} and Proposition~\ref{prop:usefulforequality}); however, the Reciprocity Conjecture~\ref{conj:reciprocity} that links the Kolyvagin systems we construct with the $L$-values could be stated most naturally in terms of the conjectural Rubin-Stark elements.
 \end{rem}
 
   Theorem B has the following important consequence towards the conjecture of Birch and Swinnerton-Dyer for the CM elliptic curve $E_{/F^+}$:
 \begin{thmc}[Theorem~\ref{thm:bsd} below]$\,$
 \begin{enumerate} 
\item If $L(E/F^+,1)\neq 0$ then $E(F^+)$ is finite. 
\item Assuming the strong form of the Rubin-Stark conjecture as well as that $L(E/F^+,1)=0$, then $\textup{Sel}_p(E/F^+)$ is infinite.
\end{enumerate}
 \end{thmc}
 \begin{rem}
 \label{rem:unconditional}
 It seems very plausible that the methods of this paper would allow us to deduce Theorems A, B and C above \emph{unconditionally} under the additional hypothesis that $F^+(E[p])/K$ is abelian. The idea goes roughly as follows (we hope to provide the details in a future note): Firstly, by the assumption that $F^+(E[p])/K$ is abelian, one may use elliptic units to construct classes in $H^1(F,T_p(E)\otimes\LL)$. We may use the main theorem of \cite{kbb} to lift these classes to a $\LL$-adic Kolyvagin system (for certain \emph{modified Selmer structures} which are defined in Section \ref{sec:selmerstructures} below) for the $G_F$-representation $T_p(E)$, so as to view these classes (obtained from elliptic units) as the initial terms of this $\LL$-adic Kolyvagin system. Using this $\LL$-adic Kolyvagin system (whose initial term is explicitly given in terms of elliptic units), one could deduce Theorems A, B and C unconditionally.
  \end{rem}
  
\begin{rem}
\label{rem:lei}
 The first version of this article was circulated among experts back in early 2013, and it later became the main motivation and the groundwork for our forthcoming joint work with Antonio Lei \cite{kbblei1}. In \emph{op. cit.}, we are able extend some of the results of this paper to treat a CM abelian variety of arbitrary dimension. This work in part relies on the techniques developed here, as well as a general theory of plus/minus Coleman maps we develop in \cite{kbblei2}. Although in \cite{kbblei1}, the authors are able to lift the hypotheses on Theorem~B and C that $p$ splits completely in $F^+/\QQ$, they are able to deduce only one of the signed main conjectures (whereas we prove both main conjectures simultaneously here). Note that we could have also formulated $2^g$ signed main conjectures (as opposed to a single plus/minus main conjecture) here as well, by assigning one of the ``plus'' or ``minus local conditions'' at each prime lying above $p$ (as opposed assigning the ``plus'' or  ``minus local condition'' everywhere above $p$ uniformly) and prove each of them. 
 
 One further advantage of the more explicit approach we take here (namely, through Kobayashi's interpretation of signed Coleman maps, which in turn rely on his explicit local elements) is that, it allows us to state our explicit reciprocity conjectures in a much more concrete form. We hope that this will allow us to verify the explicit reciprocity conjectures (and therefore deduce our main results here unconditionally) in the situation of Remark~\ref{rem:unconditional}, namely when $F^+(E[p])/K$ is abelian. For the time being, this does not seem tractable in the rather abstract set up of \cite{kbblei1}.
 
  \end{rem} 
 \subsection*{Overview of the methods and layout of the paper}
We briefly outline the basic technical ingredients that go into the proofs of the Theorems A, B and C. 

In order to prove the two-variable main conjecture (Theorem A) we use the Rubin-Stark element Euler system of rank $g$. This requires to refine the \emph{rank-$g$ Euler/Kolyvagin system machinery} in the supersingular setting where the assumption (\ref{eqn:pord}) is no longer valid. the first step is to introduce various \emph{modified Selmer structures} (Section~\ref{sec:selmerstructures}) that produces Selmer groups that compare well with their classical counterparts. We construct and study in Section~\ref{Sec:KSforGmandE} the Kolyvagin systems associated to these modified Selmer structures. We in fact do this first unconditionally, then in Section~\ref{sec:KSRS} using the Rubin-Stark elements (recalled briefly in Section~\ref{sec:RS}). 
These Kolyvagin systems are then used in Section~\ref{sec:gras} (along with the arguments of Section~\ref{subsec:compareselmer} to compare the modified Selmer groups (that we control by the Rubin-Stark Kolyvagin systems) to the classical Selmer groups) to prove the divisibility statement in Theorem A. We then show that this divisibility may be upgraded to an equality by exploiting our results in Section~\ref{Sec:KSforGmandE} on the structure of $\LL$-adic Kolyvagin systems.

To deduce Theorem B (the cyclotomic main conjecture for a CM elliptic curve $E$ for a supersingular prime $p$) we appeal to Kobayashi's local theory, with which directly apply the Kolyvagin system machinery developed in Section~\ref{Sec:KSforGmandE}. This is one of the main differences with the approach in \cite{pollackrubin}, which ultimately relies on various explicit calculations with elliptic units which are not at our disposal. We get around of this issue by systematically utilizing our results on the modules of Kolyvagin systems. Kobayashi's plus/minus Selmer groups (and the corresponding pair of $p$-adic $L$-functions of Park-Shahabi) are recalled in Sections~\ref{subsec:prelimonCMtheory} and \ref{subsec:pmpadicL}. The Explicit Reciprocity Conjecture we formulate in Section~\ref{subsec:reciprocityconj} relates the tower of Rubin-Stark elements (along $F_\infty$) to the special values of (twisted) $L$-functions attached to $E$ at $s=1$. This conjecture should be thought of as an extension of Coates-Wiles explicit reciprocity law~\cite{coateswiles77, wiles78reciprocity} for elliptic units and we believe that Conjecture~\ref{conj:reciprocity} should be of independent interest for future investigation. 

The proof of Theorem~C follows from Theorem~B easily. A key ingredient is a result of \cite{ng} on the psuedo-null submodules of a natural Iwasawa module. 


\subsection*{Acknowledgements}
We would like to thank Ming-Lun Hsieh and Antonio Lei for stimulating  conversations and Robert Pollack for his encouragement on this project. We also thank the referee for his/her constructive suggestions and for proposed corrections to improve the paper. We are particularly grateful to him/her for pointing out an inaccuracy in Section~\ref{subsec:pmpadicL}, concerning our use of the Park-Shahabi $p$-adic $L$-function.

This work was partially supported by a Marie Curie IRG grant EC-FP7 230668,  T{\"U}B{\.I}TAK grant 113F059 and  the Turkish Academy of Sciences. 

\subsection{Notation and Hypotheses}
\label{subsec:notation}
For any field $k$, let $\overline{k}$ denote a fixed separable closure of $k$ and let $G_k=\textup{Gal}(\overline{k}/k)$ denote its absolute Galois group.

Throughout we fix a rational odd prime $p$ and embeddings $\overline{\QQ}\hookrightarrow \mathbb{C}$ and $\overline{\QQ} \hookrightarrow\mathbb{C}_p$ where $\mathbb{C}_p$ is the $p$-adic completion of $\overline{\QQ}_p$. We normalize the valuation $\textup{val}_p$ and the absolute value $|\cdot|_p$ on $\mathbb{C}_p$ by assuming $\textup{val}_p(p)=1$ and $|p|_p=p^{-1}$. For any positive integer $n$, let $\pmb{\mu}_{n}$ denote the $n$th roots of unity and $\pmb{\mu}_{p^\infty}=\varinjlim \pmb{\mu}_{p^m}$.

Let $F$ be a CM field and let $F^+$ be its maximal real subfield as in the Introduction. 
Let $\chi:G_F \ra \ooo^\times$ be any Dirichlet character whose order is prime to $p$ and which has the property that
\be
\label{eqn:assnotrivxhi}
\chi(\wp)\neq 1 
\hbox{  \, for any prime   } \wp \hbox{ of } F \hbox{ above } p.
\ee
and that
\be
\label{eqn:chiisnotteich}
\chi \neq \omega,
\ee
where $\omega$ is the Teichm\"uller character giving the action of $G_F$ on $\pmb{\mu}_p$. Later in Section~\ref{subsec:cyclomain}, we will work with a particular character $\chi$ attached to a CM elliptic curve $E$. We let $L:=\overline{F}^{\ker\chi}$ denote the abelian extension of $F$ cut out by $\chi$.

Let $\mathcal{R}$ be the set of primes of $F$ that does not contain any prime above $p$ nor any prime at which $\chi$ is ramified. Define $\NN(\mathcal{R})$ to be the square free products of primes chosen from $\mathcal{R}$. For $\ell\in \mathcal{R}$, let $F(\ell)$ be the maximal $p$-extension inside the ray class field of $F$ modulo $\ell$ and for $\eta=\ell_1\cdots\ell_s \in \NN(\mathcal{R})$, set $F(\eta)=F(\ell_1)\cdots F(\ell_s)$. We write $L(\eta)=L\cdot F(\eta)$ for the composite field. We consider the following collections of finite abelian extensions of $F$ (resp., of $L$):
\begin{itemize}
\item[(i)] $\frak{T}=\{F(\eta): \eta \in \NN(\mathcal{R})\},$
\item[(ii)] $\frak{T}_0=\{L(\eta): \eta \in \NN(\mathcal{R})\},$
\item[(iii)] $\frak{E}=\{M\cdot F(\eta): \eta \in \NN(\mathcal{R}); M\subset F_\infty \hbox{ is a finite extension of } F\},$
\item[(iv)] $\frak{E}_0=\{M\cdot L(\eta): \eta \in \NN(\mathcal{R}); M\subset F_\infty \hbox{ is a finite extension of } F\},$
\end{itemize}
Let $\displaystyle{{K}_0(X)=\varinjlim_{N \in X_0}N}$ and $\displaystyle{{K}(X)=\varinjlim_{N \in X}N}$ for $X=\frak{T}$ or $\frak{E}$. 
We finally set $\frak{G(X)}=\textup{Gal}(\frak{X}/F)$ and write $\ooo[[\frak{G(X)}]]:=\varprojlim \ooo[\frak{G(X)}/U]$, where the inverse limit is over the open subgroups $U$ of $\frak{G(X)}$, for the completed group ring of $\frak{G(X)}$.   

For any non-archimedean prime $\lambda$ of $F$, fix a decomposition group $\mathcal{D}_{\lambda}$ and the inertia subgroup $\mathcal{I}_\lambda \subset \mathcal{D}_{\lambda}$. Let ${(-)}^\vee=\textup{Hom}(-,\QQ_p/\ZZ_p)$ denote Pontryagin duality functor. Observe that ${(-)}^\vee\otimes\ooo=\textup{Hom}(-,\Phi/\ooo)$. Bearing this relation in mind, we will write $X^\vee$ for $\textup{Hom}(X,\Phi/\ooo)$ \emph{when $X$ is an $\ooo$-module}. We let $X^*:=\textup{Hom}(X,\pmb{\mu}_{p^\infty})$ denote the Cartier dual of $X$.

Let $F_\infty$ and $F^{\textup{cyc}}$ be as above. Let $F_n$ denote the unique subextension of $F^\textup{cyc}/F$ which has degree $p^n$ and set $\Gamma_n=\textup{Gal}(F_n/F)$.

We let $G_F$ act on $\LL$ (resp., $\LL_{\textup{cyc}}$) via the tautological surjection $G_F\ra \Gamma$ (resp., $G_F\ra \Gamma_{\textup{cyc}}$). For an $\ooo$-module $X$ of finite type which is endowed with a continuous action of $G_F$, we let $G_F$ act on the $\LL$-module $X\otimes_{\ooo}\LL$ by acting on both factors.

\section{Selmer structures and comparing Selmer groups}
\label{sec:selmerstructures}
\subsection{Structure of the semi-local cohomology groups}
\label{subsec:localanalysis}
Let $M=M_0\cdot F(\eta)$ be a member of the collection $\frak{E}$, where $M_0$ is a finite subextension of $F_\infty/F$. Set $\Delta_M=\textup{Gal}(M/F)$, $\delta_M=|\Delta_M|$ and $\LL_M=\frak{O}[\Delta_M]$.

Let $X$ be any $\frak{O}[[G_F]]$-module which is free of rank $d$ as an $\ooo$-module. Suppose in addition that $X$ satisfies the following hypothesis:
\begin{enumerate}
\item[\textbf{(H.p1)}] $H^2(F_\wp,X)=0=H^2\left(F_\wp,\textup{Hom}_{\ooo}(X,\ooo(1))\right)$, for any prime $\wp$ of $F$ above $p$.
\end{enumerate}

\begin{lemma}
\label{lem:extendp1toM}
Suppose $X$ is above. Let $M \in \frak{E}$ be an extension of $F$ and let $\frak{P}$ be a prime of $M$ lying above $p$. Then
 $$H^2(M_\frak{P},X)=0=H^2\left(M_\frak{P},\textup{Hom}_{\ooo}(X,\ooo(1))\right).$$
\end{lemma}

\begin{proof}
Let $\wp$ be the prime of $F$ lying below $\frak{P}$ and set $D_\frak{P}=\textup{Gal}(M_\frak{P}/F_\wp)$. Then either $D_\frak{P}$ is trivial and in this case  Lemma follows from \textup{\textbf{(H.p1)}}, or otherwise $D_\frak{P}$ is a non-trivial $p$-group. Then,
$$\# H^0(M_{\frak{P}},X^*[\varpi])= \# H^0\left(D_{\frak{P}}, (H^0(M_{\frak{P}},X^*[\varpi])\right)\equiv \#H^0(F_\wp,X^*[\varpi]) \equiv 1 \mod p$$
where the last equality holds thanks to \textup{\textbf{(H.p1)}} and local duality. This shows that $ H^0(M_{\frak{P}},X^*)=0$ and thus by local duality that $H^2(M_{\frak{P}},X)=0$, as desired. The second assertion is proved in an identical manner.
\end{proof}
\begin{define}
\label{def:pmsemilocal}
For $j=0,1,2$ define the semi-local cohomology groups
$$H^j(M_{p},X):=\bigoplus_{i=1}^s \bigoplus_{\frak{q}|p}H^j(M_{\frak{q}},X),$$
and let
$$\textup{loc}_p: H^1(M,X)\lra H^1(M_{p},X)$$
denote the localization map.
\end{define}

\begin{prop}
\label{prop:localfull} Suppose \textup{\textbf{(H.p1)}} holds true.
\begin{itemize}
\item[(i)]  The corestriction map
$$\textup{cor}: H^1(M_p,X)\lra H^1(F_p,X)$$
is surjective.
\item[(ii)] the $\LL_M$-module $H^1(M_{p},X)$ is free of rank $2g\cdot d$.
\item[(iii)] The $\LL$-module $H^1(F_p,X\otimes\LL)$ is free of rank $2g\cdot d$.
\item[(iv)] The $\frak{O}[[\frak{G}(\frak{X})]]$-module $\displaystyle{\varprojlim_{M\in \frak{E}}H^1(M_{p},X)}$ is free of rank $2g\cdot d$, where the inverse limits are with respect to corestriction maps.
\end{itemize}
\end{prop}

\begin{proof}
(iii) and (iv) follows at once from (i) and (ii). Both (i) and (ii) are essentially proved in \cite[\S2.1]{kbbCMabvar}.
\end{proof}
\begin{rem}
Observe that for $T=\frak{O}\otimes\chi^{-1}$, the hypothesis $\textbf{(H.p1)}$ is verified for $X=T$ since we assumed 
(\ref{eqn:assnotrivxhi}) and (\ref{eqn:chiisnotteich}) as well as for $X=T(E)$, the $p$-adic Tate module of an elliptic curve $E/F^+$ with supersingular reduction at every prime of $F^+$ above $p$. In particular, the conclusions of Proposition~\ref{prop:localfull} hold true both choices of $G_F$-representations.
\end{rem}

\subsection{Modified Selmer structures for $\mathbb{G}_m$}
\label{subsec:modifiedselmerstr}
The constructions in this subsection and the next will be only needed for sharpening the divisibility in the cyclotomic main conjecture for the CM elliptic curve $E$, which we shall prove later. The reader who is content with one divisibility in the main conjecture may skip these two subsections.

\begin{define}
\label{def:saturation}
Let $R$ be any ring and $M$ be any $R$-module. For any submodule $N\subset M$, the \emph{$R$-saturation of $N$ in $M$} is the submodule $N^{\textup{sat}}=\phi^{-1}\phi(N)\subset M$, where $\phi:M \ra M\otimes \textup{Frac}(R)$ is the natural map and $\textup{Frac}(R)$ is the total ring of fractions of $R$.
\end{define}

\begin{lemma}
\label{lem:structure for units}
The $\ooo$-module $\oo_{L}^{\times,\chi}$ is free of rank $g$.
\end{lemma}
\begin{proof}
This follows from \cite[\S8.6.12]{neukirch}, along with our assumption that $\chi$ is different from the Teichm\"uller character $\omega$.
\end{proof}

\begin{define}
\label{def:saturatedunitcondition}
\begin{itemize}
\item[(i)]Let $\mathcal{V}_F^+:=\textup{loc}_p(\oo_{L}^{\times,\chi})^\textup{sat}$ be the $\frak{O}$-saturation of $\textup{loc}_p(\oo_{L}^{\times,\chi})$ in $H^1(F_p,T)$. Note that the $\frak{O}$-module $\mathcal{V}_F^+$ is a direct summand of the free module $H^1(F_p,T)$. Let the rank of the $\ooo$-module $\mathcal{V}_F^+$ be $g-\frak{d}$ with $\frak{d}\geq 0$. Observe that $\frak{d}=0$ if Leopoldt's conjecture holds true for $L$.
\item[(ii)] Let $\mathcal{V}_F^-$ be any free submodule of $H^1(F_p,T)$ which complements $\mathcal{V}_F^+$.
\end{itemize}
\end{define}
Note that $H^1(F,T)$ may be naturally identified by $L^{\times,\chi}$ by Kummer theory, and this is how we make sense of $\textup{loc}_p(\oo_{L}^{\times,\chi})$. Furthermore, if Leopoldt's conjecture holds true for $L$, then $\mathcal{V}_F^+$ is the unique direct summand of $H^1(F_p,T)$ of rank $g$, containing $\textup{loc}_p(\oo_{L}^{\times,\chi})$.

\begin{define}
\label{def:localUpm}
\begin{itemize}
\item[(i)] Let $\mathcal{V}_{K(\frak{T})}^\pm$ be the direct summand of $\displaystyle{\varprojlim_{M\in \frak{T}}H^1(M_{p},T)}$ which maps onto $\mathcal{V}_F^\pm$ under the natural (surjective) corestriction map. Note that such a direct summand exists thanks to Proposition~\ref{prop:localfull}(i) and Nakayama's Lemma. Note further that we have the direct sum decomposition $\displaystyle{\varprojlim_{M\in \frak{T}}H^1(M_{p},T)=\mathcal{V}_{K(\frak{T})}^+\oplus \mathcal{V}_{K(\frak{T})}^-}$.
\item[(ii)] For $M\in \frak{T}$, let $\mathcal{V}_M^\pm \subset H^1(M_{p},T)$ be the image of $ \mathcal{V}_{K(\frak{T})}^\pm$ under the natural projection. 
\end{itemize}
\end{define}
\begin{define}
\label{def:modifiedlocalcondition}
\begin{enumerate}
\item[(i)] Let $\frak{L}$ be any free, rank one $\frak{O}[[\frak{G}(K(\frak{T}))]]$-direct summand of $ \mathcal{V}_{K(\frak{T})}^+$.

\item[(ii)] For $M \in \frak{T}$,  let $\frak{l}_M \subset \mathcal{V}_{M}^+$ be the image of $\frak{L}$ under the natural projection $\varprojlim_{N}H^1(N_p,T)\twoheadrightarrow H^1(M_{p},T).$
We write $\frak{l}$ instead of $\frak{l}_F$.
\end{enumerate}
\end{define}

We will make use of the following \emph{Selmer structures} on the $G_F$-representation $T$ while proving a Gras-style conjecture in Section~\ref{sec:gras} below.
\begin{define}
\label{def:selmerstructure}
By Kummer theory, we may identify $H^1(F,T)$ with $L^{\times,\chi}$ and similarly for any prime $\frak{q}$ of $F$, the local cohomology group $H^1(F_\frak{q},T)$ with $(L\otimes_F F_\frak{q})^{\times,\chi}=\,\left(\oplus_{\frak{Q}\mid \frak{q}}\,\,L_{\frak{Q}}^\times\right)^{\chi}$\,. 
\begin{itemize}
\item The \emph{canonical Selmer structure} $\FFc$ is given by the choice of local conditions 
$$H^1_{\FFc}(F_\frak{q},T)=\left(\oplus_{\frak{Q}\mid \frak{q}}\,\,\oo_{L_{\frak{Q}}}^\times\right)^{\chi}\subset H^1(F_\frak{q},T)$$ 
for all primes $\frak{q}$ of $F$.
\item The \emph{$\frak{L}$-restricted Selmer structure} $\FF_{\frak{l}}$ is given by the local conditions
\begin{itemize}
\item $H^1_{\FF_{\frak{l}}}(F_\frak{q},T)=H^1_{\FFc}(F_\frak{q},T)$ for every prime $\frak{q} \nmid p$, and
\item $H^1_{\FF_{\frak{l}}}(F_p,T)=\mathcal{V}_F^-\oplus\frak{l}$.
\end{itemize}
\item The \emph{$p$-transversal-Selmer structure} $\FF_{\textup{tr}}$ is given by the local conditions
\begin{itemize}
\item $H^1_{\FF_{\textup{tr}}}(F_\frak{q},T)=H^1_{\FFc}(F_\frak{q},T)$ for every prime $\frak{q} \nmid p$, and
\item $H^1_{\FF_{\textup{tr}}}(F_p,T)=\mathcal{V}_F^-$.
\end{itemize}
\end{itemize}
\end{define}
We refer the reader to \cite[\S2.1]{mr02} for the definition of a Selmer structure in its most general form. 
\begin{define}
Given a Selmer structure $\FF$ on $T$, we define the \emph{dual Selmer structure} $\FF^*$ on $T^*$ using local Tate duality (as in \cite[Definition 2.3.1]{mr02}).
\end{define}

 Recall the finite set $\Sigma$ of primes of $F$ which consists of all primes that ramifies in $L/F$, all archimedean primes of $F$ and all primes of $F$ above $p$. Let $F_\Sigma$ denote the maximal extension of $F$ contained in $\bar{F}$ which is unramified outside $\Sigma$ and let $G_\Sigma$ denote the Galois group $\textup{Gal}(F_\Sigma/F)$. 
 
 \begin{define}
 \label{def:selmergroups}
 For $\FF=\FFc$\,, $\FF_{\frak{l}}$\,, or $\FF_{\textup{tr}}$\,, we define the \emph{$\FF$-Selmer group} on the quotient $T$ of $\TT$ by setting
 $$H^1_\FF(F,T)=\ker\left(H^1(G_{\Sigma},T)\lra \bigoplus_{\frak{q} \in \Sigma}H^1(F_\frak{q},T)/H^1_\FF(F_\frak{q},T)\right).$$
\end{define}
\begin{example}
\label{example:kummer}
We have $H^1_{\FFc}(F,T)=\oo_L^{\times,\chi}$ and $H^1_{\FFc^*}(F,T^*)^\vee\cong\textup{Cl}(L)^\chi$. See \cite[\S6.1]{mr02} for details.
\end{example}

\subsection{Modified Selmer structures for $\mathbb{G}_m$ along $F^{\textup{cyc}}$ and $F_\infty$}
\label{subsec:modifiedselmerG_mcyclo}
We set $\TT_{\textup{cyc}}:=T\otimes\Lambda_{\textup{cyc}}$ and $\TT=T\otimes\LL$ (with diagonal $G_F$-action). The definitions we give in this section will be used to prove various forms of CM main conjectures, which will in turn be used to turn the divisibilities in the cyclotomic (supersingular) main conjecture for CM elliptic curves into equalities.
\begin{define}
\label{def:selmerstructurecyclotomic}
The \emph{canonical Selmer structure} $\FFc$ on $X$ (where $X=\TT_{\textup{cyc}}, \TT$) is given by the choice of local conditions $H^1_{\FFc}(F_\frak{q},X)=H^1_{\FFc}(F_\frak{q},X)$, 
for all primes $\frak{q}$ of $F$. Note that the associated Selmer group $H^1_{\FFc}(F,X)$ is simply the module $H^1(F,X)$.
\end{define}
\begin{lemma}
\label{lemma:H1facalongcyclotower}
Suppose that the weak Leopoldt conjecture holds true for the number field $L$. Then the $\LL_{\textup{cyc}}$-module $H^1_{\FFc}(F,\TT_{\textup{cyc}})$ is free of rank $g$.
\end{lemma}
\begin{proof}
A form of weak Leopoldt's conjecture is that the dual (canonical) Selmer group $H^1_{\FFc^*}(F,\TT_\cyc^*)$ is $\LL_\cyc$-cotorsion. It follows from the hypothesis (\ref{eqn:assnotrivxhi}) that the $\LL_{\textup{cyc}}$-module $H^1_{\FFc}(F,\TT_{\textup{cyc}})$ is torsion-free and by Poitou-Tate global duality that it is of rank $g$. Let $\gamma$ be a topological generator of $\Gamma^{\textup{cyc}}$. To see that the module $H^1_{\FFc}(F,\TT_{\textup{cyc}})$ is in fact free, observe that the augmentation map induces an injective map
$$H^1_{\FFc}(F,\TT_{\textup{cyc}})/(\gamma-1) \hookrightarrow H^1_{\FFc}(F,T)$$
by the discussion in \S1.6.C, Proposition B.3.3 along with the proof of Proposition 3.2.6 of \cite{r00}. Note that in order to compare local conditions at $p$, we rely on our assumption (\ref{eqn:assnotrivxhi}). This and Lemma~\ref{lem:structure for units} shows by Nakamaya's lemma that the $\LL_{\textup{cyc}}$-module $H^1_{\FFc}(F,\TT_{\textup{cyc}})$ may be generated by at most $g$ elements. If this set of generators satisfied a non-trivial $\LL_{\textup{cyc}}$-linear relation, it would follow that 
the dimension of the $\textup{Frac}(\LL)$ vector space $H^1_{\FFc}(F,\TT_{\textup{cyc}}) \otimes_{\LL_\cyc} \textup{Frac}(\LL_{\cyc})$ (where $\textup{Frac}(\LL_{\cyc})$ is the field of fractions of $\LL_\cyc$) is strictly smaller than $g$ and this would contradict the fact that $H^1_{\FFc}(F,\TT_{\textup{cyc}})$ is a $\LL_{\textup{cyc}}$-module of rank $g$.
\end{proof}
\begin{rem}
\label{rem:freenessconceptually}
One may use Nekov\'a\v{r}'s theory of Selmer complexes to give a more conceptual proof of Lemma~\ref{lemma:H1facalongcyclotower} (in fact, along the way, to prove also that the $\LL$-module $H^1_{\FFc}(F,\TT)$ is free of rank $g$, which what we explain in what follows). Let $\widetilde{R\Gamma}_{f,\textup{Iw}}(F_\infty/F,T)$ be Nekov\'a\v{r}'s Selmer complex associated to $\TT$, which is given by the Greenberg local conditions determined by the choice $U_v^+=T$ for every prime $v$ of $F$ above $p$. As we have assumed (\ref{eqn:assnotrivxhi}), it follows from \cite[Lemma 9.6.3]{nekovar06} (and \cite[Proposition 8.8.6]{nekovar06} used in order to pass to limit) that
$$\widetilde{H}^1_{f}(F_{\Sigma}/F_{\infty},T) \stackrel{\sim}{\lra} H^1_{\FFc}(F,\TT)$$
where $\widetilde{H}^1_{f}$ denotes the cohomology of the Selmer complex in degree $1$. Under the hypothesis (\ref{eqn:assnotrivxhi}), Nekov\'a\v{r} proved that the Selmer complex may be represented by a perfect complex concentrated in degrees $1$ and $2$. In particular, its cohomology in degree $1$ is a projective (hence free) $\LL$-module. The fact that it is of rank $g$ may also be deduced from Nekov\'a\v{r}'s control and duality theorems: We have
\begin{align}\label{eqn:cokernelofnekdescent}\textup{coker}\left(\widetilde{H}^1_{f}(F_{\Sigma}/F_{\infty},T) \lra \widetilde{H}^1_{f}(F_{\Sigma}/F^{\cyc},T)\right)&\cong \widetilde{H}^2_{f}(F_{\Sigma}/F_{\infty},T)[\gamma_*-1]\\
\notag&\cong H^1_{\FFc^*}(F,\TT^*)^{\vee}[\gamma_*-1]\end{align}
where $\gamma_*$ is a topological generator of $\Gamma/\Gamma^{\cyc}$ and the first isomorphism follows from Nekov\'a\v{r}'s control theorem \cite[8.10.1]{nekovar06}\,;  second from his duality theorem \cite[8.9.6.2]{nekovar06}. One may identify $H^1_{\FFc^*}(F,\TT^*)^{\vee}$ with $\displaystyle{\varprojlim _{L\subset M\subset LF_\infty} \textup{Cl}(M)^\chi}$ and argue using classical Iwasawa theory that the cokernel  (\ref{eqn:cokernelofnekdescent}) is $\LL_{\cyc}$-torsion and the $\LL$-module $H^1_{\FFc}(F,\TT)$ cannot be generated by less than $g$ elements. On the other hand, the proof of Lemma \ref{lemma:H1facalongcyclotower} shows that it may be generated by at most $g$ elements as well.
\end{rem}


\begin{define}
\label{define:calucyc}
Let $\mathcal{V}_F^-$ be as in Definition~\ref{def:saturatedunitcondition} and let $\mathcal{V}_{\cyc}$ be any rank-$g$ direct summand of (the free, rank-$2g$ $\LL_\cyc$-module) $H^1(F_p,\TT_\cyc)$ which lifts $\mathcal{V}_F^-$ under the surjection $H^1(F_p,\TT_\cyc)\twoheadrightarrow H^1(F_p,T)\,$. Likewise, once $\mathcal{V}_\cyc$ is chosen, let $\mathcal{V}$ be any rank-$g$ direct summand of (the free, rank-$2g$ $\LL$-module) $H^1(F_p,\TT)$ which lifts $\mathcal{V}_\cyc$ under the surjection $H^1(F_p,\TT)\twoheadrightarrow H^1(F_p,\TT_\cyc)\,$. Such lifts exist by Nakayama's lemma. Set $\mathcal{Q}:=H^1(F_p,\TT)/\mathcal{V}$ and similarly define $\mathcal{Q}_\cyc$.

Let $\al\subset H^1(F_p,\TT)$ be any rank-one direct summand of $H^1(F_p,\TT)$ such that $\al\,\cap\, \mathcal{V}=0$ and $\al+\mathcal{V}$ is a free rank $g+1$ direct summand of $H^1(F_p,\TT)$. Let $\al_\cyc$ its image in $H^1(F_p,\TT_\cyc)$. The existence of such a direct summand follows once again from Nakayama's lemma. It is also easy to observe that $\al_\cyc \cap \mathcal{V}_\cyc=0$ and $\al_\cyc+\mathcal{V}_\cyc$ is a free rank $g+1$ direct summand of $H^1(F_p,\TT_\cyc)$.
\end{define}

\begin{prop}
\label{prop:canliftfromthegroundlevel} 
The intersection of $\mathcal{V}_\cyc$ and the image of $H^1_{\FFc}(F,\TT_\cyc)$ (under the localization map at $p$) is trivial. Likewise, the intersection of $\mathcal{V}$ and the image of $H^1_{\FFc}(F,\TT)$  is trivial as well.
\end{prop}
\begin{proof}
Consider the commutative diagram
$$\xymatrix{H^1_{\FFc}(F,\TT_\cyc)\ar[r]^(.46){\textup{Loc}_p}\ar[d]& H^1(F_p,\TT_\cyc)/\mathcal{V}_\cyc\ar@{->>}[d]\\
H^1_{\FFc}(F,T)\ar@{^{(}->}[r]_(.46){\textup{loc}_p}&H^1(F_p,T)/\mathcal{V}_F^-
}$$
Suppose for $u \in H^1_{\FFc}(F,\TT_\cyc)$ we have $\textup{Loc}_p(u)=0$ and let $\bar{u}$ denote its image under the left vertical map. The diagram above shows that $\bar{u}=0$, thence 
$$u\in \ker\left(H^1_{\FFc}(F,\TT_\cyc)\ra H^1(F,T)\right)=(\gamma-1)H^1_{\FFc}(F,\TT_\cyc),$$ 
where $\gamma$ is any topological generator of $\Gamma^\cyc$. Write $u=(\gamma-1)u_0$. We have therefore have $(\gamma-1)\textup{Loc}_p(u_0)=0$ by the choice of $u$. Since the quotient $H^1(F_p,\TT_\cyc)/\mathcal{V}_\cyc$ is $\LL_\cyc$-torsion free, it follows that $\textup{Loc}_p(u_0)=0$, and repeating the argument above we conclude that $u=(\gamma-1)u_1$ with $u_1\in H^1_{\FFc}(F,\TT_\cyc)$. On running this procedure $k$ times, we conclude that $u\in (\gamma-1)^kH^1_{\FFc}(F,\TT_\cyc)$ for every $k$ and thence $u=0$ and the map $\textup{Loc}_p$ is injective, proving the first assertion. The proof of the second follows from the first in a  similar manner.
\end{proof}

\begin{define}
\label{def:localUandLcyclo}
\begin{itemize}
\item[(i)] Let $\mathcal{V}_{K(\frak{E})}$ be the direct summand (of rank $g$) of $\displaystyle{\varprojlim_{M\in \frak{E}}H^1(M_{p},T)}$ which maps onto $\mathcal{V}$ under the natural (surjective) corestriction map. 
\item[(ii)] For $M\in \frak{E}$, let $\mathcal{V}_M \subset H^1(M_{p},T)$ be the image of $\mathcal{V}_{K(\frak{E})}$ under the natural projection. 
\item[(iii)] Let $\frak{L}$ be any free, rank-one $\frak{O}[[\frak{G}(K(\frak{E}))]]$-direct summand of ${\varprojlim_{M\in \frak{E}}H^1(M_{p},T)}$ such that
\begin{itemize} 
\item $\frak{L}$ is not contained in $\mathcal{V}_{K(\frak{E})}$\,, 
\item $\frak{L} +\mathcal{V}_{K(\frak{E})}$ is also direct summand of $\varprojlim_{M\in \frak{E}}H^1(M_{p},T)$\,,
\item $\frak{L}$ maps onto $\al$ under the natural projection\,.
\end{itemize}
(Such $\frak{L}$ exists thanks to Nakayama's lemma again.)
\item[(iv)] For $M \in \frak{E}$,  let $\al_M \subset H^1(M_p,T)$ be the image of $\frak{L}$ under the natural projection $\varprojlim_{N}H^1(N_p,T)\twoheadrightarrow H^1(M_{p},T).$
\end{itemize}
\end{define}

We will make use of the following auxiliary \emph{Selmer structures} on the $G_F$-representation $\TT$ while proving various main conjectures for the field $F$ in Sections~\ref{sec:gras} and \ref{subsec:cyclomain} below. These results will in turn be utilized in sharpening the divisibility in the supersingular main conjecture for a CM elliptic curve $E$.
\begin{define}
\label{def:selmerstructure}
\begin{itemize}
\item The \emph{$\frak{L}$-restricted Selmer structure} $\FF_{\mathcal{L}}$ on $\TT$ is given by the local conditions
\begin{itemize}
\item $H^1_{\FF_{\mathcal{L}}}(F_\frak{q},\TT)=H^1_{\FFc}(F_\frak{q},\TT)$ for every prime $\frak{q} \nmid p$, and
\item $H^1_{\FF_\mathcal{L}}(F_p,\TT)=\mathcal{V}_{\textup{cyc}}\oplus\mathcal{L}$.
\end{itemize}
\item The \emph{$p$-transversal-Selmer structure} $\FF_{\frak{tr}}$ is given by the local conditions
\begin{itemize}
\item $H^1_{\FF_{\frak{tr}}}(F_\frak{q},\TT)=H^1_{\FFc}(F_\frak{q},\TT)$ for every prime $\frak{q} \nmid p$, and
\item $H^1_{\FF_{\frak{tr}}}(F_p,\TT)=\mathcal{V}$.
\end{itemize}
\end{itemize}
\end{define}
As in Definition~\ref{def:selmergroups}, all these Selmer structures give rise to a Selmer group (as well as a dual Selmer group, attached to the dual Selmer structure)
\begin{rem}
\label{rem:propagatecycselmer}
Any of the Selmer structures above \emph{propagates} (see \cite[Example 2.1.7]{mr02}) to give rise to Selmer structures on any subquotient of $\TT$. The propagated Selmer structure will still be denoted by the same symbol $\FF$.
\end{rem}
\begin{rem}
\label{rem:propgatestogroundchoice}
The Selmer structure $\FF_\frak{tr}$ on $\TT$ propagates to recover the Selmer structure $\FF_{\textup{tr}}$ on $T$, given as in Definition~\ref{def:selmerstructure}. Likewise, if the rank-$1$ direct summand $\al$ is chosen to lift $\frak{l}$ (which was given in Definition~\ref{def:modifiedlocalcondition}), then the Selmer structure $\FF_\mathcal{L}$ on $\TT$ propagates to recover the Selmer structure $\FF_{\frak{l}}$ on $T$.
\end{rem}
\subsection{Modified Selmer structures for $E$}
\label{subsec:modifiedselmerE}
We set $\TT(E)=T(E)\otimes\LL$ and $\TT_\cyc(E)=T(E)\otimes\LL_\cyc$. The goal in this section is to define various Selmer structures for these representation which we shall study with the aid of the (conjectural) Rubin-Stark elements. Note that in order to do so, we will exploit the fact that $\TT(E)$ is closely related to the representation $\TT$ for an appropriately chosen Dirichlet character $\chi$.
\subsubsection{Preliminaries}
\label{subsubsec:prelim}
As above, let $E$ be an elliptic curve defined over $F^+$ which has CM by $K$. We shall assume that $p$ is inert in $K/\QQ$. We denote the unique prime of $K$ above $p$ also by $p$ and the completion $K_p$ by $\Phi$. By slight abuse, we let $\frak{O}$ denote the ring of integers of $\Phi$ and let
$$\rho: G_F\lra \textup{Aut}(E[p^\infty])\cong\frak{O}^\times$$
be the associated $p$-adic Hecke character. For any $G_F$-module $Y$, we define its twist by $\rho$ by setting $Y(\rho):= Y\otimes\textup{Hom}(E[p^\infty],\Phi/\ooo)$. Theory of complex multiplication allows one to identify $T_p(E)$ with $\ooo(\rho)$, the free $\ooo$-module of rank $1$ on which $G_F$ acts via $\rho$. We will implicitly identify the Cartier dual $T_p(E)^*$ with $E[p^\infty]$ via the Weil pairing.  
\begin{define}
\label{def:modpGalrepofE}
Let $\omega_E:G_F\ra \ooo^\times$ denote the character which gives the action of $G_F$ on $E[p]$ and let $\langle\rho\rangle:=\rho\otimes\omega^{-1}_E$. Note then that the character $\langle\rho\rangle$ factors through $\Gamma$. 
\end{define}
Throughout this section we will set the character $\chi=\omega_E$ so that $T=\ooo(1)\otimes \omega_E^{-1}$.
\begin{define}
Let $\textup{tw}: T \ra T(E)$ (the \emph{twisting map}) denote the compositum of the maps
$$T\lra T\otimes \langle\rho\rangle^{-1} \stackrel{\mathcal{W}}{\lra} T(E)$$
where $\mathcal{W}$ is induced from Weil pairing. The twisting map induces isomorphisms
$$\textup{tw}:\,H^1(F,\TT)\stackrel{\sim}{\lra} H^1(F,\TT(E))$$
and for every place $v$ of $F$,
$$\textup{tw}:\,H^1(F_v,\TT)\stackrel{\sim}{\lra} H^1(F_v,\TT(E))\,\,.$$
\end{define}
\subsubsection{Selmer structures}
\label{subsubsec:selmerstrE}
We set $\mathcal{M}=\textup{tw}\left(\textup{loc}_p\left(H^1_{\FFc}\left(F,\TT\right)\right)\right) \subset H^1(F_p,\TT(E))$ and let $\mathcal{M}_\cyc \subset H^1(F_p,\TT_\cyc(E))$ be its projection. Note that $\mathcal{M}$ is a free $\LL$-module and $\mathcal{M}_{\cyc}$ a free $\LL_\cyc$-module, and both have rank $g$.

\begin{lemma}
\label{lem:auxtorsion}
If the $\LL_\cyc$-module $H^1_{\FFc^*}(F,\TT_{\cyc}(E)^*)^\vee$ is torsion, then so is the quotient $$\textup{loc}_p\left(H^1(F,\TT_{\cyc}(E))\right)/\mathcal{M}_\cyc\,.$$\end{lemma}
\begin{rem}
The statement that  $H^1_{\FFc^*}(F,\TT_\cyc(E)^*)^\vee$ is $\LL_\cyc$-torsion is a form of the weak Leopoldt conjecture for the elliptic curve $E$. See Corollary~\ref{cor:localnonvanishingimpliesweakLeoforE} and Theorem~\ref{thm:theweakleopoldtconjforE} below where we verify the weak Leopoldt conjecture for $E$ (at primes $p$ which split completely in $F^+/F$) assuming the Explicit Reciprocity Conjecture~\ref{conj:reciprocity} for the Rubin-Stark elements.
\end{rem}
\begin{proof}[Proof of Lemma~\ref{lem:auxtorsion}]
Let $\gamma_*$ be any lift of a topological generator of $\Gamma/\Gamma^\cyc$. Proof follows, as in the discussion of Remark~\ref{rem:freenessconceptually} (particularly, using Nekov\'a\v{r}'s control theorem as in (\ref{eqn:cokernelofnekdescent})), once we verify that 

\begin{align*}\textup{coker}\left(\widetilde{H}^1_{f}(F_{\Sigma}/F_{\infty},T(E)) \lra \widetilde{H}^1_{f}(F_{\Sigma}/F^{\cyc},T(E))\right)\cong\widetilde{H}^2_{f}(F_{\Sigma}/F_{\infty},\TT(E))[\gamma_*-1]
\end{align*}
is $\LL_\cyc$-torsion. (This is because the quotient $\textup{loc}_p\left(H^1(F,\TT_{\cyc}(E))\right)/\mathcal{M}_\cyc$ is a homomorphic image of the quotient 
$$H^1(F,\TT_\cyc(E))/\textup{im}\left(H^1(F,\TT(E))\right)\cong \widetilde{H}^1_{f}(F_{\Sigma}/F_{\cyc},T)/\textup{im}\left(\widetilde{H}^1_{f}\left(F_{\Sigma}/F_{\infty},T\right)\right)\,.)$$
Note that we have  
$$\widetilde{H}^2_{f}(F_{\Sigma}/F_{\infty},\TT(E))\cong H^1_{\FFc^*}(F,\TT(E)^*)^\vee$$ 
by \cite[8.9.6.2]{nekovar06}. Furthermore, 
$$H^1_{\FFc^*}(F,\TT(E)^*)^\vee/(\gamma_*-1)\cong  \left(H^1_{\FFc^*}(F,\TT(E)^*)[\gamma_*-1]\right)^\vee\cong  H^1_{\FFc^*}(F,\TT_\cyc(E)^*)^\vee$$
where the second isomorphism is by \cite[Lemma 3.5.3]{mr02} and hence, we conclude thanks to our assumption that $H^1_{\FFc^*}(F,\TT(E)^*)^\vee/(\gamma_*-1)$  is $\LL_\cyc$-torsion. 

We now conclude by \cite[Lemme I.3.4(ii)]{PR84memoirs} that 
$$H^1_{\FFc^*}(F,\TT(E)^*)^\vee[\gamma_*-1]\cong \widetilde{H}^2_{f}(F_{\Sigma}/F_{\infty},\TT(E))[\gamma_*-1]$$
is $\LL_\cyc$-torsion as well.
\end{proof}
\begin{define}
\label{def:choosetransverseforE}
Let $\mathbb{V}^\cyc_E \subset H^1(F_p,\TT_\cyc(E))$ be a free rank-$g$ direct summand with the property that $\mathbb{V}^\cyc_E \cap \mathcal{M}_\cyc=0$. Note that such a direct summand exists by rank considerations. Let $\mathbb{V}_E \subset H^1(F_p,\TT(E))$ be any  rank-$g$ direct summand which maps onto $\mathbb{V}^\cyc_E$ under the surjection $H^1(F_p,\TT(E))\ra H^1(F_p,\TT_\cyc(E))$\,.
\end{define}
See Remark~\ref{rem:localconditionsnatural} below for a natural choice of $\mathbb{V}^\cyc_E$ under the additional hypothesis that $p$ splits completely in $F^+/\QQ$, using Kobayashi's plus/minus Iwasawa theory. We will use these choices in order to prove the main theorems of this article.
\begin{rem}
\label{rem:liftistransversalaswell}
The proof of Proposition~\ref{prop:canliftfromthegroundlevel} may be modified to prove that $\mathbb{V}_E\cap \mathcal{M}=0$ as well.
\end{rem}
\begin{define}
\label{def:twistbacktolocalconditionsonGm}
Let $\mathbb{V}:=\textup{tw}^{-1}\left(\mathbb{V}_E\right) \subset H^1(F_p,\TT)$. Let $\mathbb{L} \subset H^1(F_p,\TT)$ a rank-one direct summand of $H^1(F_p,\TT)$ such that $\mathbb{L}\,\cap\, \mathbb{V}=0$ and $\mathbb{L}+\mathbb{V}$ is a free rank $g+1$ direct summand of $H^1(F_p,\TT)$. As before, the existence of such a direct summand follows from Nakayama's lemma. Let $\mathbb{L}_E \subset H^1(F_p,\TT(E))$ denote its isomorphic image and $\mathbb{L}_E^{\cyc}$ its image under the projection map to $H^1(F_p,\TT_\cyc(E))$.
\end{define}
Note that $\mathbb{V}\,\cap\,\textup{loc}_p\left(H^1_{\FFc}\left(F,\TT\right)\right)=0$ by the observation in Remark~\ref{rem:liftistransversalaswell}.



\begin{define}
\label{def:selmerstructureE}
\begin{itemize}
\item The \emph{canonical Selmer structure} $\FFc$ is given by the choice of local conditions $H^1_{\FFc}(F_\frak{q},\TT(E))=H^1(F_\frak{q},\TT(E))$, for all primes $\frak{q}$ of $F$.
\item The \emph{$\mathbb{L}$-restricted Selmer structure} is given by the local conditions
\begin{itemize}
\item $H^1_{\FF_\mathbb{L}}(F_\frak{q},\TT(E))=H^1_{\FFc}(F_\frak{q},\TT(E))$ for every prime $\frak{q} \nmid p$, and
\item $H^1_{\FF_{\mathbb{L}}}(F_p,\TT(E))=\mathbb{V}_E\oplus\mathbb{L}_E$.
\end{itemize}
\item The \emph{Kobayashi Selmer structure} $\FF_{\textup{Kob}}$ is given by the local conditions
\begin{itemize}
\item $H^1_{\FF_{\textup{Kob}}}(F_\frak{q},\TT(E))=H^1_{\FFc}(F_\frak{q},\TT(E))$ for every prime $\frak{q} \nmid p$, and
\item $H^1_{\FF_{\textup{Kob}}}(F_p,\TT(E))=\mathbb{V}_E$.
\end{itemize}
\end{itemize}
\end{define}

Given a Selmer structure $\FF$ on $\TT(E)$, we can talk about the \emph{dual Selmer structure} $\FF^*$ on $\TT^*$, the Selmer group attached to it and its propagations to the various subquotients of $\TT(E)$ (most important of which are $\TT_\cyc(E)$ and $T(E)$ for our purposes), as we have done so previously. Via the twisting isomorphism $\textup{tw}$, we also obtain a Selmer structure (by a slight abuse, which we still denote by $\FF_\mathbb{L}$ or $\FF_{\textup{Kob}}$) on $\TT$ (and on its various subquotients). 
\subsection{Global duality and comparison of Selmer groups}
\label{subsec:compareselmer}
Let $R$ be a complete local noetherian domain with maximal ideal $\frak{M}$. Let $X$ be an $R$-module of finite type. We will write $\overline{X}:=X\otimes_R R/\frak{M}$ for its reduction modulo the maximal ideal of $R$. Note that 
$$\overline{\TT}=\overline{\TT}_\cyc=\overline{T}=\pmb{\mu}_p\otimes\chi^{-1}$$ 
and
$$\overline{\TT}(E)=\overline{\TT}_\cyc(E)=\overline{T(E)}=E[p]$$
as $G_F$-representations. 
In particular, when $\chi$ is chosen to be $\omega_E^{-1}$, it follows thanks to the Weil pairing that all the six residual representations we consider above agree.

Let $k$ denote the residue field of $\ooo$.
\begin{lemma}
\label{lem:comparefirststep} Assume the truth of Leopoldt's conjecture for the number field $L$. We have
$$\dim_{k}\, H^1_{\FF}(F,\overline{T})=\dim_{k}\, H^1_{\FF^*}(F,\overline{T}^*)$$
for $\FF=\FF_\textup{tr},\FF_\frak{tr}$ or $\FF_\textup{Kob}$ and 
$$\dim_{k}\, H^1_{\mathcal{G}}(F,\overline{T})=\dim_{k}\, H^1_{\mathcal{G}^*}(F,\overline{T}^*)+1$$
for $\mathcal{G}=\FF_\frak{l}, \FF_\mathcal{L}$ or $\FF_\mathbb{L}$\,. $($Note that when $\FF=\FF_\textup{Kob}$ or $\mathcal{G}=\FF_{\mathbb{L}}$ we only consider the case $\chi=\omega_E$.$)$
\end{lemma}

\begin{proof}
As explained in Example~\ref{example:kummer} we have  $H^1_{\FFc}(F,T)\cong \oo_{L_\chi}^{\times,\chi}$, and hence Lemma~\ref{lem:structure for units} shows that the $\ooo$-module $H^1_{\FFc}(F,T)$ is free of rank $g$ under the running assumptions. On the other hand, $H^1_{\FFc^*}(F,T^*)\cong \textup{CL}(L)^\chi$ is finite and it follows from the discussion in Section 5.2 of \cite{mr02} that 
\begin{align}\label{eqn:corerankforfcan}
\notag\dim_{k}\, H^1_{\FFc}(F,\overline{T})-\dim_{k}\, H^1_{\FFc^*}(F,\overline{T}^*)&=\textup{rank}_{\ooo}\, H^1_{\FFc}(F,T)-\textup{corank}_{\ooo}\, H^1_{\FFc^*}(F,{T}^*)\\
&=g
\end{align}
Observe that we have by the choices we have made that 
$$\dim_k\, H^1_{\FFc}(F_p,\overline{T}) -\dim_k\,H^1_{\FF}(F_p,\overline{T})=g$$
for $\FF=\FF_\textup{tr},\FF_\frak{tr}$ or $\FF_\textup{Kob}$ and  
$$\dim_k\, H^1_{\FFc}(F_p,\overline{T}) -\dim_k\,H^1_{\mathcal{G}}(F_p,\overline{T})=g-1$$
for $\mathcal{G}=\FF_\frak{l}, \FF_\mathcal{L}$ or $\FF_\mathbb{L}$\,. Proposition 1.6 of \cite{wiles} shows that
\begin{align*}
\left(\dim_{k}\, H^1_{\FFc}(F,\overline{T})-\dim_{k}\, H^1_{\FFc^*}(F,\overline{T}^*)\right)&-\left(\dim_k\, H^1_{\FF}(F,\overline{T})-\dim_k\, H^1_{\FF^*}(F,\overline{T}^*)\right)\\
&=\dim_k\, H^1_{\FFc}(F_p,\overline{T}) -\dim_k\,H^1_{\FF}(F_p,\overline{T})\\
&=g\,.
\end{align*}
The first part of the proposition follows from \eqref{eqn:corerankforfcan} and the second part may also be deduced by replacing $\FF$'s by $\mathcal{G}$'s.
\end{proof}
\begin{rem}
\label{rem:KSforLrestrictedSelmerstr}
Throughout this paragraph, $\FF$ will stand for any of $\FF_\frak{l}, \FF_\mathcal{L}$ or $\FF_\mathbb{L}$, with the convention that if $\FF=\FF_\mathbb{L}$ the $\chi=\omega_E$. Corollary 4.5.2 of \cite{mr02} asserts that the module of Kolyvagin systems $\textbf{KS}(\FF,\overline{T})$ is a $k$-vector space of dimension one , thanks to the second part of Lemma~\ref{lem:comparefirststep}. On the other hand, it follows from the  main theorem of \cite{kbbdeform} that these \emph{residual} Kolyvagin systems deform to $X$ (where $X=\TT,\TT(E), \TT_\cyc$ or $\TT_\cyc(E)$) and that the module $\overline{\textbf{KS}}(\FF,X)$ is free of rank one over the corresponding coefficient ring. The elements of these modules (namely, Kolyvagin systems) are used to bound the characteristic ideal of $H^1_{\FF^*}(F,X)^\vee$. The generators of the module of Kolyvagin systems are characterised by the property that the bounds they give on the characteristic ideal of $H^1_{\FF^*}(F,X)^\vee$ are sharp. 

We will later use the (conjectural) Rubin-Stark elements to construct these Kolyvagin systems and exploit facts recalled above in order to verify the sharpness of the bounds we shall obtain on the Kobayashi Selmer groups for the CM elliptic curve $E$.
\end{rem}
\begin{prop}
\label{prop:minusvanishingforunits}Assume that Leopoldt's conjecture holds for $L$. Then,
$$H^1_{\FF_{\textup{tr}}}(F,T)=H^1_{\FF_{\frak{tr}}}(F,\TT_{\cyc})=H^1_{\FF_{\frak{tr}}}(F,\TT)=0$$
and if $H^1_{\FFc^*}(F,\TT_\cyc(E)^*)^\vee$ is $\LL_\cyc$-torsion,
$$H^1_{\FF_{\textup{Kob}}}(F,\TT_{\cyc}(E))=H^1_{\FF_{\textup{Kob}}}(F,\TT(E))=0$$
\end{prop}

\begin{proof}
The first group of assertions follow from the definitions.
Let $\mathcal{M}_\cyc$ be as at the start of Section~\ref{subsubsec:selmerstrE}. The quotient $\textup{loc}_p\left(H^1(F,\TT_{\cyc}(E))\right)/\mathcal{M}_\cyc$ is a torsion $\LL_\cyc$-module by Lemma~\ref{lem:auxtorsion} (under our assumption of the weak Leopoldt conjecture for $T(E)$). Since $\mathcal{M}_\cyc\,\cap\,\mathbb{V}_E^\cyc=0$ by our very choice of $\mathbb{V}_E^\cyc$ and since $H^1(F_p,\TT_\cyc)/\mathbb{V}_E^\cyc$ is torsion free, it follows that $\textup{loc}_p\left(H^1(F,\TT_{\cyc}(E))\right) \,\cap\, \mathbb{V}_E^\cyc=0$. This means 
$$H^1_{\FF_{\textup{Kob}}}(F,\TT_{\cyc}(E)):=\ker\left(H^1(F,\TT_{\cyc}(E))\stackrel{\textup{loc}_p}{\lra} \mathbb{V}_E^\cyc\right)=0\,.$$  


Let $\gamma_*\in \Gamma$ be any lift of a topological generator of $\Gamma/\Gamma^\cyc$. The exact sequence
$$0\lra \TT\stackrel{\gamma_*-1}{\lra} \TT \lra \TT_{\textup{cyc}}\lra 0$$
yields and injection
$$H^1_{\FF_{\textup{Kob}}}(F,\TT)\Big{/} (\gamma_*-1)\hookrightarrow H^1_{\FF_{\textup{Kob}}}(F,\TT_\cyc)=0,$$
and we conclude by Nakayama's Lemma that $H^1_{\FF_{\textup{Kob}}}(F,\TT)=0$ as well.
\end{proof}

\begin{prop}
\label{prop:4termexact}
Assume that Leopoldt's conjecture holds for $L$ and the weak Leopoldt conjecture for $T(E)$. Let $(\FF,\mathcal{G},\frak{D}, X)$ be any of the following quadruples: 
\begin{align*}\{(\FF_{\textup{tr}},\FF_{\frak{l}},\frak{l},T), (\FF_{\frak{tr}},\FF_{\mathcal{L}},\frak{\mathcal{L}_\cyc},\TT_\cyc), (\FF_{\frak{tr}},\FF_{\mathcal{L}}, & \,\frak{\mathcal{L}},\TT), \\&(\FF_{\textup{Kob}},\FF_{\mathbb{L}},\,\mathbb{L}_\cyc,\TT_\cyc(E)),  (\FF_{\textup{Kob}},\FF_{\mathbb{L}},\mathbb{L},\TT)\}\end{align*}
Then the following sequence is exact:
$$0\lra{H^1_{\mathcal{G}}(F,X)}\stackrel{\textup{loc}_p}{\lra}\frak{D} {\lra} H^1_{\FF^*}(F,X^*)^{\vee}\lra{H^1_{\mathcal{G}^*}(F,X^*)^{\vee}}\lra 0.$$
\end{prop}

\begin{proof}
This follows from Poitou-Tate global duality, used along with Proposition~\ref{prop:minusvanishingforunits}.
\end{proof}

\section{Rubin-Stark Euler system of rank $r$}
\label{sec:RS}
We review Rubin's \cite{ru96} integral refinement of Stark's conjectures which we will later use to construct Kolyvagin systems for the modified Selmer structure $\FF_{\mathcal{L}}$ on $\TT$. For the rest of this paper, we assume that the Rubin-Stark conjecture~\cite[Conjecture~$\textup{B}^{\prime}$]{ru96} holds true for the fields which appear in this article.

Let $\chi,f_\chi$ and $L$ be as above, and recall the definitions of the collections of extensions $\frak{E}_0$ and $\frak{E}$ from \S\ref{subsec:notation}. Fix forever a finite set $S$ of places of $F$ that does \emph{not} contain any prime above $p$, but contains the set of infinite places $S_\infty$ and all primes $\lambda \nmid p$ at which $\chi$ is ramified. Assume that $|S| \geq g+1$. For each $\mathcal{M} \in \frak{E}_0$, let
\begin{align*}S_{\mathcal{M}}=\{\hbox{places of } \mathcal{M} \hbox{ lying above the places in } S\}\, \cup \, \{\hbox{places of }& \mathcal{M} \hbox{ at which } 
\\&\mathcal{M}/F \hbox{ is ramified}\}\end{align*}
be a set of places of $\mathcal{M}$. Let $\mathcal{O}_{\mathcal{M},S_\mathcal{M}}^{\times}$ denote the $S_\mathcal{M}$-units of $\mathcal{M}$, and $\Delta_\mathcal{M}$ ({resp.}, $\delta_\mathcal{M}$) denote $\hbox{Gal}(\mathcal{M}/F)$ ({resp.}, $|\hbox{Gal}(\mathcal{M}/F)|$).  

\begin{define}
\label{def:rubinswedge0}
Let $G$ be any finite group and let $X$ be any $\ooo[G]$-module which is of finite type over $\ooo$. Following \cite{ru96}, we define for any  integer $r \geq 0$ the submodule $\wedge^r_0 X \subset \Phi\otimes\wedge^r X $
by setting
\begin{align*} \wedge^r_0 X=\left\{x \in \Phi\otimes\wedge^r X: (\varphi_1\wedge\cdots\wedge\varphi_r)\right.&(x)\in \ooo[G]\\
&\left. \hbox{ for every } \varphi_1,\cdots,\varphi_r \in \textup{Hom}(X,\ooo[G])\right\}.
\end{align*}
We also let $\overline{\wedge^r X}$ denote the isomorphic image of $\wedge^r X$ under the map $j: \wedge^r X\ra \Phi\otimes\wedge^r X$.
\end{define}
\begin{example}
\label{example}
If $X$ is a free $\ooo[G]$-module then $\wedge^r_0 X =\overline{\wedge^r X}$. In general, $|G|\cdot\wedge^r_0 X\subset\overline{\wedge^r X}$.
 \end{example}

Rubin in~\cite[Conjecture $\textup{B}^{\prime}$]{ru96} predicts the existence of certain elements 
$$\tilde{\varepsilon}_{\mathcal{M},S_\mathcal{M}} \in {\wedge^g_0}\, \mathcal{O}_{\mathcal{M},S_\mathcal{M}}^{\times}$$ 
linked via a regulator map to the value of the corresponding Artin $L$-function at $s=0$.
\begin{rem}\label{rem:T}
Rubin's conjecture predicts that the elements $\tilde{\varepsilon}_{\mathcal{M},S_\mathcal{M}}$ should in fact lie inside  the module ${\wedge^g_0}\, \mathcal{O}_{\mathcal{M},S_\mathcal{M},\mathcal{T}}^{\times}$ where $\mathcal{T}$ is a finite set of primes disjoint from $S_\mathcal{M}$, chosen in a way that the group $\mathcal{O}_{\mathcal{M},S_\mathcal{M},\mathcal{T}}^{\times}$ of $S_\mathcal{M}$-units which are congruent to $1$ modulo all the primes in $\mathcal{T}$ is torsion-free. As explained in \cite[Remark 3.1]{kbbiwasawa}, one can safely ignore $\mathcal{T}$ as far as we are concerned in this paper.
\end{rem}

As further explained in \cite[\S3.1]{kbbCMabvar}, the Rubin-Stark elements may be used to construct an \emph{Euler sytem of rank $g$} for $T$ (in the sense of \cite{pr-es}, as appropriately generalized in \cite{kbbesrankr} so as to allow denominators). We omit the details here and refer the reader to \cite{kbbCMabvar}. This Euler system of rank $g$ is a collection $\mathcal{C}^{(g)}_{\textup{R-S}}=\{\varepsilon_{\KKK}^{\chi}\}_{_{\KKK\in \frak{E}}}$ where $\varepsilon_{\KKK}^{\chi}\in \wedge^g_0\, H^1(\KKK,T)$.
The collection $\mathcal{C}^{(g)}_{\textup{R-S}}$ will be called \emph{Rubin-Stark Euler system of rank g} for $T$.
\subsection{Strong Rubin-Stark Conjectures}
\label{subset:strongRS}
Let $F_\dagger \subset F_\infty$ be any $\ZZ_p$-extension of $F$ disjoint from $F^\cyc$ over $F$ and let $\Gamma_\dagger=\textup{Gal}(F_\dagger/F)$ so that we have $\Gamma=\Gamma_\dagger\times\Gamma_\cyc$. Let $\gamma_\dagger$ be a topological generator of $\Gamma_\dagger$ and let $\gamma_\cyc$ denote a fixed topological generator of $\Gamma^\cyc$.

Given positive integers $m,n$ we let $F_\cyc\subset F^\dagger_n\subset F_\infty$ denote the fixed field of $\Gamma_\dagger^{p^n}$ (so that we have $\textup{Gal}(F_n^\dagger/F)=\Gamma_\cyc\times \Gamma_\dagger^{(n)}$ with $ \Gamma_\dagger^{(n)}=\Gamma_\dagger/\Gamma_\dagger^{p^n}$) and let $F\subset F_{m,n} \subset F^\dagger_n$ be the fixed field of $\Gamma_\cyc^{p^m}$ (so that  $\textup{Gal}(F_{m,n}/F)=\Gamma_{(m)}\times \Gamma^{(n)}$ with $\Gamma_{(m)}=\Gamma_\cyc/\Gamma_\cyc^{p^m}$). We write $F_{(m)}=F_{m,0} \subset F^\cyc$ and $F^{(n)}=F_{0,n}\subset F_\dagger$. Observe that $F_{m,n}$ is the joint of $F_{(m)}$ and $F^{(n)}$. The following diagram summarizes the definitions in this paragraph:
$$
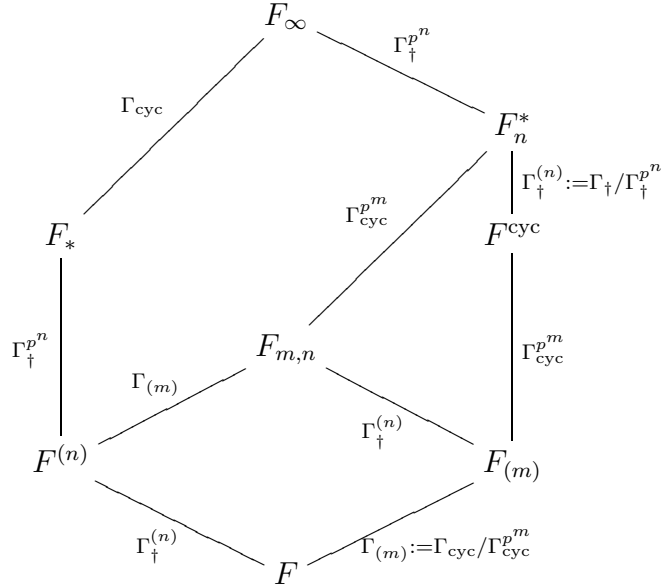
\xymatrix{&&F_\infty\ar@{-}[rrd]^{\Gamma_\dagger^{p^n}}\ar@{-}[lldd]_{\Gamma_\cyc}&&\\
&&&&F^*_n\ar@{-}[d]^{\Gamma_\dagger^{(n)}:=\Gamma_\dagger/\Gamma_\dagger^{p^n}}\ar@{-}[lldd]_{\Gamma_\cyc^{p^m}}\\
F_*\ar@{-}[dd]_{\Gamma_\dagger^{p^n}}&&&& F^\cyc\ar@{-}[dd]^{\Gamma_\cyc^{p^m}}\\
&&F_{m,n}\ar@{-}[lld]_{\Gamma_{(m)}}\ar@{-}[rrd]_{\Gamma_\dagger^{(n)}}&&\\
F^{(n)}\ar@{-}[rrd]_{\Gamma_\dagger^{(n)}}&&&&F_{(m)}\ar@{-}[lld]^{\Gamma_{(m)}:=\Gamma_\cyc/\Gamma_\cyc^{p^m}}\\
&&F&&}$$
\begin{prop}
\label{prop:freenessoverfinitelayers}
Let $m,n$ be arbitrary positive integers.
\begin{itemize}
\item[(i)] $\textup{coker}\left(H^1(F,\TT_\cyc)\ra H^1_{\FFc}(F_{(m)},T)\right)$ is finite.
\item[(ii)] The $\ooo[\Gamma_{(m)}]$-module $H^1_{\FFc}(F_{(m)},T)$ is free of rank $g$.
\item[(iii)]The $\ooo[\Gamma_{(m)}\times\Gamma^{(n)}]$-module $H^1_{\FFc}(F_{(m,n)},T)$ is free of rank $g$.
\end{itemize}
\end{prop}
\begin{proof}
We argue as in Remark~\ref{rem:freenessconceptually}. By Nekov\'a\v{r}'s control theorem  
$$\textup{coker}\left(H^1(F,\TT_\cyc)\ra H^1_{\FFc}(F_{(m)},T)\right)\cong \widetilde{H}^2_f(F_\Sigma/F,T)[\gamma^{p^m}_\cyc-1]$$
and $ \widetilde{H}^2_f(F_\Sigma/F_{(m)},T)\cong H^1_{\FFc^*}(F,\TT_\cyc^*)^\vee$. Since 
$$H^1_{\FFc^*}(F,\TT_\cyc^*)^\vee/(\gamma^{p^m}_\cyc-1) \cong H^1_{\FFc^*}(F_{(m)},T^*)^\vee\cong \textup{Cl}(LF_{(m)})^\chi$$
is finite, the characteristic ideal of the torsion $\LL_\cyc$-module $H^1_{\FFc^*}(F,\TT_\cyc^*)^\vee$ is prime to $\gamma^{p^m}_\cyc-1$, and by the structure theorem for finitely generated $\LL_\cyc$-modules we see that $H^1_{\FFc^*}(F,\TT_\cyc^*)^\vee[\gamma^{p^m}_\cyc-1]$ is finite, concluding the proof of (i). 

The argument above may be used to prove that $\textup{coker}\left(H^1(F,\TT_\cyc)\ra H^1_{\FFc}(F,T)\right)$ is finite, which in turn implies that $\textup{coker}\left(H^1(F_{(m)},T)\stackrel{\textup{pr}}{\lra} H^1_{\FFc}(F,T)\right)$ is finite as well. Thence the image of the map $\textup{pr}$
(induced by projection modulo $\gamma_\cyc-1$) is a free $\ooo$-module of rank $g$. It follows by Nakayama's lemma that the $\ooo[\Gamma_{(m)}]$-module $H^1(F_{(m)},T)$ may be generated by at most $g$ elements, say by $\{v_1,\cdots,v_g\}$. On the other hand, it follows from the first part that $H^1(F_{(m)},T)$ contains a free $\ooo[\Gamma_{(m)}]$-module of rank $g$ (isomorphic image of the free module $H^1(F,\TT_\cyc)/(\gamma_\cyc^{p^m}-1)$), say with basis $\{y_1,\cdots,y_g\}$. One may easily verify that any non-trivial   $\ooo[\Gamma_{(m)}]$-linear relation $\{v_1,\cdots,v_g\}$ would yield a non-trivial  $\ooo[\Gamma_{(m)}]$-linear relation of $\{y_1,\cdots,y_g\}$, which is impossible. This shows that $\{v_1,\cdots,v_g\}$ is indeed a basis and (ii) follows. 

The proof of (iii) follows similarly. We indicate the main steps. First, we verify that the $\ooo[[\Gamma_{(m)}\times\Gamma_\dagger]]$-module $H^1_{\FFc}(F_{(m)},T\otimes\LL_{\dagger})$ is free of rank $g$. Next, we check that  the map $H^1_{\FFc}(F_{(m)},T\otimes\LL_{\dagger})\ra H^1_{\FFc}(F_{m,n},T)$ has finite cokernel, thence $H^1_{\FFc}(F_{m,n},T)$ contains a free $\ooo[\Gamma_{(m)}\times\Gamma^{(n)}]$-module of rank $g$ (with finite index in $ H^1_{\FFc}(F_{m,n},T)$), say again with basis $\{y_1,\cdots,y_g\}$. Furthermore, it follows by Nakayama's lemma that $H^1_{\FFc}(F_{m,n},T)$ may be generated by at most $g$ elements, say by $\{v_1,\cdots,v_g\}$. It is easy to check as above that a non-trivial linear relation of $\{v_1,\cdots,v_g\}$ would yield a non-trivial relation among $\{y_1,\cdots,y_g\}$, concluding the proof that $\{v_1,\cdots,v_g\}$ is a basis of $H^1_{\FFc}(F_{m,n},T)$.
\end{proof}
\begin{rem}
\label{rem:RSaresomewhatintegral}
By Proposition~\ref{prop:freenessoverfinitelayers}(iii) it follows that 
$$\varepsilon_{F_{m,n}}^{\chi} \in \wedge^g H^1_{\FFc}(F_{m,n},T)\,,$$
since we have $\wedge^g_0 H^1_{\FFc}(F_{m,n},T)=\wedge^g H^1_{\FFc}(F_{m,n},T)$ by Example~\ref{example}. 
\end{rem}
Inspired by \cite[Definition 1.2.2]{pr-es}, we propose the following strengthening (along the tower $F_\infty/F$) of the Rubin-Stark conjectures:
\begin{conj}[Strong Rubin-Stark conjecture]
\label{conj:strongRS}
There exists an element 
$$\frak{S}_\infty=\frak{S}_{\infty,1}\wedge \cdots \wedge \frak{S}_{\infty,g} \in \wedge^g H^1(F,\TT)$$
(where the exterior product is evaluated in the category of $\LL$-modules) such that for every subextension $F\subset M=F_{m,n} \subset F_\infty$ as above, the image of $\frak{S}_{\infty}$ under the natural projection to $\wedge^g H^1_{\FFc}(M,T)$ is $\varepsilon_{M}^{\chi}$, the $\chi$-isotypic component of the Rubin-Stark element.

Assuming the truth of the Strong Rubin-Stark conjecture, we set 
$$\frak{S}_{\cyc}=\frak{S}_{\cyc,1}\wedge \cdots \wedge \frak{S}_{\cyc,g} \in \wedge^g H^1(F,\TT_\cyc)$$
to denote the image of $\frak{S}_{\infty}$.
\end{conj}
\begin{rem}
\label{rem:strongequivtoordinaryRS}
If we knew that neither the $\LL$-module ${\varprojlim_{L\subset M\subset LF_\infty}\textup{Cl}(M)^\chi}$ nor the $\LL_\cyc$-module ${\varprojlim_{L\subset M\subset LF^\cyc}\textup{Cl}(M)^\chi}$ has no pseudo-null submodules, Strong Rubin-Stark conjecture would have been trivial. Indeed in that case,  it follows that the maps 
$$H^1_{\FFc}(F_{m,n},T)\lra H^1_{\FFc}(F_{m^\prime,n^\prime},T)$$
(for positive integers $m\geq m^\prime$ and $n\geq n^\prime$) are surjective and using Proposition~\ref{prop:freenessoverfinitelayers} that 
$$\varprojlim \wedge^g H^1_{\FFc}(F_{m,n},T)=\wedge^g \varprojlim H^1_{\FFc}(F_{m,n},T)=\wedge^g H^1(F,\TT)\,.$$
\end{rem}
\section{Kolyvagin systems for $\mathbb{G}_m$ and $E$}
\label{Sec:KSforGmandE}
\emph{Until the end of this paper, we assume the truth of Leopoldt's conjecture for $L$}. This in particular shows that $\frak{d}=0$. Let $\PP$ be the set of all primes of $F$ that complements the set of primes $F$ at which $T$ is ramified and the set of primes above $p$. Let 
$$\overline{\mathbf{KS}}(\TT,\FF_{\al},\PP):=\varprojlim_{k, \bar{\alpha}} \left(\varinjlim_{\substack{k^\prime\geq k, \\ \bar{\beta} \succ \bar{\alpha}}} \mathbf{KS}(\TT_{k,\bar{\alpha}},\FF_{\al}, \PP_{k^\prime,\bar{\beta}})\right)$$ 
denote the module of $\al$-restricted Kolyvagin systems for the triple $(\TT,\FF_{\al},\PP)$. Here we borrowed notation from \cite[Appendix A]{kbbCMabvar}; we note for the convenience of the reader that 
\begin{itemize}
\item we have $r=3$ in this portion of the current article, 
\item $\bar{\alpha}$ and $\bar{\beta}$ stand for triples of positive integers,
\item our $\FF_{\mathcal{L}}$ corresponds to $\FF_{\mathcal{L}_\infty}$ in loc. cit.
\end{itemize}
We similarly define the modules $\overline{\mathbf{KS}}(X,\FF,\PP)$ where $(X,\FF)$ is one of the pairs $(\TT_\cyc,\FF_{\al})$, $(T,\FF_{\al})$, $(\TT(E),\FF_{\mathbb{L}})$ or $(\TT_\cyc(E),\FF_{\mathbb{L}})$. 

It follows from \cite[Theorem 5.1.1]{mr02} and Lemma~\ref{lem:comparefirststep} that the $k$-vector space ${\mathbf{KS}}(\overline{T},\FF,\PP)$ has dimension one ($\FF=\FF_\al$ or $\mathbb{\FF_\mathbb{L}}$). The following theorem asserts that these Kolyvagin systems may be lifted to various deformations of $\overline{T}$.

\begin{thm}
\label{thm:lamdaadicKS}$\,$
\begin{itemize}
\item[(i)] Both $\LL$-modules of  $\overline{\mathbf{KS}}(\TT,\FF_{\al},\PP)$ and $\overline{\mathbf{KS}}(\TT(E),\FF_{\mathbb{L}},\PP)$ as well as the $\LL_\cyc$-modules $\overline{\mathbf{KS}}(\TT,\FF_{\al},\PP)$ and $\overline{\mathbf{KS}}(\TT_\cyc(E),\FF_{\mathbb{L}},\PP)$ and the $\ooo$-module $\overline{\mathbf{KS}}(T,\FF_{\al},\PP)$ are free of rank one.
\item[(ii)] All five free modules in $(\textup{i})$ are generated by a \textbf{primitive} {Kolyvagin system} $\pmb{\kappa}$, namely by a Kolyvagin system whose  image $\overline{\pmb{\kappa}}\in {\mathbf{KS}}(\overline{T},\FF,\PP)$ $($where $\FF=\FF_{\al}$ or $\FF_{\mathbb{L}}$ depending on which module of Kolyvagin systems we are talking about$)$ is non-zero.
\end{itemize}
\end{thm}
\begin{proof}
The assertions in (i) and (ii) over $\ooo$ is \cite[Theorem 5.2.10]{mr02} and over $\LL$ or $\LL_\cyc$, they both follow from \cite[Theorem A.14]{kbbCMabvar}.
\end{proof}
The following theorem summarizes the main applications of the Kolyvagin systems whose existence are guaranteed by the previous Theorem. Let $(R,X,\FF)$ be any one of the five triples $(\ooo,T,\FF_\al)$, $(\LL_\cyc,\TT_\cyc,\FF_{\mathcal{L}})$, $(\LL,\TT,\FF_{\mathcal{L}})$, $(\LL_\cyc,\TT_\cyc(E),\FF_{\mathbb{L}})$ or $(\LL,\TT(E),\FF_{\mathbb{L}})$.
\begin{thm}
\label{thm:mainapplicationofrestrictedKS} Suppose that $\pmb{\kappa}\in \overline{\mathbf{KS}}(X,\FF,\PP)$ is a Kolyvagin system whose initial term $\kappa_1 \in H^1_{\FF}(F,X)$ is non-zero.
\begin{itemize}
\item[(i)] The $R$-module $H^1_{\FF^*}(F,X^*)^\vee$ is $R$-torsion and the $R$-module $H^1_{\FF}(F,X)$ has rank one.
\item[(ii)] If $R=\ooo$, then $\#H^1_{\FF^*}(F,X^*)^\vee \mid \#\left(H^1_{\FF}(F,X)/R\cdot\kappa_1\right)$. If $R=\LL$ or $\LL_\cyc$, then 
$$\textup{char}\left(H^1_{\FF^*}(F,X^*)^\vee\right)\mid \textup{char}\left(H^1_{\FF}(F,X)/R\cdot\kappa_1\right)\,.$$
\item[(iii)] When $R=\ooo$ or $R=\LL_\cyc$, we have equality in the divisibilities of $(\textup{ii})$ if and only if the Kolyvagin system $\pmb{\kappa}$ is primitive.
\end{itemize}
\end{thm}
\begin{proof}
When $R=\ooo$ all assertions follow from \cite[\S5.2]{mr02}. The arguments of \cite[\S5.3]{mr02} essentially verify all the three assertions when $R=\LL_\cyc$ as well. Here we provide a sketch of their proof in that case.

For (i), we may choose a height one prime ideal $\wp=(\gamma_\cyc-1+p^N)$ of $\LL_\cyc$ (where $N \in \ZZ^+$) such that
\begin{itemize}
\item $\LL_\cyc/(\gamma_\cyc-1)\cong \LL_\cyc/\wp$, 
\item The image $\kappa_1^\wp\in H^1_{\FF}(F,X\otimes\LL/\wp)$ of $\kappa_1$ is non-zero. 
\end{itemize} 
Note that $\kappa_1^\wp$ is the initial term of the Kolyvagin system $\pmb{\kappa}^\wp \in \overline{\mathbf{KS}}(X\otimes\LL/\frak{P},\FF,\PP)$ and it follows from (i) applied with $R=\LL/\wp\cong \ooo$ that the $\LL/\wp$-module 
$$H^1_{\FF^*}\left(F,\left(X\otimes\LL/\wp\right)^*\right)^\vee=H^1_{\FF^*}\left(F,X^*[\wp]\right)^\vee\cong H^1_{\FF^*}\left(F,X^*\right)^\vee/\wp H^1_{\FF^*}\left(F,X^*\right)^\vee$$
is finite. This shows by the structure theorem for finitely generated $\LL_\cyc$-modules that the $\LL_\cyc$-module $H^1_{\FF^*}\left(F,X^*\right)^\vee$ is torsion. Since the module $H^1_{\FF^*}\left(F,\left(X\otimes\LL/\wp\right)^*\right)^\vee$ is finite, it follows from Lemma \ref{lem:comparefirststep} that $H^1_{\FF}\left(F,X\otimes\LL/\wp\right)$ has $\LL/\wp$-rank one. Furthermore, as we have a natural injection 
$$H^1_{\FF}\left(F,X\right)/\wp H^1_{\FF}\left(F,X\right)\hookrightarrow H^1_{\FF}\left(F,X\otimes\LL/\wp\right)\,,$$
it follows by Nakayama's lemma that the $\LL_\cyc$-module $H^1_{\FF}\left(F,X\right)$ is cyclic. On the other hand, $\kappa_1$ is a non-zero element of the $\LL_\cyc$-torsion free module $H^1_{\FF}(F,X)$, it follows that $H^1_{\FF}(F,X)$ has positive $\LL_\cyc$-rank. This concludes the proof of (i) when $R=\LL_\cyc$.

We next sketch a proof of (ii) when $R=\LL_\cyc$. Fix a pseudo-isomorphism 
$$H^1_{\FF^*}(F,X^*)^\vee \lra \bigoplus_{i}\Lambda_\cyc/\frak{P}^{m_i} \oplus\left(\bigoplus_j\Lambda_\cyc/f_j\Lambda_\cyc\right)$$
where $\frak{P}$ is any height one prime dividing ${\textup{char}}\left(H^1_{\FF^*}(F,X^*)^\vee\right)$ and where each $f_j$ is prime to $\frak{P}$. We content to prove  that
\be\label{eqn:sufficientinequality}
\sum_{i}m_i\leq \textup{ord}_\frak{P}\, \textup{char}\left(H^1_{\FF}(F,X)/\LL_\cyc\cdot\kappa_1\right)\,,
\ee
from which follows (ii). We will assume that $p\notin \frak{P}$ and therefore $\frak{P}$ is generated by a distinguished polynomial $P \in \LL_\cyc=\ooo[[\gamma_\cyc-1]]$. For a general height one prime $\frak{Q}$ of $\LL_\cyc$\,, let $S_\frak{Q}$ denote the integral closure of $\LL_\cyc/\frak{Q}$. Note that $[S_\frak{Q}:\LL_\cyc/\frak{Q}]$ is finite. Set $\frak{P}_N=(P+\pi^N)$, where $N$ is a positive integer (chosen sufficiently large to ensure that $\frak{P}_N$ is a prime ideal). Write $X_N=(X\otimes\LL_\cyc/\frak{P})\otimes S_{\frak{P}_N}$. It follows from our assumption (\ref{eqn:assnotrivxhi}) that we have injections 
$$\iota: H^1(G_\Sigma,X\otimes\LL_\cyc/\frak{P}_N) \hookrightarrow H^1(G_\Sigma,X_N) \hbox{\,\,\, and, }$$
$$\iota_p: H^1(F_p,X\otimes\LL_\cyc/\frak{P}_N) \hookrightarrow H^1(F_p,X_N) $$
with finite cokernels (whose size depend only on $[S_\frak{P}:\LL_\cyc/\frak{P}]$). Define the Selmer structure $\FF$ on $X_N$ by setting 
$$H^1_\FF(F_\lambda,X_N)=\ker\left(H^1(F_\lambda,X_N)\lra H^1(F_\lambda^\textup{ur},X_N\otimes\QQ_p)\right)$$
for $\lambda\nmid p$ (this would be the local condition denoted by $H^1_{\FFc}(F_\lambda,X_N)$ in the notation of \cite{mr02}) and defining $H^1_\FF(F_p,X_N)$ as the $S_{\frak{P}_N}$-saturation of $\iota_p\left(H^1_{\FF}(F_p,X\otimes\LL_\cyc/\frak{P}_N)\right)$\,.

As explained in the proof of Theorem 5.3.10 of \cite{mr02}, for every sufficiently large positive integer $N$ we have:

\begin{enumerate}
\item $\LL_\cyc/\frak{P}_N\cong \LL_\cyc/\frak{P}$\,, 
\item The image $\kappa_1^{\frak{P}_N}\in H^1_{\FF}(F,X_N)$ of $\kappa_1$ is non-zero. 
\item $\textup{coker}\left(H^1_\FF(F,X)/\frak{P}_NH^1_\FF(F,X) \hookrightarrow H^1_{\FF}(F,X_N)\right)$ is finite with order bounded by a constant independent of $N$,
\item $\frak{P}_N$ is prime to $f_j$ for every $j$.
\end{enumerate}  
 Only the verification of (3) requires a slight enhancement of \cite[Proposition 5.3.14]{mr02} (so as to apply with the Selmer structure $\FF$ in place of the Selmer structure $\FF_\LL$ in loc.cit.). This shows, proceeding as in the proof of Theorem 5.3.10 of loc. cit. (essentially, by only making use of the Kolyvagin system $\pmb{\kappa}^{\frak{P}_N}\in  \overline{\mathbf{KS}}(X_N,\FF,\PP)$ over the one-dimensional ring ring $S_{\frak{P}_N}$) that 
 \begin{align*}Nr\sum_{i}m_i +O(1)&=\textup{length}_{\ZZ_p}\,H^1_{\FF^*}(F,X^*)[\frak{P}_N]=\textup{length}_{\ZZ_p}\,H^1_{\FF^*}\left(F,\left(X/\frak{P}_N\right)^*\right)
 \\&\leq \textup{length}_{\ZZ_p}\,H^1_{\FF^*}(F,X_N^*) +O(1)\\
 &\leq \textup{length}_{\ZZ_p}\,\left(H^1_{\FF}(F,X_N)/S_{\frak{P}_N}\cdot\kappa_1^{\frak{P}_N}\right) +O(1) \\
 &=\textup{length}_{\ZZ_p}\,\left(\left(H^1_{\FF}(F,X)/\LL_\cyc\cdot\kappa_1 \right)\otimes \LL_\cyc/{\frak{P}_N}\right) +O(1) 
 \\&= Nr\,\textup{ord}_{\frak{P}}\, \textup{char}\left(H^1_{\FF}(F,X)/\LL_\cyc\cdot\kappa_1\right)+O(1)
 \end{align*}
where $r=\textup{rank}_{\ZZ_p}\,S_{\frak{P}}$.  (\ref{eqn:sufficientinequality}) now follows (for characteristic zero primes $\frak{P}$) taking $N$ sufficiently large in the inequality above. In case $p\in \frak{P}$, we proceed by considering the ideals $\frak{P}_N=(\pi+(\gamma_\cyc-1)^N)$ and conclude the proof.

When $R=\LL$, we may make use of the arguments of Ochiai in \cite[\S3]{ochiaideform} in order to reduce the assertions in (i) and (ii) to the case of a dimension-two regular ring. As details pertaining to this point  will soon be available (in much greater generality) as part of our forthcoming joint work with T. Ochiai, we indicate here only the key points. We follow the terminology of  \cite[\S3]{ochiaideform}. First of all, our argument above when $R=\LL_\cyc$ above shows that for all but finitely many \emph{linear elements} $l \in \LL$, we have (i) and (ii)  for the $\LL/(l)$-module $\TT\otimes \LL/(l)$. As the second step, one makes use of this input together with control theorems for the $\mathcal{F_L}$-Selmer groups (which are in fact easier than those relevant to considerations in loc. cit, due to the fact $\TT=T\otimes\LL$ is a rather simple Galois deformation) as well \cite[Proposition 3.6]{ochiaideform} (that characterizes the characteristic ideal of a torsion $\LL$-module $M$ in terms of the characteristic ideals of the quotients $M/lM$ as $\LL/(l)$-modules) to finish with the proof.
\end{proof}
\begin{cor}
\label{cor:localnonvanishingimpliesweakLeoforE}
Suppose  $\pmb{\kappa}\in \overline{\mathbf{KS}}(\TT_\cyc(E),\FF_{\mathbb{L}},\PP)$ is a Kolyvagin system with non-vanishing initial term $\kappa_1 \in H^1_{\FF_{\mathbb{L}}}(F,\TT_\cyc(E))$. Then $H^1_{\FF_{\textup{can}}^*}(F,\TT_\cyc(E)^*)^\vee$ is  a torsion $\LL_\cyc$-module and the weak Leopoldt conjecture for $E$ holds true.
\end{cor}
\begin{proof}
This follows from Theorem~\ref{thm:mainapplicationofrestrictedKS}(i) and the obvious injection
$$H^1_{\FF_{\textup{can}}^*}(F,\TT_\cyc(E)^*) \hookrightarrow H^1_{\FF_{\mathbb{L}}^*}(F,\TT_\cyc(E)^*) \,.$$
\end{proof}
\begin{prop}
\label{prop:usefulforequality}
Let $\pmb{\kappa}\in \overline{\mathbf{KS}}(\TT,\FF_{\mathbb{L}},\PP)$ be a Kolyvagin system with initial term $0\neq\kappa_1 \in H^1_{\FF_{\mathbb{L}}}(F,\TT)$ and let $\widetilde{\pmb{\kappa}}\in \overline{\mathbf{KS}}(\TT(E),\FF_{\mathbb{L}},\PP)$ with initial term $\textup{tw}(\kappa_1) \in H^1_{\FF_{\mathbb{L}}}(F,\TT(E))$. Suppose that 
$$\textup{char}\left(H^1_{\FF_{\mathbb{L}}^*}(F,\TT^*)^\vee\right)=\textup{char}\left(H^1_{\FF_{\mathbb{L}}}(F,\TT)/\LL\cdot\kappa_1\right)\,.$$
Then,
\begin{itemize}
\item[(i)] $\textup{char}\left(H^1_{\FF_{\mathbb{L}}^*}(F,\TT(E)^*)^\vee\right)=\textup{char}\left(H^1_{\FF_{\mathbb{L}}}(F,\TT(E))/\LL\cdot\textup{tw}\left(\kappa_1\right)\right)\,.$
\item[(ii)] The Kolyvagin system $\widetilde{\pmb{\kappa}}$ and its image $\pi_\cyc\left(\widetilde{\pmb{\kappa}}\right) \in \overline{\mathbf{KS}}(\TT_\cyc(E),\FF_{\mathbb{L}},\PP)$ are both primitive.
\item[(iii)] Let $\frak{K}_1\in H^1_{\FF_{\mathbb{L}}}(F,\TT_\cyc(E))$ be the initial term of the Kolyvagin system $\pi_\cyc\left(\widetilde{\pmb{\kappa}}\right)$. Then
$$\textup{char}\left(H^1_{\FF_{\mathbb{L}}^*}(F,\TT_\cyc(E)^*)^\vee\right)=\textup{char}\left(H^1_{\FF_{\mathbb{L}}}(F,\TT_\cyc(E))/\LL_\cyc\cdot\frak{K}_1\right)\,.$$
\end{itemize}
\end{prop}

\begin{proof}
(i) follows using a formal twisting argument, c.f. Lemma VI.1.2 and  Theorem VI.4.1 of \cite{r00}.

Let $\frak{g} \in   \overline{\mathbf{KS}}(\TT(E),\FF_{\mathbb{L}},\PP)$ be a generator and let $g_1$ be its initial term. Write $\widetilde{\pmb{\kappa}}=r\cdot\frak{g}$ (where $r\in \LL$) so that $\textup{tw}\left(\kappa_1\right)=r\cdot g_1$\,. It follows from (i) and Theorem~\ref{thm:mainapplicationofrestrictedKS}(ii) that 
\begin{align*}\textup{char}\left(H^1_{\FF_{\mathbb{L}}}(F,\TT(E))/\LL\cdot r g_1\right)&=\textup{char}\left(H^1_{\FF_{\mathbb{L}}}(F,\TT(E))/\LL\cdot\textup{tw}\left(\kappa_1\right)\right)\\
&\mid \textup{char}\left(H^1_{\FF_{\mathbb{L}}}(F,\TT(E))/\LL\cdot g_1\right)\end{align*}
which shows that $r \in \LL^\times$, proving the first assertion in (ii). It now follows from Theorem~\ref{thm:lamdaadicKS}(ii) that the image $\overline{\pmb{\kappa}} \in {\mathbf{KS}}(\overline{T}(E),\FF_{\mathbb{L}},\PP)$ of $\widetilde{\pmb{\kappa}}$ is non-zero 
and the second assertion in (ii) holds true by the commutative diagram
$$\xymatrix@C=.05cm{\overline{\mathbf{KS}}(\TT(E),\FF_{\mathbb{L}},\PP)\ar[rd] \ar[rr]^{\pi_\cyc}&&\overline{\mathbf{KS}}(\TT_\cyc(E),\FF_{\mathbb{L}},\PP)\ar[ld]\\
&{\mathbf{KS}}(\overline{T}(E),\FF_{\mathbb{L}},\PP)&
}$$
and by Theorem~\ref{thm:lamdaadicKS}(ii). The final portion of the Proposition follows now from (ii) and Theorem~\ref{thm:mainapplicationofrestrictedKS}(iii).
\end{proof}

\subsection{Rubin-Stark  $\al$-restricted Kolyvagin systems}
\label{sec:KSRS}
The purpose of this section is to construct the $\al$-restricted Kolyvagin systems $\TT$ (that we proved t exist unconditionally in the previous section) out of the Rubin-Stark elements. In order to do so, we will first construct an Euler system of rank one (namely, an Euler system in the sense of \cite{r00}) that enjoys additional local properties at $p$. We will then apply Kolyvagin's descent on this Euler system.

\begin{define}
\label{def:alrestrictedES}
\begin{itemize} 
\item[(i)]For $X=T$ or $T(E)$, let $\textup{ES}(X)=\textup{ES}(X,\frak{E})$ denote the collection of Euler systems for $X$ in the sense of \cite[\S2]{r00}.
    \item[(ii)]Let $\frak{L} \subset \mathcal{V}_{K(\frak{E})}^+$ be a $\ooo[[\frak{G}(K(\frak{E}))]]$-direct summand as in Definition~\ref{def:modifiedlocalcondition}. An Euler system $\textbf{c}=\{c_\mathcal{K}\} \in \textup{ES}(T)$ is called an \emph{$\frak{L}$-restricted Euler system} if
$$\textup{loc}_p(c_\mathcal{K}) \in \mathcal{V}_\mathcal{K}^-\oplus \frak{l}_\mathcal{K}$$
for every $\mathcal{K}\in \frak{E}$. The module of $\frak{L}$-restricted Euler systems for $T$ is denoted by $\textup{ES}_{\frak{L}}(T)$. We similarly define the module of $\mathbb{L}$-restricted Euler systems $\textup{ES}_{\mathbb{L}}(T(E))$ for $T(E)$.
    \end{itemize}
\end{define}

\begin{thm}[Mazur-Rubin]
\label{thm:ESKSmain} For $X=T$ or $T(E)$, there is a canonical map
$$\ES(X) \lra \overline{\KS}(X\otimes\LL,\FFc,\PP),$$
with the property that if $\mathbf{c}=\{c_\KKK\}_{_{\KKK\in\frak{E}}} \in \textup{ES}(X)$ maps to $\pmb{\kappa} \in \overline{\KS}(X\otimes\LL,\FFc,\PP)$ then
$$\kappa_1=\{c_{M}\} \in \varprojlim_{M} H^1(M,X)=H^1(F,X\otimes\LL),$$
where the inverse limit is over the finite sub-extensions $M$ of $F_{\infty}/F$.
\end{thm}

For any field $\KKK\in\frak{E}$, recall that $\Delta_\KKK:=\Gal(\KKK/F)$ and write $\delta_{\KKK}=|\Delta_\KKK|$. Let $\varPhi=\{\varphi_\KKK\}$ be any element of $\varprojlim_{\KKK \in \frak{E}} \wedge^{r-1}\,\hbox{Hom}_{\ooo[\Delta_\KKK]}\left(H^1(\KKK,T), \ooo[\Delta_K]\right).$ As explained in \cite[\S1.2]{ru96}, there is a natural map $$\wedge^{r-1}\,\hbox{Hom}_{\ooo[\Delta_\KKK]}\left(H^1(\KKK,T), \ooo[\Delta_K]\right)\lra \hbox{Hom}_{\ooo[\Delta_\KKK]}\left(\wedge^g H^1(\KKK,T), H^1(\KKK,T)\right).$$
We denote the image of $\varphi_\KKK$ under this map  still by $\varphi_\KKK$. Given a collection $\varPhi=\{\varphi_\KKK\}$ as above, we obtain an element   $\varphi_{\KKK}(\varepsilon_{\KKK}^\chi) \in H^1(\KKK,T)$ by the defining (integrality) property of the elements $\varepsilon_{\KKK}^\chi \in \frac{1}{\delta_\KKK} \wedge^g H^1(\KKK,T)$. In other words, the denominators $\delta_\KKK$ disappear once we apply the homomorphisms $\varphi_\KKK$.

\begin{thm}[Perrin-Riou, Rubin]
\label{thm:pr-rubin}
 $c_{\varPhi}^\chi:=\{\varphi_\KKK(\varepsilon_{\KKK}^\chi)\}\in \textup{ES}(T)$.
\end{thm}
\begin{proof}This is proved in \cite[\S1.2.3]{pr-es}; see also \cite[\S6]{ru96}.
\end{proof}

Localization followed by projection to $\mathcal{V}^+_\KKK$ induces a canonical homomorphism
\be\label{eqn:localizationhoms}{
\varprojlim_{\KKK \in \frak{E}} \wedge^{r-1}\,\hbox{Hom}_{\ooo[\Delta_\KKK]}\left(\mathcal{V}_\KKK^+, \ooo[\Delta_\KKK]\right) \lra
\varprojlim_{\KKK \in \frak{E}} \wedge^{r-1}\,\hbox{Hom}_{\ooo[\Delta_\KKK]}\left(H^1(\KKK,T_\chi), \ooo[\Delta_\KKK]\right)
}
\ee
If $\varPhi$ is an element of the left side of (\ref{eqn:localizationhoms}), its image under this homomorphism will still be denoted by the same symbol.

The following theorem tell us how to obtain $\frak{L}$-restricted Euler systems (and $\al$-restricted Kolyvagin systems) starting off with the Rubin-Stark Euler system $\mathcal{C}^{(g)}_{\textup{R-S}}$ of rank $g$.
\begin{thm} $\,$ Recall the quotients $\mathcal{Q}=H^1(F_p,\TT)/\mathcal{V}$ and $\mathcal{Q}_\cyc=H^1(F_p,\TT_\cyc)/\mathcal{V}_\cyc$\,.
\label{thm:ellrestrictedESmain}
\begin{enumerate}
\item[(i)] There exists an element
$\Psi=\{\psi_{\KKK}\} \in \varprojlim_{\KKK \in \frak{E}} \wedge^{r-1}\,\textup{Hom}_{\ooo[\Delta_\KKK]}\left(\mathcal{V}_K^+, \ooo[\Delta_\KKK]\right)$ such that $\psi_\KKK$ maps $\wedge^g \mathcal{V}_\KKK^+$ isomorphically onto $\al_\KKK$ (likewise, $\wedge^g\, \mathcal{Q}_\cyc$ to $\al_\cyc$ and $\wedge^g\, \mathcal{Q}$ to $\al$).
\item[(ii)] For $\Psi$ as in \textup{(i)}, $\textbf{c}^{\chi}_{\Psi}:=\{\psi_{\KKK}(\varepsilon_{\KKK}^\chi)\} \in \textup{ES}_{\frak{L}}(T).$
\item[(iii)] Let $\pmb{\kappa}^{\textup{R-S}} \in \overline{\KS}(\TT,\FFc,\PP)$  be the image of $\textbf{c}^{\chi}_{\Psi}$ under the Euler systems to Kolyvagin systems map of Theorem~\ref{thm:ESKSmain}. Then $\pmb{\kappa}^{\textup{R-S}} \in \overline{\KS}(\TT,\FF_{\mathcal{L}},\PP)$, i.e., $\pmb{\kappa}^{\textup{R-S}}$ is an $\al$-restricted Kolyvagin system.
\end{enumerate}
\end{thm}
\begin{proof}
(i) may be proved mimicking the arguments of \cite[\S3.3.1]{kbbesrankr}. To prove (ii) and (iii), one makes use of Proposition~\ref{prop:localfull} and adapts (completely formally) the proof of Theorem 3.25 of loc.cit.
\end{proof}
Let $c_{F,\Psi}^\chi:=\psi_F(\varepsilon_F^\chi)\in H^1_{\FF_{\frak{L}}}(F,T)$ denote the initial term of the $\frak{L}$-restricted Euler system $\textbf{c}^{\chi}_{\Psi}$. Similarly, define $c_{F_\cyc,\Psi}^\chi=\{c_{M,\Psi}^{\chi}\} \in \varprojlim_{M\subset F_\cyc} H^1(M,T)=H^1(F,\TT_\cyc)$ and $c_{F_\infty,\Psi}^\chi=\{c_{M,\Psi}^{\chi}\} \in \varprojlim_{M\subset F_\infty} H^1(M,T)=H^1(F,\TT)$, where the inverse limit is over the finite sub-extensions $M$ of $F_{\infty}/F$.
\begin{prop}
\label{prop:KSwithnonzeroinitialterm}
$c_{F,\Psi}^\chi\neq 0$.
\end{prop}
\begin{proof}
This follows from the proof of Proposition 6.6 in \cite{ru96} since we assumed Leopoldt's conjecture.
\end{proof}
\begin{rem}
\label{rem:ESKStwist}
Definition VI.3.1 \cite{r00} equips us with a \emph{twisting morphism}  
$\textup{ES}(T)\ra \textup{ES}(T(E))$, which then evidently restricts to a map $\textup{ES}_{\mathbb{L}}(T)\ra \textup{ES}_{\mathbb{L}}(T(E))$ on the $\mathbb{L}$-restricted Euler systems. Let $\textbf{c}_{\Psi}(E)\in \textup{ES}_{\mathbb{L}}(T(E))$ denote the image of $\textbf{c}^{\chi}_{\Psi}$. Then the image $\pmb{\kappa}^{\textup{R-S}}(E)$ of $\textbf{c}_{\Psi}(E)$ under the map of Theorem~\ref{thm:ESKSmain} (applied with $X=T(E)$) lies in $\overline{\KS}(\TT(E),\FF_{\mathbb{L}},\PP)$. The initial term ${\kappa}^E_1 \in H^1_{\FF_\mathbb{L}}(F,\TT(E))$ of the Kolyvagin system $\pmb{\kappa}^{\textup{R-S}}(E)$ may be explicitly described: ${\kappa}^E_1=\textup{tw}\left(c_{F_\infty,\Psi}^\chi\right)\,.$ In particular, it follows from Proposition~\ref{prop:KSwithnonzeroinitialterm} that ${\kappa}^E_1\neq 0$.
\end{rem}
\section{Gras' conjecture and CM main conjectures over $F$}
\label{sec:gras}
Although our sights are set ultimately on the arithmetic of CM elliptic curves defined over $F^+$, we present  the following results for $\mathbb{G}_m$, first of which may be thought of a generalization of Gras' conjecture and second and third as the one- and two-variable main conjectures for the CM field $F$. We will later use these results to promote all inequalities we shall obtain using the Rubin-Stark Euler/Kolyvagin systems for $T(E)$ of Remark~\ref{rem:ESKStwist} into equalities.

We assume until the end of this article that the following hypothesis on $S$ holds true (recall as well that we assume the truth of the Rubin-Stark conjectures and Leopoldt's conjecture for $L$):

$(\textbf{H.S})$ The set $S$ that appears in the definition of Rubin-Stark elements (see the start of Section \ref{sec:RS}) contains no non-archimedean prime of $F$ that splits in $L/F$.
\begin{define}
\label{def:limitofidealclassgroups}
Let $\mathcal{A}_\cyc^\chi=\varprojlim_{M\subset LF_\cyc} \textup{Cl}(M)^\chi$ and similarly, $\mathcal{A}_\infty^\chi=\varprojlim_{M\subset LF_\infty} \textup{Cl}(M)^\chi$\,. We have the identifications (by class field theory) 
$$\mathcal{A}_\cyc^\chi=H^1_{\FFc^*}(F,\TT_\cyc^*)^\vee \hbox{\,\,\,\ and \,\,\,} \mathcal{A}_\infty^\chi=H^1_{\FFc^*}(F,\TT^*)^\vee\,.$$
\end{define}
\begin{thm}$\,$
\label{thm:gras1}
\begin{itemize}
\item[(i)] $\#\, \textup{Cl}(L)^\chi = [\wedge^g\, \mathcal{O}_{L}^{\times,\chi}\,:\,\frak{O}\cdot \varepsilon_{F}^\chi]$ and the Rubin-Stark $\al$-restricted Kolyvagin system $\pmb{\kappa}^{\textup{R-S}} \in \overline{\KS}(\TT,\FF_\al,\PP)$ is primitive.
\end{itemize}
If in addition the strong Rubin-Stark conjecture holds true, then
\begin{itemize}
\item[(ii)] $\textup{char}\left(\mathcal{A}_\cyc^\chi\right)=\textup{char}\left(\wedge^g H^1_{\FFc}(F,\TT_\cyc)/\LL_\cyc\cdot \frak{S}_\cyc\right)$\,.
\item[(iii)] $\textup{char}\left(\mathcal{A}_\infty^\chi\right)=\textup{char}\left(\wedge^g H^1_{\FFc}(F,\TT)/\LL\cdot \frak{S}_\infty\right)$\,.
\end{itemize}
\end{thm}
\begin{proof} 
It follows from Theorem~\ref{thm:mainapplicationofrestrictedKS}(i) and Proposition~\ref{prop:KSwithnonzeroinitialterm} that $H^1_{\FF_\al^*}(F,T^*)$ is finite, the $\LL_\cyc$-module $H^1_{\FF_\al^*}(F,\TT_\cyc^*)^\vee$ and the $\LL$-module $H^1_{\FF_\al^*}(F,\TT^*)^\vee$ are torsion. Furthermore, by Theorem~\ref{thm:mainapplicationofrestrictedKS}(ii) we have
$$
\label{eqn:grasstep1}
\textup{Fitt}(H^1_{\FF_\al^*}(F,T^*)^\vee)\mid \textup{Fitt}(H^1_{\FF_\al}(F,T)/\ooo\cdot c_{F,\Psi}^\chi)\,,
$$
\be
\label{eqn:cyclomainconjstep1}
\textup{char}\left(H^1_{\FF_\al^*}(F,\TT_\cyc^*)^\vee\right)\mid \textup{char}\left(H^1_{\FF_\al}(F,\TT_\cyc)/R\cdot c_{F_\cyc,\Psi}^\chi\right)\,,
\ee
$$
\label{eqn:fullmainconjstep1}
\textup{char}\left(H^1_{\FF_\al^*}(F,\TT^*)^\vee\right)\mid \textup{char}\left(H^1_{\FF_\al}(F,\TT)/\LL\cdot c_{F_\infty,\Psi}^\chi\right)\,.
$$
It is these divisibilities we shall upgrade to equalities (and conclude with the proof of the theorem) with the aid of an analytic class number formula. In order to save space, we shall do this simultaneously. To that end, let $R$ denote any of the coefficient rings $\ooo, \LL_\cyc$ or $\LL$. Correspondingly, let $X$ stand for one of the representations $T, \TT_\cyc$ or $\TT$\,; $V$ for one of the submodules $\mathcal{V}_F^+, \mathcal{V}_\cyc$, or $\mathcal{V}$ (of $H^1(F_p,X)$)\,; $D$ for one of the $R$-lines $\frak{l}, \al_\cyc$ or $\al$\,; $c$ for one of the elements $c_{F,\Psi}^\chi, c_{F_\cyc,\Psi}^\chi$ or $c_{F_\infty,\Psi}^\chi$ and $\varepsilon$ for $\varepsilon_{F}^\chi$ (when we assume the strong Rubin-Stark conjecture, for one of $\frak{S}_\cyc$ and $\frak{S}_\infty$ as well). Let $Q=H^1(F_p,X)/V$, a free $R$-module of rank $g$. Define the map $\locu$ to be the compositum of the maps 
$$\textup{loc}_{/V}: H^1_{\FFc}(F,X)\lra H^1(F_p,X) \lra Q\,.$$
Note that this map is injective by our choice of $U$. By slight abuse, we denote the isomorphic image of $D$ inside $Q$ also by $D$. Note with this convention that the map $\locu$ induces an injection $\locu:\,H^1_{\FF_\al}(F,X)\ra D\,.$ Henceforth, whenever the element $\varepsilon$ is used with a coefficient ring $R$ other than $\ooo$, we implicitly assume the strong Rubin-Stark conjecture. When $R=\ooo$, we mean by the characteristic ideal of a torsion $R$-module its initial Fitting ideal.

As we have indicated in the statement of Theorem~\ref{thm:ellrestrictedESmain}, $\Psi$ induces an isomorphism $\Psi:\wedge^g\,Q \ra D$ and furthermore verifies that $\locu(c)=\Psi(\locu(\varepsilon))$ (in fact by its very choice). We therefore have
\be\label{eqn:prf1}R\cdot\locu(c)=\textup{Fitt}\left(\wedge^g Q/R\cdot \locu(\varepsilon)\right) D=\textup{char}\left(\wedge^g Q/R\cdot \locu(\varepsilon)\right) D\ee
Furthermore, the following sequences are exact:
$$0\lra H^1_{\FF_{\al}}(F,X) \lra H^1_{\FFc}(F,X)\stackrel{\locu}{\lra} Q/D$$
$$0\lra H^1_{\FFc^*}(F,X^*) \lra H^1_{\FF_\al^*}(F,X^*)\stackrel{\locu^*}{\lra} H^1_{\FF_\al^*}(F_p,X^*)/H^1_{\FFc^*}(F_p,X^*),$$
Global duality states that the images of $\locu$ and $\locu^*$ are orthogonal complements. Hence
\be\label{eqn:prf2} \charr\left(\frac{H^1_{\FF_\al^*}(F,X^*)^\vee}{H^1_{\FFc^*}(F,X^*)^\vee}\right)=\charr
(\textup{coker}(\locu))=\charr\left(\frac{Q}{D+\locu(H^1_{\FFc}(F,T))}\right)\ee
Observe further that
\begin{align*}\frac{Q}{D+\locu(H^1_{\FFc}(F,T))}&\cong \frac{{Q}/{\locu(H^1_{\FFc}(F,T))}}{(D+\locu(H^1_{\FFc}(F,T)))/{\locu(H^1_{\FFc}(F,T))}}\\
&\cong \frac{{Q}/{\locu(H^1_{\FFc}(F,T))}}{D/{\left(\locu(H^1_{\FFc}(F,T))\cap D\right)}}\cong \frac{{Q}/{\locu(H^1_{\FFc}(F,T))}}{D/{\locu(H^1_{\FF_\al}(F,T))}}\,\,.
\end{align*}
This together with (\ref{eqn:prf2}) and (\ref{eqn:cyclomainconjstep1}) proves that
\begin{align}\notag\charr(H^1_{\FFc^*}(F,X^*)^\vee)=&\,\,\charr (H^1_{\FF_\al^*}(F,X^*)^\vee)\cdot\frac{\charr\left(D/{\locu(H^1_{\FF_\al}(F,X))}\right)}{\charr\left(Q/\locu(H^1_{\FFc}(F,X))\right)}\\
\label{eqn:prf3}\Big{|}& \,\,\charr\left(H^1_{\FF_{\al}}(F,T)/R\cdot c\right)\cdot\frac{\charr\left(D/{\locu(H^1_{\FF_\al}(F,X))}\right)}{\charr\left(Q/\locu(H^1_{\FFc}(F,X))\right)}\\
\notag = &\,\,\frac{\charr\left(D/R\cdot\locu(c)\right)}{\charr\left(Q/\locu(H^1_{\FFc}(F,X))\right]} \,\,\,,
\end{align}
where the final equality is because $\locu$ is injective. (\ref{eqn:prf1}) shows further that
\begin{align}
\label{eqn:usefulconversionR}
\notag\charr(H^1_{\FFc^*}(F,X^*)^\vee)&\mid \frac{\charr\left(\wedge^g\,Q/R\cdot\locu(\varepsilon)\right)}{{\charr\left(Q/\locu(H^1_{\FFc}(F,X))\right)}}\\
\notag&=\frac{\charr\left(\wedge^g\,Q/R\cdot\locu(\varepsilon)\right)}{{\charr\left(\wedge^g\,Q/\wedge^g\,\locu(H^1_{\FFc}(F,X))\right)}}\\
&=\charr\left(\wedge^g\,H^1_{\FFc}(F,X)/R\cdot\varepsilon\right).
\end{align}
This concludes when $R=\ooo$ that $\#\, \textup{Cl}(L)^\chi = [\wedge^g\, \mathcal{O}_{L}^{\times,\chi}\,:\,\frak{O}\cdot \varepsilon_{F}^\chi]$. Choosing the auxiliary set of primes $\mathcal{T}$ that appears in the definition of Rubin-Stark elements carefully (as in \cite[\S 2.1]{kbbstark}, see also the discussion preceding Theorem 3.11 in loc.cit.), one may use the analytic class number formula (together with an inclusion-exclusion argument) for all fields between $L$ and $F$ to convert the inequality of Theorem~\ref{thm:gras1}(i) into an equality, concluding the proof of the first assertion in (i). See \cite[\S5]{ru92}, \cite[Corollary 5.4]{ru96} and \cite[\S4.2]{popescu} for details. Tracing back the inequalities above, we see that we in fact have an equality in the divisibility 
$$\textup{Fitt}(H^1_{\FF_\al^*}(F,T^*)^\vee)\mid \textup{Fitt}(H^1_{\FF_\al}(F,T)/\ooo\cdot c_{F,\Psi}^\chi)$$
of (\ref{eqn:cyclomainconjstep1}) and it follows from Theorem~\ref{thm:mainapplicationofrestrictedKS}(iii) that the Kolyvagin system $\pmb{\kappa}^{\textup{R-S}}(T) \in \overline{\KS}(T,\FF_\al,\PP)$ (which is the image of $\pmb{\kappa}^{\textup{R-S}}$) is primitive. The second assertion in (i) now follows from Theorem~\ref{thm:lamdaadicKS}(ii).

Theorem~\ref{thm:lamdaadicKS}(ii) shows that the image $\pmb{\kappa}^{\textup{R-S}}(\TT_\cyc) \in \overline{\KS}(\TT_\cyc,\FF_\al,\PP)$ of $\pmb{\kappa}^{\textup{R-S}}$ is primitive as well. Hence we have equality in the divisibility
$$\textup{char}\left(H^1_{\FF_\al^*}(F,\TT_\cyc^*)^\vee\right)\mid \textup{char}\left(H^1_{\FF_\al}(F,\TT_\cyc)/R\cdot c_{F_\cyc,\Psi}^\chi\right)$$
of (\ref{eqn:cyclomainconjstep1}) and therefore, also in (\ref{eqn:usefulconversionR}) when $R=\LL_\cyc$. This is exactly the statement of (ii). 

When $R=\LL$, the divisibility (\ref{eqn:usefulconversionR}) reads 
\be\label{eqn:twovarmainconjdiv}\charr(\mathcal{A}_\infty^\chi)\mid \charr\left(\wedge^g H^1_{\FFc}(F,\TT)/\LL\cdot\frak{S}_\infty\right)\,.\ee
Let $\pi_\cyc:\LL\twoheadrightarrow\LL_\cyc$ denote the obvious projection. We will check below in Lemma~\ref{lem:reductionstep1} that 
$$\pi_\cyc(\charr_\LL(\mathcal{A}_\infty^\chi))=\charr_{\LL_\cyc}(\mathcal{A}_\cyc^\chi) \neq 0$$ and in Lemma~\ref{lem:reductionstep2} that 
$$\pi_\cyc(\charr\left(\wedge^g H^1_{\FFc}(F,\TT)/\LL\cdot\frak{S}_\infty\right))=\charr\left(\wedge^g H^1_{\FFc}(F,\TT_\cyc)/\LL\cdot\frak{S}_\cyc\right)\,.$$
All this shows (along with (\ref{eqn:twovarmainconjdiv}) and (ii)) that there are generators $f$ (resp., $g$) of $\charr_\LL(\mathcal{A}_\infty^\chi)$ (resp., of $\charr\left(\wedge^g H^1_{\FFc}(F,\TT)/\LL\cdot\varepsilon_{F_\infty}^\chi\right)$) such that  $f \notin \ker\pi_\cyc$, $f-g \in \ker\pi_\cyc$ and $f$ divides $g$. We conclude using Lemma~\ref{lem:equalitymodAimpliesequality} that $g/f \in \LL^\times$, concluding the proof of (iii).
\end{proof}
\begin{lemma}
\label{lem:equalitymodAimpliesequality}
Suppose $f,g \in \LL$ are such that $f \mid g$, $f - g \in \ker\pi_\cyc$ and $f\notin \ker\pi_\cyc$. Then $g/f\in \LL^\times$.
\end{lemma}
\begin{proof}
Write $g=f\cdot h$ with $h\in \LL$, so that $f-g=f(1-h) \in \ker\pi_\cyc$. Since $f\notin \ker\pi_\cyc$, it follows that $1-h \in \ker\pi_\cyc \subset \frak{m}_{\LL}$, where $\frak{m}_{\LL}$ is the maximal ideal. Hence $h$ is indeed a unit.
\end{proof}
\begin{lemma}
\label{lem:reductionstep1}
$\pi_\cyc(\charr_\LL(\mathcal{A}_\infty^\chi))=\charr_{\LL_\cyc}(\mathcal{A}_\cyc^\chi) \neq 0\,.$
\end{lemma}
\begin{proof}
Observe that $\ker\pi_\cyc=(\gamma_*-1)\LL$ where $\gamma_*\in \Gamma$ is any lift of a topological generator of $\Gamma/\Gamma_\cyc$. By the control theorem, 
$$\mathcal{A}_\infty^\chi/(\gamma^*-1)=H^1_{\FFc^*}(F,\TT^*)^\vee/(\gamma_*-1)\cong H^1_{\FFc^*}(F,\TT_\cyc^*)^\vee=\mathcal{A}_\cyc^\chi\,.$$
As the $\LL_\cyc$-module $\mathcal{A}_\cyc^\chi$ is torsion, it follows from Lemme 4 of \cite[\S1.1.3]{PR84memoirs} that $\charr_\LL(\mathcal{A}_\infty^\chi)$ is prime to 
$(\gamma_*-1)$ and $\pi_\cyc\left(\charr_\LL(\mathcal{A}_\infty^\chi)\right)=\charr_{\LL_\cyc}(\mathcal{A}_\cyc^\chi)$, as desired.
\end{proof}
\begin{lemma}
\label{lem:reductionstep2}
$\pi_\cyc(\charr\left(\wedge^g H^1_{\FFc}(F,\TT)/\LL\cdot\frak{S}_\infty\right))=\charr\left(\wedge^g H^1_{\FFc}(F,\TT_\cyc)/\LL\cdot\frak{S}_\cyc\right)\,.$
\end{lemma}
\begin{proof}
It suffices to verify that the $\LL_\cyc$-module 
$$\textup{coker}(H^1_{\FFc}(F,\TT)\stackrel{\pi_\cyc}{\lra}H^1_{\FFc}(F,\TT_\cyc))$$ 
is pseudo-null. It follows from Nekov\'a\v{r}'s control theorem (as utilized in Remark~\ref{rem:freenessconceptually}) $\textup{coker}(\pi_\cyc)\cong \mathcal{A}_\infty^\chi[\gamma_*-1]$. This module is pseudo-null by Lemme 4 of \cite[\S1.1.3]{PR84memoirs}.
\end{proof}


\subsection{A two-variable CM main conjecture}
\label{subsecCMmainconj}
The goal in this section is to prove somewhat less precise version of Theorem~\ref{thm:gras1}(iii)   assuming only the Rubin-Stark conjecture (but not the strong Rubin-Stark conjecture). Hypotheses from the previous section is in effect. Recall the map $\locu$ defined as in the proof of Theorem~\ref{thm:gras1}.
\begin{thm}
\label{thm:mainconjforTchi} 
We have
\be\label{eqn:divinmainconjTchi}
\textup{char} \left(\frak{\al}\big{/}\LL\cdot  \locu\left(c_{F_\infty,\Psi}^\chi \right)\right) \subseteq \textup{char}\left(H^1_{\FF^*_{\frak{tr}}}(F,\TT^*)^\vee\right).
\ee 
In particular, the module $H^1_{\FF^*_{\frak{tr}}}(F,\TT^*)$ is $\LL$-cotorsion. Furthermore, the containment in (\ref{eqn:divinmainconjTchi}) may be promoted to an equality if the Strong Rubin-Stark conjecture holds true.
\end{thm}
\begin{proof}
The first part may be deduced from from Proposition~\ref{prop:4termexact} (used with $\FF=\FF_{\frak{tr}}$ and $\mathcal{G}=\FF_{\mathcal{L}}$) and Theorem~\ref{thm:mainapplicationofrestrictedKS}(ii) (used with $\FF_{\mathcal{L}}$).  The second assertion follows from Proposition~\ref{prop:KSwithnonzeroinitialterm}, the fact that $\locu$ is injective (see the proof of Theorem~\ref{thm:gras1}) and the  containment (\ref{eqn:divinmainconjTchi}). Finally, the third portion follows from the proof of Theorem~\ref{thm:gras1}.
\end{proof}
\begin{rem}
\label{rem:RSinwedge0}
Let $\KKK$ be any field contained in the collection $\frak{C}$. The defining property of the Rubin-Stark elements and Example~\ref{example} show that $\locu(\varepsilon_\mathcal{K}^\chi) \in \overline{\wedge^g\,Q_\KKK}$ 
where $Q_\KKK:=H^1(\KKK_p,T)/\mathcal{V}_\KKK^-$ and the exterior product is taken in the category of $\ooo[\textup{Gal}(\KKK/F)]$-modules. We will simply write $\locu(\varepsilon_\mathcal{K}^\chi)$ in place of $j^{-1}(\locu(\varepsilon_\mathcal{K}^\chi)) \in  \wedge^g\,Q_\KKK$. 
\end{rem}
\begin{define}
\label{def:rubinstarktowerlocal}
Recall the free-module $\mathcal{Q}=H^1(F_p,\TT)/\mathcal{V}$ of rank $g$.  Define
$$\locu(\varepsilon_{F_\infty}^\chi)=\{\locu(\varepsilon_M^\chi)\} \in \varprojlim \wedge^g\, Q_M=\wedge^g\varprojlim\, Q_M=\wedge^g \,\mathcal{Q}$$
to be the tower of Rubin-Stark elements along $F_\infty$. Here the inverse limit is taken over all finite subextensions  of $F_\infty/F$ and the second equality holds thanks to the fact that each module $Q_M$ is free as an $\ooo[\textup{Gal}(M/F)]$-module and the transition maps $Q_M \ra Q_{M^\prime}$ ($F\subset M^\prime\subset M \subset F_\infty$) are all surjective (because all the maps $\mathcal{Q}\ra Q_M$ are). 
\end{define}
\begin{thm}
\label{thm:mainconjdivisibility}
The ideal $\textup{char}\left(H^1_{\FF^*_{\frak{tr}}}(F,\TT^*)^\vee\right) $ divides $\textup{char} \left(\wedge^g \,\mathcal{Q}/\LL\cdot  \locu\left(\varepsilon_{F_\infty}^\chi \right)\right)\,,$ with equality if we further assume the Strong Rubin-Stark conjecture.

\end{thm}
\begin{proof}
Thanks to our choice of $\Psi$ we have
$$\textup{char} \left(\frak{\al}/\LL\cdot  \textup{loc}_p\left(c_{F_\infty,\Psi}^\chi \right)\right)=\textup{char} \left(\wedge^g\, Q/\LL\cdot  \locu\left(\varepsilon_{F_\infty}^\chi \right)\right),$$
and the proof follows from Therorem~\ref{thm:mainconjforTchi}. \end{proof}
\begin{rem}
\label{rem:comparetoRubin}
The Iwasawa module $H^1_{\FF^*_{\frak{tr}}}(F,\TT^*)^\vee$ should be compared to the module $\hat{X}$ of \cite[\S11]{rubinmainconj} and  Theorems~\ref{thm:gras1} and \ref{thm:mainconjdivisibility} to Rubin's main conjecture~\cite[Theorem 4.1(ii)]{rubinmainconj}, generalized to the setting where the base field $F$ is now a general CM field.
\end{rem}
\section{The cyclotomic (supersingular) main conjecture for CM elliptic curves}
\label{subsec:cyclomain}
The goal of this section is to apply results from Section~\ref{Sec:KSforGmandE} to study the cyclotomic Iwasawa theory of a CM elliptic curve at a supersingular prime.

Recall that $F^{\textup{cyc}}\subset F_\infty$ denotes the cyclotomic $\ZZ_p$-extension of $F$ and $F_n$ its $n$th layer. Further notation from Section~\ref{subsubsec:prelim} is also still in effect. In particular, recall the characters $\rho$, $\langle \rho\rangle$ and $\omega_E$. Also until the end, the Dirichlet character $\chi$ is chosen to be $\omega_E$.

Throughout Section~\ref{subsec:cyclomain} \emph{we assume that $p$ splits completely in $F^+/\QQ$}. As before, let $T(E)=T_p(E)$ denote the $p$-adic Tate module of $E$. Let $S_p=\{\wp_1,\cdots,\wp_g\}$ denote the set of primes of $F^+$ lying above $p$. Note that each $\wp_i$ remains inert in the quadratic extension $F/F^+$. We denote the unique prime of $F$ above $\wp_i$ by $\frak{p}_i$. By a slight abuse, we denote the unique prime of $F_n$ (and of $F^{\textup{cyc}}$) above $\frak{p}_i$ by the same symbol $\frak{p}_i$.

\subsection{Preliminaries}
\label{subsec:prelimonCMtheory}
In this subsection we recall some classical results (due mostly to Coates and Wiles) in the Iwasawa theory of CM elliptic curves. We shall initially record them being faithful to the original notions and notation, and later in Remark~\ref{rem:compareXwithstrictSelmer} explain which objects we have introduced in the previous sections they correspond to.

Let $\frak{M}$ be the maximal abelian pro-$p$ extension of $\frak{F}:=F(E[p^\infty])$ unramified outside primes above $p$. Set $\frak{X}:=\textup{Gal}(\frak{M}/\frak{F})$ and $\LL_{\frak{F}}:=\ooo[[\textup{Gal}(\frak{F}/F)]]$. For any extension $M$ of $F$ (finite or infinite), consider the \emph{relaxed} Selmer group
$$\textup{Sel}_p^{\prime}(E/M)=\ker\left( H^1(M,E[p^\infty]) \lra \prod_{v\nmid p} \frac{H^1(M_v,E[p^\infty])}{E(M_v)\otimes\QQ_p/\ZZ_p} \right).$$
\begin{lemma}[Rubin]
\label{lem:relaxedistrue}
For any infinite extension $M_\infty$ of $F$ contained in $\frak{F}$,
$$\textup{Sel}_p^{\prime}(E/M_\infty)=\textup{Sel}_p(E/M_\infty).$$
\end{lemma}
\begin{proof}
This follows from \cite[Lemma 2.2]{rubincompositio85}.
\end{proof}
\begin{define}
\label{def:coinvariants}
Given a $\LL_{\frak{F}}$-module $Y$ and a continuous character $\psi:\textup{Gal}(\frak{F}/F)\ra\ooo^\times$, we define $Y(\psi):=Y\otimes\ooo_{\psi^{-1}}$ where $\ooo_{\psi^{-1}}$ is the cyclic $\ooo$-module on which $\textup{Gal}(\frak{F}/F)$ acts via $\psi^{-1}$.
\end{define}

\begin{define}
\label{def:coinvariants}
Given a $\LL_{\frak{F}}$-module $Y$, we define $Y^{\rho}_{\infty}:=Y(\rho^{-1})\otimes_{\LL_{\frak{F}}}\LL$  (resp., $Y^{\rho}_{\textup{cyc}}:=Y(\rho^{-1})\otimes_{\LL_{\frak{F}}}\LL_{\textup{cyc}}$), the $F_\infty$-coinvariants (resp., $F^{\textup{cyc}}$-coinvariants) of $Y(\rho^{-1})$.
\end{define}
\begin{lemma}
\label{lemma:RubinCW}
\begin{itemize}
\item[(i)] $\textup{Sel}_p^{\prime}(E/\frak{F})\cong\textup{Hom}_{\ooo}(\frak{X},E[p^\infty])$.
\item[(ii)] $\textup{Sel}_p(E/F_\infty)^\vee\cong\frak{X}_{\infty}^\rho$ and $\textup{Sel}_p(E/F^{\textup{cyc}})^\vee\cong\frak{X}_{\textup{cyc}}^\rho$.
\end{itemize}
\end{lemma}

\begin{proof}
Proof of (i) is essentially due to Coates and Wiles and follows from the criterion of N\'eron-Ogg-Shafarevich utilized as in the proof of \cite[Theorem 2]{coateswiles77}. Proposition 1.2 of \cite{rubincompositio85} shows (for $M_\infty=F^\cyc$ or $F_\infty$) that
$$\textup{Sel}_p^{\prime}(E/M_\infty)=\textup{Sel}_p^{\prime}(E/\frak{F})^{\textup{Gal}(\frak{F}/M_\infty)}\,.$$
(In fact, the case $M_\infty=F_\infty$ is a straightforward consequence of the inflation restriction sequence, as $p\nmid[\frak{F}:F_\infty]$.) It follows from Lemma~\ref{lem:relaxedistrue} and (i) that
\begin{align*}
\textup{Sel}_p(E/M_\infty)&\cong \textup{Hom}_{\frak{O}}\left(\frak{X},E[p^\infty]\right)^{\textup{Gal}(\frak{F}/M_\infty)}\\
&\cong \textup{Hom}_{\ooo}\left(\frak{X}(\rho^{-1}),\Phi/\ooo\right)^{\textup{Gal}(\frak{F}/M_\infty)}\cong \textup{Hom}_{\ooo}\left(\frak{X}^\rho_?,\Phi/\ooo\right)
\end{align*}
where $?=\cyc$ or $\infty$ (depending on whether $M_\infty=F^\cyc$ or $F_\infty$).
\end{proof}

For every prime $\frak{p}_i$ of $F$ above $p$, we denote the prime of $\frak{F}$ above $\frak{p}_i$ also by the symbol $\frak{p}_i$. Let $\frak{U}_{i}=\varprojlim \frak{U}_M$ be the inverse limit (with respect to norm maps) of the local units (at $\frak{p}_i$), where $M$ varies over finite subextensions of $\frak{F}_{\frak{p}_i}/F_{\frak{p}_i}$.   The compositum of the maps
\begin{align*} E(\frak{F}_{\frak{p}_i})\otimes\QQ_p/\ZZ_p\stackrel{}{\ra} \textup{Hom}(G_{\frak{F}_{\frak{p}_i}},E[p^\infty])\stackrel{}{\ra} \textup{Hom}(\frak{U}_i,E[p^\infty])\stackrel{\sim}{\ra}\textup{Hom}_{\ooo}(\frak{U}_i(\rho^{-1}),\Phi/\ooo)
\end{align*}
(where the first map comes from the identification $H^1(\frak{F}_{\frak{p}_i},E[p^\infty])\stackrel{\sim}{\ra}\textup{Hom}(G_{\frak{F}_{\frak{p}_i}},E[p^\infty])$ and the second map by the inclusion $\frak{U}_i\hookrightarrow G_{\frak{F}_{\frak{p}_i}}$ of local class field theory) induces a non-degenerate (c.f., \cite[Prop. 5.2]{rubin87}), $\ooo$-linear \emph{Kummer pairing}
$$\langle\,,\,\rangle: \left(E(F^{\textup{cyc}}_{\frak{p}_i})\otimes\QQ_p/\ZZ_p\right)\,\,{\times}\,\, \frak{U}_{i,\textup{cyc}}^{\rho}\lra \Phi/\ooo.$$

\begin{rem}
\label{rem:compareXwithstrictSelmer}
Let $\FF_{\textup{str}}$ denote the \emph{strict Selmer structure} on $Z$ (where $Z=\TT$ or $\TT(E)$) given by the local conditions:
\begin{itemize}
\item $H^1_{\FF_{\textup{str}}}(F_\frak{q},Z)=H^1_{\FFc}(F_\frak{q},Z)$ for every prime $\frak{q} \nmid p$,
\item $H^1_{\FF_{\textup{str}}}(F_p,Z)=0$.
\end{itemize}
It follows easily using the inflation-restriction sequence that
\be\label{eqn:twistX1}H^1_{\FF_{\textup{str}}^*}(F,\TT^*)^\vee=\frak{X}(\omega_E^{-1})\otimes_{\LL_\frak{F}}\LL=:\frak{X}^{\omega_E}_\infty.\ee By Lemma~\ref{lemma:RubinCW}, we conclude that
\begin{align}\label{eqn:twistX2}\textup{Sel}_p(E/F_\infty)^\vee&=\frak{X}^\rho_\infty=\frak{X}(\rho^{-1})\otimes_{\LL_\frak{F}}\LL\\
\notag&=H^1_{\FF_{\textup{str}}^*}(F,\TT^*)^\vee\otimes\langle\rho\rangle\cong H^1_{\FF_{\textup{str}}^*}(F,\TT(E)^*)^\vee
\end{align}
and on tensoring with $\LL_\cyc$ (and using once again the perfect control theorem) 
\begin{align}\label{eqn:twistX2}\textup{Sel}_p(E/F_\cyc)^\vee&=\frak{X}^\rho_\cyc=\frak{X}^\rho_\infty\otimes_{\LL}\LL_\cyc\\
\notag&\cong H^1_{\FF_{\textup{str}}^*}(F,\TT_\cyc(E)^*)^\vee.
\end{align}
Furthermore, we have
\begin{align}\label{eqn:twist}
H^1(F_{\frak{p}_i},\TT(E))\mathop{\longleftarrow}^{\sim}_{\textup{tw}} H^1(F_{\frak{p}_i},\TT)\otimes \langle \rho^{-1}\rangle=\frak{U}^\rho_{i,\infty}\,\,,
\end{align}
and on applying both sides with $\otimes_{\LL}\LL_\cyc$\,,
$$H^1(F_{\frak{p}_i},\TT_\cyc(E))\mathop{\longleftarrow}^{\sim} \frak{U}^\rho_{i,\cyc}\,.$$
Here $\textup{tw}: H^1(F_{\frak{p}_i},\TT) \ra H^1(F_{\frak{p}_i},\TT(E))$ is the twisting morphism which factors through the $\LL$-isomorphism (that we still denote by $\textup{tw}$) 
$$H^1(F_{\frak{p}_i},\TT)\otimes  \langle \rho^{-1}\rangle \stackrel{\sim}{\lra} H^1(F_{\frak{p}_i},\TT(E))\,.$$
\end{rem}

\subsection{Plus/Minus Selmer groups and $p$-adic $L$-functions}
\label{subsec:pmpadicL}
Following~\cite{kobayashi03} (see also \cite{iovitapollack}), we define the $\pm$-subgroups as follows:

\begin{define}
\label{define:pmsubgroups}
For every positive integer $n$, set
$$E^+(F_{n,\frak{p}_i}):=\{x \in E(F_{n,\frak{p}_i}): \textup{Tr}_{n/m}(x) \in E(F_{m-1,\frak{p}_i}) \hbox{ for } 0<m\leq n, m:\hbox{ odd}\},$$
$$E^-(F_{n,\frak{p}_i}):=\{x \in E(F_{n,\frak{p}_i}): \textup{Tr}_{n/m}(x) \in E(F_{m-1,\frak{p}_i}) \hbox{ for } 0<m\leq n, m:\hbox{ even}\},$$
where $\textup{Tr}_{n/m}(x): E(F_{n,\frak{p}_i})\lra E(F_{m,\frak{p}_i})$ is the trace map. We also set
$$E^\pm(F^{\textup{cyc}}_{\frak{p}_i})=\varinjlim E^\pm(F_{n,\frak{p}_i}).$$
\end{define}


\begin{define}
\label{def:pmselmergroup}
Let $\textup{Sel}_p(E/F_n)$ denote the classical Selmer group attached to $E$ and set $\textup{Sel}_p(E/F^{\textup{cyc}})=\varinjlim \textup{Sel}_p(E/F_n)$. Define the $\pm$-Selmer groups by setting
$$\textup{Sel}^\pm_p(E/F_n):=\ker\left(\textup{Sel}_p(E/F_n) \lra \bigoplus_{i=1}^g\frac{H^1(F_{n,\frak{p}_i},E[p^\infty])}{\frak{Kum}_i\left(E^\pm(F_{n,\frak{p}_i}) \otimes \QQ_p/\ZZ_p\right)}\right).$$
Let $\textup{Sel}^\pm_p(E/F^{\textup{cyc}})=\varinjlim \textup{Sel}^\pm_p(E/F_n)$.
\end{define}
We note that these two  Selmer groups actually correspond to the cases $(+,\cdots,+)$ and $(-,\cdots,-)$-Selmer groups among $2^g$ possible options.

\begin{define}
\label{def:omeganpm}
For a fixed topological generator $\gamma$ of $\Gamma_{\textup{cyc}}$ and $n\geq 1$, we define the element $\nu_n=\sum_{i=0}^{p-1}\gamma^{ip^{n-1}} \in \LL_{\textup{cyc}}$
and set
$$\omega_n^+=\prod_{\substack{ 1\leq i \leq n \\ i:\textup{ even }}} \nu_i\,\,,\,\,\,\,\,\omega_n^-=\prod_{\substack{ 1\leq i \leq n \\ i:\textup{ odd }}} \nu_i\,.$$
\end{define}

At the analytic end of things, Park and Shahabi~\cite{parkshahabi} have constructed a pair of signed (bounded) $p$-adic $L$-functions $L_p^\pm(E/F^+) \in \LL_{\textup{cyc}}\otimes \QQ_p$ whose basic properties are outlined in the following theorem. For each prime $\frak{P}$ of $F^+$ above $p$, we let $\alpha=\alpha(\frak{P})$ denote a distinguished roof of the Hecke polynomial for $E$ at $\frak{P}$. As the prime $\frak{P}$ is inert in $F/F^+$, it follows that $\alpha(\frak{P})^2=-p$ and our convention is that we pick always the same square root $\alpha$ of $-p$ (other choices would only alter the bounded $p$-adic $L$-functions of Park and Shahabi only by $\pm1$).
\begin{thm}[Park-Shahabi]
\label{thm:parkshahabi}
There exist a pair of elements $L_p^+(E/F^+), L_p^-(E/F^+) \in \LL_{\textup{cyc}}\otimes \QQ_p$ which are characterized by the following interpolation properties: For every non-trivial character $\chi$ of $\Gamma^\cyc$ of finite order $p^n$,
\begin{itemize}
\item for odd $n$, we have 
$$\displaystyle{\chi\left(L_p^+(E/F^+)\right)=(-1)^{\frac{n+1}{2}g}p^{\frac{n+1}{2}(g-1)}\frac{\tau(\chi)}{\chi(\omega_n^+)}\frac{L(E,\overline{\chi},1)}{\Omega_E(F^+)}\,,}$$
\item for even $n$, we have 
$$\displaystyle{\chi\left(L_p^-(E/F^+)\right)=(-1)^{g\cdot(\frac{n}{2}+1)}p^{n/2(g-1)}\frac{\tau(\chi)}{\chi(\omega_n^-)}}\frac{L(E,\overline{\chi},1)}{\Omega_E(F^+)}\,.$$  
\end{itemize}
Furthermore, their value at the trivial character is given by 
$$\mathbf{1}\left(L_p^+(E/F^+)\right)=u_1\cdot \frac{L(E,1)}{\Omega_E(F^+)} \,\,\,\,,\,\,\,\,\mathbf{1}\left(L_p^-(E/F^+)\right)=u_2\cdot \frac{L(E,1)}{\Omega_E(F^+)}$$
where  $u_1, u_2 \in \overline{\QQ}^\times$ (whose precise values we need not know).
\end{thm}
Here, the period $\Omega_E(F^+)$ corresponds to the quantity $\Omega(\epsilon_0,f_E)D_{F^+}^{-1}(\sqrt{-1})^{-g}$  in op.cit., where $f_E$ is the Hilbert modular form of parallel weight two that one associates (via the Weil-Jacquet-Langlands correspondence) to our CM elliptic curve $E$ and $D_{F^+}$ is the discriminant of $F^+$.
\begin{proof}
This is an immediate consequence of the interpolation formula \cite[Theorem 2.3]{parkshahabi} and the factorization \cite[Theorem 2.7]{parkshahabi}, used together with \cite[Lemma 4.7]{pollackduke2003}.
\end{proof}
The following is the signed-main conjecture that Park and Shahabi posed in this context.
\begin{conj}[Park-Shahabi]$\,$
\label{conj:parkshahabi}
\begin{itemize}
\item[(i)] Both modules $\textup{Sel}^+_p(E/F^{\textup{cyc}})$ and $\textup{Sel}^-_p(E/F^{\textup{cyc}})$ are $\LL_{\textup{cyc}}$-cotorsion.
\item[(ii)] Any generator of the ideal $\textup{char}\left(\textup{Sel}^\pm_p(E/F^{\textup{cyc}})^\vee\right)$ generates $L_p^\pm(E/F^+)(\LL_{\textup{cyc}}\otimes\QQ_p)$. 
\end{itemize}
\end{conj}

\begin{define}
Let $V^\pm_i \subset \frak{U}_{i,\textup{cyc}}^{\rho}$ denote the orthogonal complement of $E^\pm(F^{\textup{cyc}}_{\frak{p}_i})\otimes\QQ_p/\ZZ_p$ under the Kummer pairing defined as above. Via the identifications in Remark~\ref{rem:compareXwithstrictSelmer}, we view $V^\pm_i$ as a submodule of $H^1(F_{\frak{p}_i},\TT_\cyc(E))$.
\end{define}
Set $\frak{U}=\oplus_{i=1}^g \frak{U}_i$ and $\mathbb{V}_{E,\cyc}^\pm=\oplus_{i=1}^g {V}_i^\pm$. Define $\frak{U}_{\infty}^\rho$ and $\frak{U}_{\textup{cyc}}^\rho$ in a similar manner. Let $\alpha:\frak{U}\ra \frak{X}$ be the Artin map of global class field theory (and likewise the compositum
$$\frak{a}: \frak{U}_{\textup{cyc}}^\rho\stackrel{\sim}{\lra} H^1(F_p,\TT(E))\lra H^1_{\FF_{\textup{str}}^*}(F,\TT_\cyc(E)^*)^\vee$$
be the map obtained by the Poitou-Tate global duality). The following properties of the submodules $\mathbb{V}_{E,\cyc}^\pm$ may be obtained as in \cite[Theorem 4.3 and Proposition 4.4]{pollackrubin} (and using the comparisons of Remark~\ref{rem:compareXwithstrictSelmer} wherever necessary):
\begin{prop}[Pollack-Rubin]For every $1\leq i\leq g$ we have,
\label{prop:pmstructureandselmer}
\begin{enumerate}
\item[(i)]  the $\LL$-module $\frak{U}_{i,\infty}^\rho\cong H^1(F_{\frak{p}_i},\TT(E))$ and the $\LL_{\textup{cyc}}$-module  $\frak{U}_{i,\textup{cyc}}^\rho\cong H^1(F_{\frak{p}_i},\TT_\cyc(E))$ are free of rank two,
\item[(ii)]  the $\LL_{\textup{cyc}}$-modules $V^\pm_i$ and $H^1(F_{\frak{p}_i},\TT_\cyc(E))/V^\pm_i$ are both free of rank one,
\item[(iii)] there is a (non-canonical) submodule $\mathbf{V}^\pm_i \subset H^1(F_{\frak{p}_i},\TT(E))$ whose image under the natural map 
$$\pi_{\cyc}:H^1(F_{\frak{p}_i},\TT(E)) \lra H^1(F_{\frak{p}_i},\TT_\cyc(E))$$ 
is the module $V^\pm_i$ and is such that both $\mathbf{V}^\pm_i $ and $H^1(F_{\frak{p}_i},\TT(E))/\mathbf{V}^\pm_i$ are free of rank one over $\LL$.
\item[(iv)] $\textup{Sel}_p^\pm(E/F^{\textup{cyc}})^\vee=\frak{X}_{\textup{cyc}}^\rho/\alpha(\mathbb{V}_{E,\cyc}^\pm)\cong H^1_{\FF_{\textup{str}}^*}(F,\TT_\cyc(E)^*)^\vee/\frak{a}(\mathbb{V}_{E,\cyc}^\pm)\,.$
\end{enumerate}
\end{prop}

\subsection{An explicit reciprocity conjecture for Rubin-Stark elements}
\label{subsec:reciprocityconj}
As we have assumed that the prime $p$ splits completely in $F^+/\QQ$, we may identify $F^+_{\wp_i}$ with $\QQ_p$ and the constructions of Kobayashi~\cite[\S4]{kobayashi03} for a supersingular elliptic curve defined over $\QQ_p$ carries over.
\begin{define}
\label{def:formallog}
Given positive integers $n$ and $1\leq i\leq g$, let $E_1(F_{n,\frak{p}_i})\subset E(F_{n,\frak{p}_i})$ denote the kernel of the reduction map modulo $\frak{p}_i$. Then $E_1(F_{n,\frak{p}_i})$ is the pro-$p$ part of $E(F_{n,\frak{p}_i})$ and we define the logarithm map $\lambda_E$ to be the compositum
$$\lambda_E:E(F_{n,\frak{p}_i})\twoheadrightarrow E_1(F_{n,\frak{p}_i})\stackrel{\sim}{\lra}\hat{E}(\frak{p}_i)\lra F_{n,\frak{p}_i},$$
where $\hat{E}$ is the formal group of $E/F_{\frak{p}_i}$.
\end{define}

We consider Kobayashi's trace-compatible sequence of points $d_{n,i}\in E(F_{n,\wp_i}^+)$; we refer the reader to \cite[\S3]{pollackrubin} for basic properties of these points and their comparison with Kobayshi's original construction. Using the complex multiplication map $E(F_{n,\wp_i}^+)\otimes\ooo\ra E(F_{n,\frak{p}_i})$, we define the element $\frak{d}_{n,i} \in E(F_{n,\frak{p}_i})$ as the image of $d_{n,i}$. Key properties of the elements $\frak{d}_{n,i}$ are outlined in the following Proposition:
\begin{prop}[Kobayashi] Let $\Gamma_n:=\textup{Gal}(F_n/F)$. For every positive integers $n$ and $1\leq i \leq g$,
\label{prop:dngenerators}
\begin{itemize}
\item[(i)] $\sum_{\sigma\in \Gamma_n}\chi(\sigma)\lambda_E(\frak{d}_{n,i}^\sigma)=(-1)^{\left[n/2\right]}\tau(\chi)$,
    where $\tau(\chi)$ is the Gauss sum,
\item[(ii)] if $\epsilon$ is the sign of $(-1)^n$, then $E^{\epsilon}(F_{n,\wp_i})=\ooo[\Gamma_n]\frak{d}_{n,i}$ and $E^{-\epsilon}(F_{n,\wp_i})=\ooo[\Gamma_n]\frak{d}_{n-1,i}.$ Moreover, we have $E^{\epsilon}(F_{n,\wp_i})+E^{-\epsilon}(F_{n,\wp_i})=E(F_{n,\wp_i})$.
\end{itemize}
\end{prop}
\begin{proof}
This a restatement of \cite[Theorem 3.2]{pollackrubin}.
\end{proof}


\begin{define}
\label{def:moduleofRSelements}
Let $\mathcal{S}_\infty \subset \varprojlim_{M\subset F_\infty}\wedge^g\, H^1_{\FFc}(M,T)$ denote the cyclic $\LL$-module generated by the tower of Rubin-Stark elements $\{\varepsilon_{M}^\chi\}$ and let $\mathcal{S}^{E,p}_{\infty}$ be the image of $\mathcal{S}_\infty$ under the compositum of maps
$$\varprojlim_{M\subset F_\infty} \wedge^g  H^1_{\FFc}(M,T)\ra \varprojlim_{M\subset F_\infty} \wedge^g  H^1(M_p,T)\stackrel{\sim}{\lra} \wedge^g H^1(F_p,\TT)\mathop{\lra}^{\sim}_{\textup{tw}} \wedge^g H^1(F_p,\TT(E)).$$
We write $\textup{pr}_i: H^1(F_p,\TT) \ra H^1(F_{\frak{p}_i},\TT)$ for the obvious projection map (similarly, for the map defined on $H^1(M_p,T)$ for any $M$ as above).
\end{define}
\begin{conj} $\,$
\label{conj:reciprocity}
\begin{itemize}
\item[(i)] There exists a generator ${\Xi}_1\wedge\cdots\wedge{\Xi}_g$ of the cyclic $\LL$-module $\mathcal{S}^{E,p}_{\infty}$ such that for every $n\in \ZZ^+$, primitive character $\chi:\textup{Gal}(F_n/F)\ra \pmb{\mu}_{p^\infty}$ and for every positive integer $k$,

$$\det\left(\sum_{\sigma\in \Gamma_n}\chi^{-1}(\sigma)\langle \frak{d}_{n,i}^\sigma\otimes p^{-k},\frak{u}_{i,j}\rangle\right)=p^{-kg}(-1)^{[n/2]g}\tau(\chi)\chi(\omega_n^{\epsilon})^{g-1} p^{\left[\frac{n+1}{2}\right](g-1)}\frac{L(E/F^+,\chi,1)}{\Omega_E(F^+)}$$
where 
$\displaystyle{\frak{u}_{i,j}:=\textup{pr}_{i}({\Xi}_j) \in H^1(F_{\frak{p}_i},\TT(E))\mathop{=}^{\textup{tw}} \frak{U}_{\infty}^\rho}$\,, $L(E/F^+,\chi,s)$ is the $L$-series twisted by the character $\chi$ and $\epsilon$ is the sign of $(-1)^{n}$. 
\item[(ii)] For all but finitely many characters $\chi$ of $\Gamma_{\textup{cyc}}$, we have $L(E/F^+,\chi,1)\neq0$.
\end{itemize}
\end{conj}
\begin{rem}
The first part of Conjecture~\ref{conj:reciprocity} is a natural (but partial, in that it only concerns the plus/minus subgroups of the local cohomology groups) generalization of Coates and Wiles' reciprocity law~\cite{coateswiles77,wiles78reciprocity}. The second part proposes an extension of Rohrlich's \cite{rohrlich84cyclo} non-vanishing theorem in the special case $F^+=\QQ$; see also~\cite{rohrlich89} for a result in this direction (which proves the weaker statement that (ii) holds true for infinitely many characters $\chi$).
\end{rem}
Recall the lift $\mathbf{V}^{\pm}_i\subset \frak{U}^\rho_{i,\infty}\displaystyle{\mathop{=}^{\textup{tw}}}H^1(F_{\frak{p}_i},\TT(E))$ of $V^\pm_{i}$ and set $$\mathbb{V}_{E}^\pm:=\oplus_{i=1}^g \mathbf{V}^{\pm}_i \subset H^1(F_p,\TT(E)).$$
Let $\mathbb{V}^\pm \subset H^1(F_p,\TT)$ be the inverse image of $\mathbb{V}_{E}^\pm$ under the twisting isomorphism $\textup{tw}$.

Recall the modules $\mathcal{M}$ and $\mathcal{M}_\cyc$ from Section~\ref{subsubsec:selmerstrE}.
\begin{thm}
\label{thm:transversepm}
If Conjecture~\ref{conj:reciprocity} holds true, then
$\mathcal{M}_\cyc \cap  \mathbb{V}^\pm_{E,\cyc}=0=\mathcal{M}\cap \mathbb{V}^\pm_E$\,.
\end{thm}
\begin{proof}
Let $\Xi$ denote the $\LL$-submodule of $H^1(F_p,\TT(E))$ generated by $S=\{{\Xi}_1,\cdots,{\Xi}_g\}$ and $\Xi_{\textup{cyc}}$ its image inside $H^1(F_p,\TT_\cyc(E))$ generated by $S_\cyc=\{{\Xi}_1^\cyc,\cdots,{\Xi}_g^\cyc\}$ where $\Xi_j^\cyc\in H^1(F_p,\TT_\cyc(E))$ is the image of $\Xi_j$. First, notice that $S$ is linearly independent over $\LL$ and $S_\cyc$ is linearly independent over $\LL_\cyc$. Indeed, if $a_1\cdot{\Xi}_1^\cyc +\cdots+a_g\cdot{\Xi}_g^\cyc=0$ for some $a_1,\cdots,a_g\in \LL_\cyc$, then this would imply that 
$$\overline{a}_1\cdot\textup{col}_1+\cdots + \overline{a}_g\cdot\textup{col}_g=0,$$
where $\textup{col}_j$ denote the $j$th column of the matrix $\left[\sum_{\sigma\in \Gamma_n}\chi^{-1}(\sigma)\langle \frak{d}_{n,i}^\sigma\otimes p^{-k},\frak{u}_{i,j}\rangle\right]_{i,j}$ and $\overline{a} \in \ooo$ is the image of $a \in \LL_\cyc$ under the augmentation map. The explicit reciprocity conjecture (applied for large enough $n$) shows that $\overline{a}_i=0$ for every $i$, so that $a_i=(\gamma_\cyc-1)b_i$ for some $b_i \in \LL_\cyc$. As the $\LL_\cyc$-module  $H^1(F_p,\TT_\cyc(E))$ is $\LL_\cyc$-torsion free, we conclude
$$b_1\cdot{\Xi}_1^\cyc +\cdots+b_g\cdot{\Xi}_g^\cyc=0$$
and in turn that each $b_i$ is divisible by $\gamma_\cyc-1$. Iterating this argument, we conclude that each $a_i$ is divisible by arbitrarily large powers of $\gamma_\cyc-1$, then $a_i=0$ for every $i$. This completes the verification that $S_\cyc$ is $\LL_\cyc$-linearly independent. The assertion that the set $S$ is $\LL$-linearly independent is proved similarly. We therefore conclude that $\Xi$ is a free $\LL$-module, $\Xi_\cyc$ is a free $\LL_\cyc$-module and both have rank $g$.

We now content to prove that $\Xi_\cyc \cap \mathbb{V}^\pm_{E,\cyc}=0\,. $ Suppose that $\sum_i a_i\cdot \Xi_i^\cyc$ belongs to $\mathbb{V}^\epsilon_{E,\cyc}$ (where $\epsilon=+$ or $-$) for some $a_1,\cdots,a_g \in \LL_\cyc$. Let $n$ be a positive integer chosen so that the sign of $(-1)^n$ is $\epsilon$ and $L(E/F^+,\chi,1)\neq 0$ for some primitive character $\chi$ of $\Gamma_n$. If $\textup{col}_1,\cdots,\textup{col}_g$ are the column vectors of $\left[\sum_{\sigma\in \Gamma_n}\chi^{-1}(\sigma)\langle \frak{d}_{n,i}^\sigma\otimes p^{-k},\frak{u}_{i,j}\rangle\right]_{i,j}$ as above, we conclude once again that 
$$\overline{a}_1\cdot\textup{col}_1+\cdots + \overline{a}_g\cdot\textup{col}_g=0,$$
and by the explicit reciprocity conjecture, that each $a_i$ is divisible by $\gamma_\cyc-1$. Write $a_i=(\gamma_\cyc-1)b_i$ so that we have 
$$(\gamma_\cyc-1)\cdot\sum_{i} b_i\cdot\Xi_i^{\cyc} \in \mathbb{V}^\epsilon_{E,\cyc}\,.$$ But according to Proposition~\ref{prop:pmstructureandselmer}(ii), the $\LL_\cyc$-module $ H^1(F_p,\TT_\cyc(E))/\mathbb{V}^\epsilon_{E,\cyc}$ is torsion-free and therefore $\sum_{i} b_i\cdot\Xi_i^{\cyc} \in \mathbb{V}^\epsilon_{E,\cyc}$. Repeating the argument $s$ times (for every positive integer $s$) we conclude that $(\gamma_\cyc-1)^s$ divides each $a_i$, and therefore that $a_i=0$, as desired.

It is not hard to see that there is an $r \in \LL$ with $\pi_\cyc(r)\neq0$ and $r\cdot\Xi\subset \mathcal{M}$ (hence, we also have $\pi_{\cyc}(r)\cdot\Xi_\cyc\subset \mathcal{M}_\cyc$). 
The submodule $\pi_{\cyc}(r)\cdot\Xi_\cyc$ has $\LL_\cyc$-rank $g$ and therefore the quotient $\mathcal{M}_\cyc/\pi_{\cyc}(r)\cdot\Xi_\cyc$ is torsion. This in turn shows that there is a nonzero element $\tilde{r}\in \LL_\cyc$ with $\tilde{r}\mathcal{M}_\cyc\subset \Xi_\cyc$. Using this observation, the fact that $ H^1(F_p,\TT_\cyc(E)$ is $\LL_\cyc$-torsion free and our conclusion from the previous paragraph that $\Xi_\cyc \cap \mathbb{V}^\pm_{E,\cyc}=0$, it follows that $\mathcal{M}_\cyc \cap\mathbb{V}^\pm_{E,\cyc}=0 $.

 It now follows at once from Nakayama's lemma that $\mathcal{M} \cap\mathbb{V}^\pm_E=0 $ as well.
\end{proof}
\begin{rem}
\label{rem:localconditionsnatural}
Theorem~\ref{thm:transversepm} supplies us with two \emph{natural} choices for the free $\LL$-module $\mathbb{V}_E$ in Definition~\ref{def:choosetransverseforE}: $\mathbb{V}_E^+$ or $\mathbb{V}_E^-$\,.
\end{rem}
\subsection{Rubin-Stark elements and the plus/minus main conjecture}
Throughout this subsection, we assume the truth of the Explicit Reciprocity Conjecture~\ref{conj:reciprocity} (therefore, implicitly the truth of Rubin-Stark conjectures) and of Leopoldt's conjecture for the number field $L$. Throughout, let $\epsilon$ stand for one of $+$ or $-$.  

\begin{define}
\label{def:twistedRS}\begin{itemize}
\item[(i)]Let $\frak{Q}_{\epsilon,\infty}:=H^1(F_p,\TT(E))/\mathbb{V}^\epsilon_E$ and let $\textup{loc}_p^\epsilon$ denote the compositum
    $$\textup{loc}_p^\epsilon: H^1(F,\TT(E))\stackrel{\textup{loc}_p}{\lra}H^1(F_p,\TT(E))\lra \frak{O}_{\epsilon,\infty}.$$
    (The quotient $\frak{Q}_{\epsilon,\infty}$ is related (via the twisting map $\textup{tw}$) to the quotients $Q$ defined as in Section~\ref{sec:gras} and the map $\textup{loc}_p^\epsilon$ to $\locu$.) Observe that $\frak{Q}_{\epsilon,\infty}$ is a free $\LL$-module of rank $g$ by Proposition~\ref{prop:pmstructureandselmer}(iii) and the map $\textup{loc}_p^\epsilon$ is injective by Theorem~\ref{thm:transversepm}.
\item[(ii)]  Let ${\frak{u}}_{\textup{R-S}}=\frak{u}_1\wedge\cdots\wedge \frak{u}_g \in \wedge^g \frak{Q}_{\epsilon,\infty}$ denote the image of the tower of Rubin-Stark elements
$\textup{loc}_p^{\epsilon}(\varepsilon_{F_\infty}^{\omega_E}) \in \wedge^g\,\mathcal{Q}$ (given as in Definition~\ref{def:rubinstarktowerlocal}, with the choice $\mathcal{V}=\textup{tw}^{-1}(\mathbb{V}_E^\epsilon)$)  under the twisting map $\wedge^g \,\mathcal{Q} \ra \wedge^g \frak{Q}_{\epsilon,\infty}$.
 \item[(iv)] Similarly define $\frak{Q}_{\epsilon,\textup{cyc}}$ and the map 
 $$\textup{loc}_p^\epsilon: H^1(F,\TT_\cyc(E))\lra  \frak{Q}_{\epsilon,\textup{cyc}}\,.$$
  Let $\bar{\frak{u}}_{\textup{R-S}}=\bar{\frak{u}}_1\wedge\cdots\wedge \bar{\frak{u}}_g \in \wedge^g \frak{Q}_{\epsilon,\textup{cyc}}$ be the image of $\frak{u}_{\textup{R-S}}$ under the projection map
$\wedge^g\frak{Q}_{\epsilon,\infty} \ra \wedge^g \frak{Q}_{\epsilon,\textup{cyc}}\,.$
\end{itemize}
\end{define}

\begin{thm}
\label{thm:localmainconj}
Any generator of the ideal $\textup{char}\left(\wedge^g \frak{Q}_{\epsilon,\textup{cyc}}/\LL_{\textup{cyc}}\cdot\bar{\frak{u}}_{\textup{R-S}}\right)$ generates the cyclic  $(\LL_\cyc\otimes\QQ_p)$-module $(\LL_{\textup{cyc}}\otimes \QQ_p)\cdot L_p^\epsilon(E/F^+)$.
\end{thm}
The proof we shall present below for this theorem is essentially identical to the proof of \cite[Theorem 7.2]{pollackrubin} after a number of obvious modifications. 
\begin{proof}
Let $\mu_{i}^\pm \in \textup{Hom}(E^\pm(F^{\textup{cyc}}_{\frak{p}_i})\otimes\QQ_p/\ZZ_p,\QQ_p/\ZZ_p)$ be the generator which was essentially constructed by Kobayashi~\cite[Theorem 6.2]{kobayashi03}, whose properties are outlined in \cite[Theorem 7.1]{pollackrubin}. Let $\Xi=\xi_1\wedge \cdots\wedge \xi_g \in \wedge^g \in \frak{S}_{\infty}^\rho$ be as in the statement of Conjecture~\ref{conj:reciprocity}. Let $\varphi^{\pm}_{i,j}$ denote the image of $\xi_j$ inside $\textup{Hom}(E^\pm(F^{\textup{cyc}}_{\frak{p}_i})\otimes\QQ_p/\ZZ_p,\QQ_p/\ZZ_p)$. Then
\be\label{eqn:defhij}\varphi^{\pm}_{i,j}=h^{\pm}_{i,j}\,\mu^{\pm}_{i}\ee
for some $h_{i,j}^\pm \in \LL_{\textup{cyc}}$ and
\be
\label{eqn:crucialdet1}
\wedge^g \frak{Q}_{\textup{cyc}}^{\pm,\rho}/\LL_{\textup{cyc}}\cdot\bar{\frak{u}}_{\textup{R-S}} \stackrel{\sim}{\ra} \LL_{\textup{cyc}}\big{/}\det\left(h_{i,j}^\pm\right)\,\,,\,\, \textup{char}\left(\wedge^g \frak{Q}_{\textup{cyc}}^{\pm,\rho}/\LL_{\textup{cyc}}\cdot\bar{\frak{u}}_{\textup{R-S}} \right)=\det\left(h_{i,j}^\pm\right) \LL_\cyc.
\ee
Let $\chi:\Gamma_{\textup{cyc}}\ra \pmb{\mu}_{p^n}$ be any character of order $p^n>1$. It follows from (\ref{eqn:defhij}) that for every $k\geq 1$ and $1\leq i\leq g$,
\be
\label{eqn:identityphih}
L_{i,j}^\pm:=\sum_{\sigma\in \Gamma_n}\chi(\sigma)\varphi_{i,j}^\pm(\frak{d}_{n,i}^\sigma\otimes p^{-k})=\chi(h_{i,j}^\pm)\sum_{\sigma\in \Gamma_n}\chi(\sigma)\mu_{i}^\pm(\frak{d}_{n,i}^\sigma\otimes p^{-k})=:R_{i,j}^{\pm}.
\ee
A computation of Kobayashi (c.f., \cite[Theorem 7.1]{pollackrubin}) shows that
$R_{i,j}^\pm=\chi(h_{i,j}^{\pm})\chi(\omega_n^\pm)p^{-k}$ so that we have
\be
\label{eqn:rightexplicit}
\det\left(R_{i,j}^\pm\right)=p^{-kg}\chi\left(\det\left(h_{i,j}^{\pm}\right)\right)\chi(\omega_n^\pm)^g 
\ee
On the other hand, Conjecture~\ref{conj:reciprocity} (which we assume) together with Proposition~\ref{prop:dngenerators} shows that
\be
\label{eqn:leftexplicit}
\det\left(L_{i,j}^\epsilon\right)=p^{-kg}(-1)^{[n/2]g}\tau(\chi)\chi(\omega_n^{\epsilon})^{g-1} p^{\left[\frac{n+1}{2}\right](g-1)}\frac{L(E/F^+,\chi,1)}{\Omega_E(F^+)}
\ee
where $\epsilon$ is the sign of $(-1)^{n+1}$. It follows from (\ref{eqn:identityphih}), (\ref{eqn:rightexplicit}) and (\ref{eqn:leftexplicit}) that
$$(-1)^{[n/2]g}\tau(\chi)\chi(\omega_n^{\epsilon})^{g-1} p^{\left[\frac{n+1}{2}\right](g-1)}\frac{L(E/F^+,\chi,1)}{\Omega_E(F^+)} \equiv \chi\left(\det\left(h_{i,j}^{\epsilon}\right)\right)\chi(\omega_n^\epsilon)^g \mod p^{kg}.$$
for every $k$. The proof follows from Theorem~\ref{thm:parkshahabi}.
\end{proof}

\begin{thm}
\label{thm:theweakleopoldtconjforE}
The $\LL_\cyc$-module $H^1_{\FFc}(F,\TT_\cyc(E)^{*})^\vee$ is torsion.
\end{thm}
This statement is a reformulation of the weak Leopoldt conjecture for our CM elliptic curve $E$ and the cyclotomic $\ZZ_p$-extension $F_\cyc$ of $F$. 
\begin{proof}
By Corollary~\ref{cor:localnonvanishingimpliesweakLeoforE}, it suffices to prove the existence of a Kolyvagin system $\pmb{\kappa}\in\overline{\mathbf{KS}}(\TT_\cyc(E),\FF_{\mathbb{L}},\PP)$ with non-vanishing initial term $\kappa_1 \in H^1_{\FF_{\mathbb{L}}}(F,\TT_\cyc(E))$. A suitable modification in Theorem~\ref{thm:ellrestrictedESmain} (so as to allow the replacement of $\TT$ with $\TT(E)$ and $\al$ with $\mathbb{L}_E$, etc.) shows that  there is an isomorphism $\Psi:\wedge^g\,\frak{Q}_{\epsilon,\infty} \ra \mathbb{L}_{E}$  and a Kolyvagin system $\pmb{\kappa}(\textup{R-S})\in \overline{\mathbf{KS}}(\TT(E),\FF_{\mathbb{L}},\PP)$  with the following properties: 
\begin{itemize}
\item The initial term $\kappa_1(\textup{R-S})\in H^1_{\FF_{\mathbb{L}}}(F,\TT(E))$ of $\pmb{\kappa}(\textup{R-S})$ verifies that 
\be\label{eqn:RStwistedKSinfty}
\textup{loc}_p^\epsilon(\kappa_1(\textup{R-S}))=\Psi ({\frak{u}}_{\textup{R-S}})\,.
\ee
\item Let $\pmb{\kappa}^{\cyc}(\textup{R-S})\in \overline{\mathbf{KS}}(\TT_\cyc(E),\FF_{\mathbb{L}},\PP)$ be the image of $\pmb{\kappa}(\textup{R-S})$ and let $\kappa_1^\cyc(\textup{R-S}) \in H^1_{\FF_{\mathbb{L}}}(F,\TT_\cyc(E))$ be its initial term. Then
\be\label{eqn:RStwistedKSinfty} 
\textup{loc}_p^\epsilon(\kappa_1^\cyc(\textup{R-S}))=\Psi_\cyc({\bar{\frak{u}}}_{\textup{R-S}})\,.
\ee
where $\Psi_\cyc:\wedge^g\,\frak{Q}_{\epsilon,\cyc} \ra \mathbb{L}_{E}^\cyc$ is the isomorphism induced from $\Psi$ by base change.
\end{itemize}
The proof now follows from (\ref{eqn:RStwistedKSinfty}) Theorem~\ref{thm:localmainconj} and the second part of the Explicit Reciprocity Conjecture~\ref{conj:reciprocity} (from which follows that $L_p^\epsilon(E/F^+)\in \LL_\cyc\otimes\QQ_p$ is non-zero).
\end{proof}
\begin{define}
\label{def:pminfinity}
Let $\FF_\epsilon$ denote the Selmer structure on $\TT$ (and on its subquotients, given by propagation as usual) defined by the local conditions
\begin{itemize}
\item $H^1_{\FF_{\epsilon}}(F_\frak{q},\TT(E))=H^1_{\FFc}(F_\frak{q},\TT(E))$ for every prime $\frak{q} \nmid p$\,,
\item $H^1_{\FF_{\epsilon}}(F_p,\TT(E))=\mathbb{V}^\epsilon_{E}$\,.
\end{itemize}
Set $\mathcal{Y}_{\epsilon,\infty}=H^1_{\FF_\epsilon^*}(F,\TT(E)^*)^\vee$ and $\mathcal{Y}_{\epsilon,\cyc}=H^1_{\FF_\epsilon^*}(F,\TT_\cyc(E)^*)^\vee$.
\end{define}
\begin{rem}
\label{rem:comparepmwithnull}
Thanks to Theorem~\ref{thm:transversepm}, the Selmer structure $\FF_\epsilon$ agrees with the Kobayashi Selmer structure $\FF_{\textup{Kob}}$ with the choice $\mathbb{V}_E=\mathbb{V}_E^\epsilon$ in Definition~\ref{def:selmerstructureE}.
\end{rem}
\begin{lemma}
\label{lemma:comparepmselmergroups}
$\mathcal{Y}_{\epsilon,\textup{cyc}}\cong \textup{Sel}^{\epsilon}(E/F^{\textup{cyc}})^\vee.$
\end{lemma}
\begin{proof}

The Poitou-Tate global duality sequence
 $$0\lra H^1_{\FF_{\textup{str}}}(F,\TT(E))\lra H^1_{\FF_\epsilon}(F,\TT(E))\stackrel{\textup{loc}_p}{\lra} \mathbb{V}_E^{\epsilon}\stackrel{\frak{a}}{\lra} H^1_{\FF_{\textup{str}}^*}(F,\TT(E)^*)^\vee\lra \mathcal{Y}_{\epsilon,\infty}\lra 0$$
 reduces to the sequence
 $$0\lra \mathbb{V}_E^{\epsilon}\lra H^1_{\FF_{\textup{str}}^*}(F,\TT(E)^*)^\vee\lra \mathcal{Y}_{\epsilon,\infty}\lra 0$$
 thanks to Proposition~\ref{prop:minusvanishingforunits} and Theorem~\ref{thm:theweakleopoldtconjforE}. Applying the functor $-\otimes_\LL \LL_{\textup{cyc}}$ and using the control theorem \cite[Lemma 3.5.3]{mr02}, we obtain the exact sequence
 \be\label{eqn:resolutionofY}0\lra \mathbb{V}_{E,\cyc}^{\epsilon}\stackrel{\frak{a}}{\lra} H^1_{\FF_{\textup{str}}^*}(F,\TT_\cyc(E)^*)^\vee\lra \mathcal{Y}_{\epsilon,\cyc}\lra 0\,.\ee
 (where the exactness on the left follows from rank considerations and the fact that $\mathbb{V}_{E,\cyc}^{\epsilon}$ is free of rank $g$) This shows using Proposition~\ref{prop:pmstructureandselmer}(iv) that
$$\mathcal{Y}_{\epsilon,\cyc}\cong  H^1_{\FF_{\textup{str}}^*}(F,\TT_\cyc(E)^*)^\vee/\frak{a}(\mathbb{V}_{E,\cyc}^{\epsilon})\cong\textup{Sel}^{\pm}(E/F^{\textup{cyc}})^\vee.$$
 \end{proof}
Fix a generator $L_p^{\pm,\textup{alg}} \in \LL_\cyc$ of the ideal $\textup{char}\left(\textup{Sel}^\pm(E/F^{\textup{cyc}})^\vee\right)$. We have the following result towards Conjecture~\ref{conj:parkshahabi}. 
\begin{thm}
\label{thm:improvedmainconj} We have $L_p^{\pm,\textup{alg}}\mid L_p^\pm(E/F^+)$
(inside the ring $\LL_{\textup{cyc}}\otimes\QQ_p$). This divisibility is in fact an equality if we assume the Strong Rubin-Stark Conjecture for $\frak{E}$.
\end{thm}
\begin{proof}
The proof of Theorem~\ref{thm:mainconjdivisibility} applied with the Kolyvagin system 
$$\pmb{\kappa}^{\cyc}(\textup{R-S})\in\overline{\mathbf{KS}}(\TT_\cyc(E),\FF_{\mathbb{L}},\PP)$$
in place of the Rubin-Stark $\al$-restricted Kolyvagin system for $\TT$ and the Selmer structure $\FF_\epsilon$ in place of $\FF_{\frak{tr}}$ shows that 
$$\textup{char}\left(H^1_{\FF^*_{\epsilon}}(F,\TT(E)^*)^\vee\right)\mid \textup{char} \left(\wedge^g \frak{Q}_{\epsilon,\textup{cyc}}/\LL_{\textup{cyc}}\cdot\bar{\frak{u}}_{\textup{R-S}} \right)\,.$$
The first part of the Theorem now follows from Theorem~\ref{thm:localmainconj} and Lemma~\ref{lemma:comparepmselmergroups}.

If the Strong Rubin-Stark conjecture holds true, Theorem~\ref{thm:mainconjdivisibility} (after twisting) shows that 
$$\textup{char}\left(H^1_{\FF^*_{\epsilon}}(F,\TT(E)^*)^\vee\right)=\textup{char} \left(\wedge^g \frak{Q}_{\epsilon,\infty}/\LL\cdot{\frak{u}}_{\textup{R-S}} \right)\,.$$
This, however, means using Proposition~\ref{prop:4termexact} and Theorem~\ref{thm:localmainconj} that
$$\textup{char}\left(H^1_{\FF^*_{\mathbb{L}^*}}(F,\TT(E)^*)^\vee\right)=\textup{char} \left(H^1_{\FF_{\mathbb{L}}}(F,\TT(E))/\LL\cdot\kappa_1(\textup{R-S}) \right)$$
and by Proposition~\ref{prop:usefulforequality} that the Kolyvagin systems  $\pmb{\kappa}(\textup{R-S})$ and its image $\pmb{\kappa}^\cyc(\textup{R-S})$ are both primitive. This shows that 
$$\textup{char}\left(H^1_{\FF^*_{\mathbb{L}^*}}(F,\TT_\cyc(E)^*)^\vee\right)=\textup{char} \left(H^1_{\FF_{\mathbb{L}}}(F,\TT_\cyc(E))/\LL_\cyc\cdot\kappa_1^\cyc(\textup{R-S})\right)$$
and once again applying Proposition~\ref{prop:4termexact}, we conclude that
$$\textup{char}\left(H^1_{\FF^*_{\epsilon}}(F,\TT_\cyc(E)^*)^\vee\right)= \textup{char} \left(\wedge^g \frak{Q}_{\epsilon,\cyc}/\LL_{\textup{cyc}}\cdot{\bar{\frak{u}}}_{\textup{R-S}} \right)\,.$$
The second assertion follows as well.


\end{proof}
Assuming the validity of the Strong Rubin-Stark Conjecture for $\frak{E}$, we may therefore write
$$L_p^{\pm,\textup{alg}}= u\,\pi^{\varepsilon} L_p^\pm(E/F^+)$$
where $\pi\in \frak{O}$ is a uniformizer, $\varepsilon \in \ZZ$ and  $u\in \LL_\cyc$ is a unit. 
\subsection{Applications of the supersingular main conjecture}
The assumptions of the previous subsection are in effect until the end. We have the following consequence of Theorem~\ref{thm:improvedmainconj} to the Birch and Swinnerton-Dyer conjecture for $E/F^+$, generalizing parts of \cite[Theorem 11.4]{rubinmainconj} (which applies in the case $F^+=\QQ$).
\begin{thm}
\label{thm:bsd}
\begin{enumerate} 
\item If $L(E/F^+,1)\neq 0$ then $E(F^+)$ is finite. 
\item Assuming the validity of the Strong Rubin-Stark conjecture and that $L(E/F^+,1)=0$, the classical Selmer group $\textup{Sel}_p(E/F^+)$ is infinite.
\end{enumerate}
\end{thm}
\begin{rem}
Assuming the Strong Rubin-Stark conjecture and in case $L(E/F^+,1)\neq 0$, one may in fact express the cardinality of $\textup{III}(E/F)[p^\infty]$ in terms of $\varepsilon$, $u_1$ and the $L$-value. Since this lacks the desired level of precision, we do not include this statement as part of Theorem~\ref{thm:bsd}.
\end{rem}
\begin{proof}
The proof of this Theorem is essentially identical to the proof of \cite[Theorem 8.2]{pollackrubin}. Besides Theorem~\ref{thm:improvedmainconj}, the key points are as follows:
\begin{itemize}
\item[(a)] The perfect control theorem for the Selmer group $H^1_{\FF^*_{\textup{str}}}(F,T_p(E)\otimes\LL_{\textup{cyc}})^{\vee}$, which asserts that
$$H^1_{\FF^*_{\textup{str}}}(F,T_p(E)\otimes\LL_{\textup{cyc}})^{\vee}\otimes_{\LL_\cyc}{\ooo}\stackrel{\sim}{\lra}H^1_{\FF^*_{\textup{str}}}(F,T_p(E))^{\vee},$$
holds true thanks to \cite[Lemma 3.5.3]{mr02} (or \cite[Proposition 8.10.1]{nekovar06}).
\item[(b)] For every $n$ and $1\leq i \leq g$, the maps
$$E(F_{\wp_i})\otimes \Phi/\ooo \lra H^0(\Gamma_{\textup{cyc}},E^\pm(F_{n,\frak{p}_i})\otimes \Phi/\ooo)$$
are surjective. This assertion is proved as part of \cite[Lemma 8.3]{pollackrubin}.
\item[(c)] Using (a) and (b) above, one may deduce Kobayashi's control Theorem:
$$\textup{Sel}_p^\pm(E/F^{\textup{cyc}})^\vee\otimes_{\LL_\cyc} \ooo \stackrel{\sim}{\lra} \textup{Sel}_p(E/F)^\vee.$$
\item[(d)] The exact sequence (\ref{eqn:resolutionofY}), \cite[Lemma 6.5]{pollackrubin} and the proof of \cite[Theorem 11.16]{rubinmainconj} (applied with \cite[Theorem 3.1]{ng}) shows that $\textup{Sel}_p^\pm(E/F^{\textup{cyc}})^\vee$ has no finite-submodules. This together with (c) implies
$$|\left(\textup{Sel}_p^\pm(E/F^{\textup{cyc}})^\vee\right)\otimes_{\LL_\cyc} \ooo|=|\textup{Sel}_p(E/F)^\vee|\,.$$
\end{itemize}
The proof now follows from the interpolation property of the signed $p$-adic $L$-function $L_p^\pm(E/F^+)$ (considered at the identity character on $\Gamma_\cyc$) together with the isomorphism $\textup{Sel}_p(E/F^+)^\vee\otimes_{\ZZ_p}\ooo\stackrel{\sim}{\lra}\textup{Sel}_p(E/F)^\vee$
induced by the theory of complex multiplication.
\end{proof}
\begin{rem}
The analogous statements to Theorem~\ref{thm:bsd} may be proved in the ordinary case using the ordinary CM main conjectures, c.f. \cite{hsiehCMmainconj, kbbCMabvar}.
\end{rem}






{
\bibliographystyle{halpha}
\bibliography{references}

\def\Dbar{\leavevmode\lower.6ex\hbox to 0pt{\hskip-.23ex \accent"16\hss}D}
  \def\cftil#1{\ifmmode\setbox7\hbox{$\accent"5E#1$}\else
  \setbox7\hbox{\accent"5E#1}\penalty 10000\relax\fi\raise 1\ht7
  \hbox{\lower1.15ex\hbox to 1\wd7{\hss\accent"7E\hss}}\penalty 10000
  \hskip-1\wd7\penalty 10000\box7}
\begin{thebibliography}{B{\"u}y09b}

\bibitem[BL15]{kbblei1}
K\^az{\i}m {B\"uy\"ukboduk} and Antonio {Lei}.
\newblock {Coleman-adapted Rubin-Stark Kolyvagin systems and supersingular
  Iwasawa theory of CM abelian varieties.}
\newblock {\em {Proc. Lond. Math. Soc. (3)}}, 111(6):1338--1378, 2015.

\bibitem[BL16]{kbblei2}
K\^az{\i}m B{\"u}y\"ukboduk and Antonio Lei.
\newblock Integral {I}wasawa theory of motives for non-ordinary primes, 2016.
\newblock \emph{Math. Zeitschrift}, to appear.

\bibitem[B{\"u}y09a]{kbbstark}
K\^az{\i}m B{\"u}y\"ukboduk.
\newblock Kolyvagin systems of {S}tark units.
\newblock {\em J. Reine Angew. Math.}, 631:85--107, 2009.

\bibitem[B{\"u}y09b]{kbbiwasawa}
K\^az{\i}m B{\"u}y\"ukboduk.
\newblock Stark units and the main conjectures for totally real fields.
\newblock {\em Compositio Math.}, 145:1163--1195, 2009.

\bibitem[B{\"u}y10]{kbbesrankr}
K{\^a}zim B{\"u}y{\"u}kboduk.
\newblock On {E}uler systems of rank {$r$} and their {K}olyvagin systems.
\newblock {\em Indiana Univ. Math. J.}, 59(4):1277--1332, 2010.

\bibitem[B{\"u}y11]{kbb}
K\^az{\i}m B{\"u}y\"ukboduk.
\newblock {$\Lambda$}-adic {K}olyvagin systems.
\newblock {\em IMRN}, 2011(14):3141--3206, 2011.

\bibitem[B{\"u}y14]{kbbCMabvar}
K\^{a}z{\i}m B{\"u}y\"{u}kboduk.
\newblock {Main conjectures for CM fields and a Yager-type theorem for
  Rubin-Stark elements.}
\newblock {\em {Int. Math. Res. Not.}}, 2014(21):5832--5873, 2014.

\bibitem[B{\"u}y16]{kbbdeform}
K\^az{\i}m B{\"u}y\"ukboduk.
\newblock Deformations of {K}olyvagin systems, 2016.
\newblock \emph{Annales math\'ematiques du Qu\'ebec} (Special Issue on the
  Occasion of the 60th Birthday of Glenn Stevens), to appear.

\bibitem[CW77]{coateswiles77}
J.~Coates and A.~Wiles.
\newblock On the conjecture of {B}irch and {S}winnerton-{D}yer.
\newblock {\em Invent. Math.}, 39(3):223--251, 1977.

\bibitem[Hsi12]{hsiehCMmainconj}
Ming-Lun Hsieh.
\newblock Iwasawa main conjecture for {CM} fields, 2012.
\newblock 81pp., Preprint.

\bibitem[HT93]{ht93}
H.~Hida and J.~Tilouine.
\newblock Anti-cyclotomic {K}atz {$p$}-adic {$L$}-functions and congruence
  modules.
\newblock {\em Ann. Sci. \'Ecole Norm. Sup. (4)}, 26(2):189--259, 1993.

\bibitem[HT94]{ht94}
H.~Hida and J.~Tilouine.
\newblock On the anticyclotomic main conjecture for {CM} fields.
\newblock {\em Invent. Math.}, 117(1):89--147, 1994.

\bibitem[IP06]{iovitapollack}
Adrian Iovita and Robert Pollack.
\newblock Iwasawa theory of elliptic curves at supersingular primes over {$\Bbb
  Z_p$}-extensions of number fields.
\newblock {\em J. Reine Angew. Math.}, 598:71--103, 2006.

\bibitem[Kat78]{katz78}
Nicholas~M. Katz.
\newblock {$p$}-adic {$L$}-functions for {CM} fields.
\newblock {\em Invent. Math.}, 49(3):199--297, 1978.

\bibitem[Kat04]{ka1}
Kazuya Kato.
\newblock {$p$}-adic {H}odge theory and values of zeta functions of modular
  forms.
\newblock {\em Ast\'erisque}, (295):ix, 117--290, 2004.
\newblock Cohomologies $p$-adiques et applications arithm\'etiques. III.

\bibitem[Kob03]{kobayashi03}
Shin-ichi Kobayashi.
\newblock Iwasawa theory for elliptic curves at supersingular primes.
\newblock {\em Invent. Math.}, 152(1):1--36, 2003.

\bibitem[Mai08]{mainardi}
Fabio Mainardi.
\newblock On the main conjecture for {CM} fields.
\newblock {\em Amer. J. Math.}, 130(2):499--538, 2008.

\bibitem[MR04]{mr02}
Barry Mazur and Karl Rubin.
\newblock Kolyvagin systems.
\newblock {\em Mem. Amer. Math. Soc.}, 168(799):viii+96, 2004.

\bibitem[Nek06]{nekovar06}
Jan Nekov{\'a}{\v{r}}.
\newblock Selmer complexes.
\newblock {\em Ast\'erisque}, (310):viii+559, 2006.

\bibitem[NQ{\Dbar}84]{ng}
T.~Nguy{\cftil{e}}n-Quang-{\Dbar}{\cftil{o}}.
\newblock Formations de classes et modules d'{I}wasawa.
\newblock In {\em Number theory, {N}oordwijkerhout 1983 ({N}oordwijkerhout,
  1983)}, volume 1068 of {\em Lecture Notes in Math.}, pages 167--185.
  Springer, Berlin, 1984.

\bibitem[NSW08]{neukirch}
J{\"u}rgen Neukirch, Alexander Schmidt, and Kay Wingberg.
\newblock {\em Cohomology of number fields}, volume 323 of {\em Grundlehren der
  Mathematischen Wissenschaften [Fundamental Principles of Mathematical
  Sciences]}.
\newblock Springer-Verlag, Berlin, second edition, 2008.

\bibitem[Och05]{ochiaideform}
Tadashi Ochiai.
\newblock Euler system for {G}alois deformations.
\newblock {\em Ann. Inst. Fourier (Grenoble)}, 55(1):113--146, 2005.

\bibitem[Pol03]{pollackduke2003}
Robert Pollack.
\newblock On the {$p$}-adic {$L$}-function of a modular form at a supersingular
  prime.
\newblock {\em Duke Math. J.}, 118(3):523--558, 2003.

\bibitem[Pop04]{popescu}
Cristian~D. Popescu.
\newblock {Rubin's integral refinement of the Abelian Stark conjecture.}
\newblock {Providence, RI: American Mathematical Society (AMS)}, 2004.

\bibitem[PR84]{PR84memoirs}
Bernadette Perrin-Riou.
\newblock Arithm\'etique des courbes elliptiques et th\'eorie d'{I}wasawa.
\newblock {\em M\'em. Soc. Math. France (N.S.)}, (17):130, 1984.

\bibitem[PR93]{pr93grenoble}
Bernadette Perrin-Riou.
\newblock Fonctions {$L$} {$p$}-adiques d'une courbe elliptique et points
  rationnels.
\newblock {\em Ann. Inst. Fourier (Grenoble)}, 43(4):945--995, 1993.

\bibitem[PR98]{pr-es}
Bernadette Perrin-Riou.
\newblock Syst\`emes d'{E}uler {$p$}-adiques et th\'eorie d'{I}wasawa.
\newblock {\em Ann. Inst. Fourier (Grenoble)}, 48(5):1231--1307, 1998.

\bibitem[PR04]{pollackrubin}
Robert Pollack and Karl Rubin.
\newblock The main conjecture for {CM} elliptic curves at supersingular primes.
\newblock {\em Ann. of Math. (2)}, 159(1):447--464, 2004.

\bibitem[PS11]{parkshahabi}
Jeehoon Park and Shahab Shahabi.
\newblock Plus/minus {$p$}-adic {$L$}-functions for {H}ilbert modular forms.
\newblock {\em J. Algebra}, 342:197--211, 2011.

\bibitem[Roh84]{rohrlich84cyclo}
David~E. Rohrlich.
\newblock On {$L$}-functions of elliptic curves and cyclotomic towers.
\newblock {\em Invent. Math.}, 75(3):409--423, 1984.

\bibitem[Roh89]{rohrlich89}
David~E. Rohrlich.
\newblock Nonvanishing of {$L$}-functions for {${\rm GL}(2)$}.
\newblock {\em Invent. Math.}, 97(2):381--403, 1989.

\bibitem[Rub85]{rubincompositio85}
Karl Rubin.
\newblock Elliptic curves and {${\bf Z}_p$}-extensions.
\newblock {\em Compositio Math.}, 56(2):237--250, 1985.

\bibitem[Rub87]{rubin87}
Karl Rubin.
\newblock Local units, elliptic units, {H}eegner points and elliptic curves.
\newblock {\em Invent. Math.}, 88(2):405--422, 1987.

\bibitem[Rub91]{rubinmainconj}
Karl Rubin.
\newblock The ``main conjectures'' of {I}wasawa theory for imaginary quadratic
  fields.
\newblock {\em Invent. Math.}, 103(1):25--68, 1991.

\bibitem[Rub92]{ru92}
Karl Rubin.
\newblock Stark units and {K}olyvagin's ``{E}uler systems''.
\newblock {\em J. Reine Angew. Math.}, 425:141--154, 1992.

\bibitem[Rub96]{ru96}
Karl Rubin.
\newblock A {S}tark conjecture ``over {$\bold Z$}'' for abelian {$L$}-functions
  with multiple zeros.
\newblock {\em Ann. Inst. Fourier (Grenoble)}, 46(1):33--62, 1996.

\bibitem[Rub00]{r00}
Karl Rubin.
\newblock {\em Euler systems}, volume 147 of {\em Annals of Mathematics
  Studies}.
\newblock Princeton University Press, Princeton, NJ, 2000.
\newblock Hermann Weyl Lectures. The Institute for Advanced Study.

\bibitem[Wil78]{wiles78reciprocity}
A.~Wiles.
\newblock Higher explicit reciprocity laws.
\newblock {\em Ann. Math. (2)}, 107(2):235--254, 1978.

\bibitem[Wil95]{wiles}
Andrew Wiles.
\newblock Modular elliptic curves and {F}ermat's last theorem.
\newblock {\em Ann. of Math. (2)}, 141(3):443--551, 1995.

\end{thebibliography}
}
\end{document}